\documentclass[11pt]{amsart}
\usepackage{amsfonts,latexsym,amsthm,amssymb,amsmath,amscd,euscript,tikz, tikz-cd}
\usepackage[alphabetic, msc-links, bibtex-style, nobysame]{amsrefs}
\usepackage{stackengine}
\usepackage{framed}
\usepackage{xfrac}
\usepackage[makeroom]{cancel}
\usepackage{faktor}
\usepackage{braket}
\usepackage{pgf,tikz,pgfplots}
\pgfplotsset{compat=1.17}
\usepgfplotslibrary{fillbetween} 
\usepackage{mathrsfs}
\usetikzlibrary{arrows}
\definecolor{wrwrwr}{rgb}{0.3803921568627451,0.3803921568627451,0.3803921568627451}
\definecolor{rvwvcq}{rgb}{0.08235294117647059,0.396078431372549,0.7529411764705882}
\definecolor{mblue}{rgb}{0.2, 0.3, 0.8}
\definecolor{morange}{rgb}{1, 0.5, 0}
\definecolor{mgreen}{rgb}{0.1, 0.4, 0.2}
\definecolor{mred}{rgb}{0.5, 0, 0}
\definecolor{ForestGreen}{RGB}{34,139,34}
\usepackage{hyperref}
\hypersetup{colorlinks=true,citecolor=ForestGreen,linkcolor = blue,urlcolor =black,linkbordercolor={1 0 0}}
\usepackage{float}

\usepackage{lmodern}
\usepackage{supertabular}
\usepackage{verbatim}
\usepackage{amssymb}
\usepackage{enumerate}
\usepackage{stmaryrd}
\usepackage{bbm}
\usepackage{mathtools}
\usepackage[nopatch=footnote]{microtype}
\usepackage{setspace}
\usepackage{tikz,tikz-cd}
\usetikzlibrary{matrix,calc,positioning,arrows,decorations.pathreplacing,patterns,knots}

\newcommand{\la}{\langle}
\newcommand{\rg}{\rangle}

\newcommand{\mres}{\mathbin{\vrule height 1.6ex depth 0pt width
0.13ex\vrule height 0.13ex depth 0pt width 1.3ex}}

\newtheorem{theorem}{{Theorem}}[section]
\newtheorem*{theorem*}{Theorem}
\newtheorem{lemma}[theorem]{Lemma}
\newtheorem{proposition}[theorem]{Proposition}

\newtheorem{corollary}[theorem]{Corollary}
\newtheorem*{corollary*}{Corollary}

\theoremstyle{definition}

\usepackage{lmodern,url,enumerate,mathtools,microtype}
\usepackage[hmargin = 1in,vmargin=1in]{geometry}
\usepackage{graphicx}
\usepackage{subcaption}

\newcommand{\ve}{\varepsilon}

\newcommand{\mr}[1]{{\rm #1}}
\newcommand{\cA}{\mathcal{A}}\newcommand{\cB}{\mathcal{B}}

\newcommand{\cH}{\mathcal{H}}

\newcommand{\cN}{\mathcal{N}}

\newcommand{\cT}{\mathcal{T}}

\newcommand{\bC}{\mathbb{C}}

\newcommand{\bR}{\mathbb{R}}
\newcommand{\bS}{\mathbb{S}}

\newcommand{\bZ}{\mathbb{Z}}

\newcommand{\fq}{\mathfrak{q}}

\newcommand{\nc}{\newcommand}

\nc{\on}{\operatorname}
\nc{\p}{\partial}
\nc{\ol}{\overline}
\nc{\ul}{\underline}
\nc{\pa}{\partial}

\nc{\pb}{\partial_b}
\nc{\pc}{\partial_c}
\nc{\pd}{\partial_d}
\nc{\pe}{\partial_e}
\nc{\pf}{\partial_f}
\nc{\pg}{\partial_g}
\nc{\ph}{\partial_h}
\nc{\pari}{\partial_i}
\nc{\pj}{\partial_j}
\nc{\pk}{\partial_k}
\nc{\pl}{\partial_l}
\nc{\pell}{\partial_\ell}
\nc{\parm}{\partial_m}
\nc{\pn}{\partial_n}
\nc{\po}{\partial_o}
\nc{\pp}{\partial_p}
\nc{\pq}{\partial_q}
\nc{\pr}{\partial_r}
\nc{\ps}{\partial_s}
\nc{\pt}{\partial_t}
\nc{\pu}{\partial_u}
\nc{\pv}{\partial_v}
\nc{\pw}{\partial_w}
\nc{\px}{\partial_x}
\nc{\py}{\partial_x}
\nc{\pz}{\partial_z}

\nc{\Spec}{\on{Spec}}

\nc{\sn}{\mr{sn}}
\nc{\cn}{\mr{cn}}
\nc{\dn}{\mr{dn}}

\allowdisplaybreaks

\newcommand{\@dotsep}{4.5}

\def\@tocline#1#2#3#4#5#6#7{%
  \relax
  \ifnum #1>\c@tocdepth
  \else
    \par \addpenalty\@secpenalty \addvspace{#2}%
    \begingroup
      \hyphenpenalty\@M
      \@ifempty{#4}{%
        \@tempdima\csname r@tocindent\number#1\endcsname\relax
      }{%
        \@tempdima#4\relax
      }%
      \parindent\z@
      \leftskip#3\relax
      \advance\leftskip\@tempdima\relax
      \rightskip\@pnumwidth plus4em
      \parfillskip-\@pnumwidth
      #5\leavevmode\hskip-\@tempdima #6\nobreak
      \leaders\hbox{$\m@th
        \mkern \@dotsep mu.\mkern \@dotsep mu$}\hfill
      \nobreak\hbox to\@pnumwidth{\@tocpagenum{#7}}\par
    \endgroup
  \fi
}

\renewcommand{\l@section}{\@tocline{1}{2pt}{0pt}{}{\bfseries}}
\renewcommand{\l@subsection}{\@tocline{2}{0pt}{1.5em}{3.2em}{}}
\renewcommand{\l@subsubsection}{\@tocline{3}{0pt}{4.7em}{3.8em}{}}

\makeatother

\title[Stable $s$-minimal cones in $\mathbb{R}^3$ are flat for $s \sim 0$]{Stable $s$-minimal cones in $\mathbb{R}^3$ are flat for $s$ close to zero}
\date{\today}

\author[Raphael Tsiamis]{Raphael Tsiamis}
\address{
{\href{mailto:r.tsiamis@columbia.edu}{r.tsiamis@columbia.edu}}\hfill Department of Mathematics, Columbia University}

\begin{document}

\begin{abstract}
We prove that stable nonlocal minimal cones in three dimensions are flat, for $s$ sufficiently close to zero. Our approach develops new structural properties of stable $s$-minimal surfaces in every dimension as $s \downarrow 0$, notably a compactness theory and pinching theorems near a half-space. 
\end{abstract}

\vspace*{-0.2in}
\maketitle

%\vspace{-0.2in}

\section{Introduction}

Nonlocal minimal surfaces arise as minimizers of the fractional perimeter, a geometric energy that accounts for long-range interactions across an interface. 
They provide a natural nonlocal analogue of classical minimal surfaces and appear in models involving phase transitions, anomalous diffusion, and interacting particle systems. 
The theory of nonlocal minimal surfaces was initiated by Caffarelli, Roquejoffre, and Savin in~\cite{caffarelli-roquejoffre-savin}, who established the foundational existence and regularity theory, including an improvement of flatness theorem, a monotonicity formula, blow-up analysis, and a Federer-type dimension reduction argument.
As in the classical theory, the fundamental problem for the nonlocal perimeter is to understand the regularity of minimizers by studying their blow-up models at singularities.
These objects are known as $s$-minimal cones, where $s \in (0,1)$ encodes the strength of the long-range interaction.
When the parameter $s$ approaches zero, the nonlocal interaction becomes stronger, and the regularity of $s$-minimal surfaces is largely undeveloped.
We resolve this question in three dimensions.

\begin{theorem}\label{thm:main}
There exists an $s_* > 0$ such that for all $s \in (0,s_*)$, every locally $s$-perimeter minimizing cone in $\mathbb{R}^3$ is a half-space.
The same property holds for every stable $s$-minimal cone in $\mathbb{R}^3$ with $C^2$ boundary away from the origin.
\end{theorem}

For every $s \in (0,1)$, Savin and Valdinoci~\cites{savin-valdinoci , savin-valdinoci-2 } proved that minimizing cones in $\bR^2$ are flat, and Caffarelli-Valdinoci~\cite{caffarelli-valdinoci-advances} established the flatness of minimizing cones in $\bR^N$, for every $2 \leq N \leq 7$ and $s$ sufficiently close to $1$, using a compactness argument.
By the work of Figalli-Valdinoci, the flatness of $s$-minimizing cones in $\bR^N$ implies the nonlocal Bernstein theorem in $\bR^{N+1}$~\cite{figalli-valdinoci}. 
For stable cones, such compactness arguments are generally unavailable, and the classification becomes significantly more challenging in all regimes of $s$; indeed, the cross $\{ x_1 x_2 >0\} \subset \bR^2$ is stable for $s$ close to $1$. 
Working with a general class of nonlocal energies that includes the $s$-perimeter for all $s \in (0,1)$, Cinti-Serra-Valdinoci~\cite{cinti-serra-valdinoci} obtained BV estimates and proved the flatness of cones in $\bR^2$ satisfying the stronger condition of stability under rearrangements.
More recently, Caselli~\cite{caselli} established that stable cones in $\bR^2$ are flat for $s$ close to zero.

In the near-local regime $s \sim 1$, significant breakthroughs due to Cabr\'e-Cinti-Serra~\cites{cabre-cinti-serra , s-allen-cahn} proved the flatness of stable $s$-minimal cones in $\bR^3$ and used this result to obtain fractional versions of the de Giorgi conjecture and the stable Bernstein theorem.
Chan-Dipierro-Serra-Valdinoci~\cite{nonlocal-approximation} introduced further ideas that extended this flatness result to $\bR^4$ and $s \sim 1$.
Finally, Caselli-Florit--Simon-Serra~\cite{caselli-fs-serra} further generalized the flatness property to solutions with finite Morse index and proved a striking nonlocal version of Yau's conjecture for $s$-minimal surfaces.

Recent work of Dipierro-Fern\'andez-Real-Valdinoci~\cite{generically-unique} further relates the critical regularity dimension for $s$-perimeter minimizers to their generic regularity.
Together with the existing theory, notably~\cites{figalli-valdinoci , cabre-cinti-serra , caselli-fs-serra }, Theorem~\ref{thm:main} yields the following additional consequences.

\begin{corollary}\label{cor:other-conjectures}
    There exists a $s_*>0$ such that, for every $s \in (0,s_*)$, the following properties hold.
    \begin{enumerate}[(i)]
        \item Every fractional Allen-Cahn solution in $\bR^3$ with finite Morse index is one-dimensional.
        \item Every entire $C^2$ nonlocal minimal surface in $\bR^3$ with finite Morse index is a plane. 
        \item Every closed $3$-manifold contains infinitely many smooth $s$-minimal surfaces.
        \item Every entire globally minimizing $s$-minimal graph in $\bR^4$ is a hyperplane.
        \item For $N \geq 4$, let $G \subseteq \bR^N \setminus B_1$ be an exterior datum such that $\partial (G \cup B_1)$ has locally finite $(N-1)$-Minkowski content.
Then, there are $L^1_{\textup{loc}}$-small perturbations $G'$ of $G$ for which there exists a unique $s$-perimeter minimizer $E'$ in $B_1$ with $E' = G'$ in $\bR^N \setminus B_1$, which is smooth when $N = 4$ and satisfies $\dim_{\cH} \textup{sing} (\partial E') \cap B_1 \leq N-5$ for $N \geq 5$.
    \end{enumerate}
\end{corollary}
On the other hand, the D\'avila-Del Pino-Wei $s$-minimal cones, constructed in~\cite{davila-delpino-wei}, are stable in dimension $N=7$ for small $s$.
We expect that the techniques developed in this paper, notably the compactness result of Theorem~\ref{thm:compactness-thm}, will be useful in determining the regularity of stable and minimizing nonlocal cones in the remaining dimensions $N \in \{ 4,5,6\}$ when $s$ is close to zero.

\subsection{Strategy of the proof}

The proof of Theorem~\ref{thm:main} is motivated by the Mellin multiplier approach in the beautiful work of Fern\'andez-Real and Ros-Oton on the fractional one-phase problem~\cite{stable-thin-one-phase }, as well as the work of Caffarelli, Jerison, and Kenig~\cite{caffarelli-jerison-kenig} on the one-phase problem.
It is also influenced by the results of Cabr\'e-Cinti-Serra~\cite{cabre-cinti-serra} in the regime $s \sim 1$.
The first step is to recast the stability inequality for an $s$-minimal cone entirely in terms of the spherical link, namely the trace of its boundary on $\bS^{N-1}$.
This reduction is performed in Section~\ref{section:stability-inequality}.

To analyze the stability inequality for $s$ close to $0$, we then extract and classify the limits of $s$-minimal cones as $s \downarrow 0$.
Since the $s$-perimeter approaches the density at infinity of the enclosed domain for $s \downarrow 0$, by~\cites{asymptotics-s-perimeter , valdinoci-bucur-lombardini }, these limits are expected to be half-volume Caccioppoli sets in $\bS^{N-1}$.
To extract such limits, Proposition~\ref{prop:bound-s} establishes uniform perimeter estimates for stable $s$-minimal surfaces that not degenerate as $s \downarrow 0$, improving the fundamental theory developed in~\cite{cinti-serra-valdinoci}.
These uniform estimates are of independent interest and have important applications to the behavior of nonlocal minimal surfaces as $s \sim 0$, which differs substantially from the regime $s \sim 1$.
In particular, Theorem~\ref{thm:compactness-thm} describes the limit of links of stable $s$-minimal cones and shows that they satisfy a limiting stability inequality in addition to several current- and measure-theoretic properties.
Our compactness theory also requires integral estimates for certain kernels related to the fractional $s$-perimeter; these are established in Appendix~\ref{app:kernel-estimates}.
At $s=0$, the limiting stability inequality reduces to explicit integral computations, carried out in Lemma~\ref{lemma:domain-expansion} and Proposition~\ref{prop:stable-links-s=0}.

The approach developed in this paper applies in every dimension and may play a role in the study of stable and minimizing nonlocal cones for $s \sim 0$ in the remaining dimensions $N \in \{4,5,6\}$.
Crucially, we obtain two pinching results in every dimension, Theorem~\ref{thm:uniform-pniching} and Corollary~\ref{cor:near-isoperimetric}, which show that stable $s$-minimal cones are flat if their link has almost constant normal vector or is almost isoperimetric.
Combined with our compactness theory, this leads to a dimension reduction scheme, formulated in Proposition~\ref{prop:dimension-descent}, which reduces the classification of stable cones to proving that the only stable $s=0$ limit link is an equator.
Proposition~\ref{prop:stable-links-s=0} establishes this classification in $\bR^3$.
%In $\bR^3$, we obtain a very precise description of the curves in $\bS^2$ arising from this limiting procedure, as given in Section~\ref{section:limits-3-D}, notably Proposition~\ref{prop:N=3-limiting-configuration}.
%In particular, we deduce that the only stable link is an equator, proving Theorem~\ref{thm:main}.

\subsection{Connections to other problems}
The stability reduction of Section~\ref{section:stability-inequality} is connected to a much more general principle, related to Hardy's inequality.
This technique finds applications to a large number of local and nonlocal variational problems that we develop in upcoming work~\cite{critical-hardy}.
A key novelty of our result is the compactness theory for stable nonlocal minimal surfaces as $s \downarrow 0$, established through uniform perimeter estimates for stable $s$-minimal surfaces in Section~\ref{sec:compactness}.
For variational problems depending continuously on a parameter and satisfying such compactness properties, the regularity of solutions and properties of singularity models can be quantitatively studied.
This technique has also found significant applications to the capillary, one-phase Bernoulli, and Alt-Phillips problems~\cites{chodosh-edelen-li , FTW-1 , desilva-savin }.

We expect that the techniques developed in this paper could be useful in establishing the flatness of stable $s$-minimal cones in $\bR^4$ for $s \sim 0$.
Thanks to Proposition~\ref{prop:dimension-descent}, this result is equivalent to proving that every stable $s=0$ limit link in $\bS^3$ is an equator.
We highlight a number of remarkable techniques introduced to rule out non-flat minimizing cones in low dimensions, including Almgren's topological classification of minimizing links in $\bS^3$~\cite{almgren} (see also~\cite{chodosh-edelen-li}) and the Jerison-Savin technique for the one-phase problem~\cite{jerison-savin}, as well as further advances in the capillary and Alt-Phillips problems~\cites{singular-capillary , fernandez-real }.
Unlike the four-dimensional classification of~\cite{nonlocal-approximation} in the regime $s \sim 1$, the limiting Caccioppoli sets as $s \downarrow 0$ exhibit strong far-field effects and may behave very differently from minimal surfaces, so they require further new ideas.
It would be interesting to analyze the properties of such domains, including potential analogies with the remarkable results of Ros~\cite{two-piece-property} and Brendle~\cites{brendle-acta , brendle }, and extend our regularity results to four dimensions.

\smallskip \noindent 
\textbf{Acknowledgments.}
I am thankful to Riccardo Villa for many fruitful discussions on nonlocal minimal surfaces.
I also thank Ovidiu Savin, Francisco Mart\'in, and Mariel S\'aez for helpful conversations and comments on a preliminary version of this manuscript.
I am grateful to my advisor, Simon Brendle, for his invaluable guidance and continued support.
This work was supported in part by the A.G. Leventis Foundation Scholarship and the Onassis Foundation Scholarship.

\section{Preliminaries}\label{sec:preliminaries}

For a measurable set $E \subseteq \mathbb{R}^N$ and a bounded open set $U \subseteq \mathbb{R}^N$, we define
\[
\textup{Per}_s(E;U)=\iint_{(E\cap U)\times E^c}\frac{dx\,dy}{|x-y|^{N+s}}
+\iint_{(E\setminus U)\times(E^c\cap U)}\frac{dx\,dy}{|x-y|^{N+s}}.
\]
At a smooth point of $\partial E$, we define the $s$-nonlocal mean curvature as
\[
H_s[E](x)= \frac{1}{2} \, \textup{PV}\int_{\mathbb{R}^N} \frac{\mathbf{1}_{E^c}(y)-\mathbf{1}_E(y)}{|x-y|^{N+s}}\,dy.
\]
We say that $E$ is $s$-\textbf{minimal} when $H_s[E] = 0$ everywhere on the regular part of $E$.
Sets satisfying this condition are also referred to as $s$-\textbf{stationary}, while the notion of minimality is sometimes reserved in the literature for minimizers of the $s$-perimeter.

If $M = \partial E$ is $C^2$ in $\bR^N \setminus \{ 0 \}$, we define the \textbf{nonlocal Jacobi operator}
\[
\mathcal{J}_{s,M} h(X) = \textup{PV} \int_M \frac{h(Y) - h(X)}{|X-Y|^{N+s}} \, d \sigma_Y + h(x) \int_M \frac{1 - \langle \nu_X, \nu_Y \rangle}{|X-Y|^{N+s}} \, d \sigma_Y.
\]
With these conventions, the Jacobi operator has the significance that the linearization of $H_s$ under the normal graph $X \mapsto X + \varepsilon h(X) \nu_X$ is given by $- \mathcal{J}_{s,M}h$.

If $M = \partial E$ is $s$-minimal and smooth on the support of the variation, its second variation form is
\[
Q_{s,M}[\phi] := \frac{1}{2} \iint_{M \times M} \frac{( \phi(X) - \phi(Y))^2}{|X-Y|^{N+s}} \, d \sigma_X \, d \sigma_Y - \int_{M} \phi(X)^2\int_M\frac{1-\langle \nu_X , \nu_Y \rangle}{|X-Y|^{N+s}}\, d\sigma_Y \, d\sigma_X.
\]
We say that $M$ is \textbf{stable} when $Q_{s,M}[\phi] \geq 0$ for every $\phi \in C_c^{\infty} (M \setminus \{ 0 \})$.

Throughout, we will consider conical Caccioppoli sets $E = C(\Omega) \subset \mathbb{R}^N$, where $\Omega \subset \mathbb{S}^{N-1}$ and $N \geq 3$, and write $\Gamma = \partial \Omega$ with $M = \partial E= C(\Sigma)$.
Thus, every $X \in M \setminus \{ 0 \}$ can be written uniquely as $X = e^t x$ with $t \in \mathbb{R}$ and $x \in \Sigma$.
The normal is homogeneous of degree zero, so $\nu_{e^t x} = \nu_x$, while the surface measure is $d \sigma_X = e^{(N-1) t} \, dt \, d \mu_x$.
In coordinates $(X,Y) = (e^t x, e^u y)$, we therefore obtain
\[
\phi(e^t x) = \exp ( - \tfrac{(N-2-s) t}{2} ) \psi(t,x), \qquad |e^t x - e^u y|^2 = 2 e^{t+u} ( \cosh(t-u) - \langle x,y \rangle ).
\]
We introduce the spherical kernels
\begin{equation}\label{eqn:spherical-kernels}
    \begin{split}
A_s(z) &:= \int_0^{\infty} \frac{\rho^{\frac{N-2+s}{2}}}{(1 + \rho^2 - 2 \rho z)^{\frac{N+s}{2}}} \, d \rho, \qquad B_s(z) := \int_0^{\infty} \frac{\rho^{N-2}}{(1+\rho^2 - 2 \rho z)^{\frac{N+s}{2}}} \, d \rho, \\
W_s(z) &:= B_s(z) - A_s(z) = \int_0^1 \frac{( \rho^{\frac{N-2}{2}} - \rho^{\frac{s}{2}})^2}{(1 + \rho^2 - 2 \rho z)^{\frac{N+s}{2}}} \, d \rho \geq 0.
    \end{split}
\end{equation}
The expression for $W_s(z)$ is obtained by splitting the integral expressions $A_s, B_s$ at $\rho=1$ and making the change of variables $\rho \mapsto \rho^{-1}$ on $(1,\infty)$.
\begin{lemma}\label{lemma:ws-useful-bound}
    There exists a dimensional constant $C_N$ such that, for every $0 < s<1$, we have
    \[
    W_s (\la x,y \rg) \leq C_N \, (1 + |\log |x-y||) \, |x-y|^{-(N-3+s)}.
    \]
\end{lemma}
\begin{proof}
    More precisely, we can prove that
    \begin{equation}\label{eqn:ws-bounds-N}
        \begin{aligned}
    W_s(\la x,y \rg) &\leq C_N |x-y|^{-(N-3+s)} \quad &\text{for } \; N &\geq 4, \\
    W_s(\la x,y \rg) &\leq C \, [ 1 + s^{-1}( |x-y|^{-s}-1) ] &\quad \text{for } \; N&=3.
        \end{aligned}
    \end{equation}
    Let $r := |x-y|$, so $r^2 = 2(1-\la x,y \rg)$ and $1 + \rho^2 - 2 \rho \la x,y \rg = (1-\rho)^2 + \rho r^2$.
    We split the defining integral for $W_s$ into $\rho\in(0, \frac{1}{2})$ and $\rho \in ( \frac{1}{2}, 1)$.
    The first contribution is bounded uniformly, since $(1-\rho)^2+\rho r^2\geq\frac{1}{4}$ there.
    On $\rho \in [\frac{1}{2},1]$, the functions $\rho \mapsto \rho^{\gamma}$ have uniformly bounded derivatives for $\gamma \in [ 0 , \frac{N}{2}]$, hence $|\rho^{\frac{N-2}{2}} - \rho^{\frac{s}{2}}| \leq C_N(1-\rho)$.
    Since $\rho \geq \frac{1}{2}$, we can set $1- \rho = ru$ to bound
    \[
    W_s(\la x,y \rg) \leq C_N + C_N r^{-(N-3+s)} \int_0^{\frac{1}{2r}} \frac{u^2}{(1+u^2)^{\frac{N+s}{2}}} \, du.
    \]
    If $N \geq 4$, the last integral is bounded uniformly in $s \in (0,1)$, hence proving~\eqref{eqn:ws-bounds-N}.
    If $N=3$, then
    \[
    \int_0^{\frac{1}{2r}} \frac{u^2}{(1+u^2)^{\frac{3+s}{2}}} \, du \leq C + C \int_1^{\frac{1}{2r}} u^{-1-s} \, du \leq C \bigl[ 1 + s^{-1}( r^{-s}-1) \bigr]
    \]
    for $r \leq \frac{1}{2}$.
    The second bound of~\eqref{eqn:ws-bounds-N} follows from absorbing $r \in [ \frac{1}{2}, 1]$ into the constant.
    Now,
    \[
    s^{-1} (r^{-s}-1) = |\log r| \int_0^1 r^{- \theta s} \, d \theta \leq |\log r| \, r^{-s}
    \]
    shows that $W_s( \la x,y \rg) \leq C (1 + |\log r|) r^{-s}$.
    Combining the two cases and using the boundedendess of $W_s$ away from the diagonal completes the proof.
\end{proof}

\begin{lemma}\label{eqn:ws-properties}
For every $0 < s < 1$, the integral kernel $W_s ( \langle x , y \rangle)$ is locally integrable on $\mathbb{S}^{N-1} \times \mathbb{S}^{N-1}$ and on $\Sigma \times \Sigma$, and is positive-definite on test functions in the sense that 
\begin{equation}\label{eqn:on-test-functions}
    \iint_{\Sigma \times \Sigma} W_s ( \langle x,y \rangle) f(x) f(y) \, d \mu_x \, d \mu_y \geq 0, \qquad \text{for every } \; f \in C^{\infty}(\Sigma).
\end{equation}
\end{lemma}
\begin{proof}
As $z \uparrow 1$, we can write $1 + \rho^2 - 2 \rho z = (1 - \rho)^2 + 2 \rho(1-z)$ near $\rho=1$ to bound
\[
A_s(z) + B_s(z) \leq C(1-z)^{- \frac{N+s-1}{2}}, \qquad W_s(z) \leq C (1-z)^{- \frac{N+s-3}{2}},
\]
whereas the numerator in the expression for $W_s$ vanishes quadratically at $\rho=1$.
Since $1 - \langle x , y \rangle \asymp d_{\mathbb{S}^{N-1}}(x,y)^2$, we find $W_s( \langle x,y \rangle) \leq C\,  d_{\mathbb{S}^{N-1}}(x,y)^{-(N+s-3)}$. 
Thus, applying Corollary~\ref{cor:local-integrability} with $m=N-2$ shows that $W_s$ is locally integrable on $\Sigma \times \Sigma$.
Similarly, $A_s( \langle x,y \rangle) (g(x) - g(y))^2$ and $B_s ( \langle x,y \rangle) (1 - \langle x,y \rangle)$ are locally integrable for every smooth $g$, since both squared differences vanish quadratically on the diagonal and $1 - \langle \nu_x, \nu_y \rangle = \frac{1}{2} |\nu_x - \nu_y|^2 = O ( d_{\Sigma}(x,y)^2)$.

Next, to prove the positive-definiteness~\eqref{eqn:on-test-functions}, we consider $\delta \in (0,1)$ and let $W_{s,\delta}(z)$ be the truncation
\[
W_{s,\delta}(z) = \int_0^{1-\delta} \frac{(\rho^{\frac{N-2}{2}} - \rho^{\frac{s}{2}})^2}{(1 + \rho^2 - 2 \rho z)^{\frac{N+s}{2}}} \, d \rho.
\]
For every $0 \leq \rho \leq 1-\delta$, we have the expansion
\[
\frac{1}{(1 + \rho^2 - 2 \rho z)^{ \frac{N+s}{2}}} = \frac{1}{(1 + \rho^2)^{\frac{N+s}{2}}} \sum_{k=0}^{\infty} \frac{1}{k!} \Bigl( \frac{N+s}{2} \Bigr)_k \Bigl( \frac{2 \rho}{1 + \rho^2} \Bigr)^k z^k, \qquad z = \langle x, y \rangle
\]
which converges absolutely and uniformly for $z \in [-1,1]$.
This expression has non-negative coefficients, with $(a)_k = a(a+1) \cdots (a+k-1) \geq 0$ the Pocchammer symbol.
Moreover, each kernel $\langle x,y \rangle^k$ is positive semidefinite because for any $\{ x_i \}_{i=1}^m \in \mathbb{R}^N$ and $\{ c_i \}_{i=1}^m \in \mathbb{R}$, we have
\[
\textstyle {\sum_{i,j=1}^m c_i c_j \langle x_i, x_j \rangle^k = \bigl\| \sum_{i=1}^m c_i x_i^{\otimes k} \bigr\|^2 \geq 0.}
\]
Thus, $(1 + \rho^2 - 2 \rho z)^{- \frac{N+s}{2}}$ is positive semidefinite for each $\rho \in (0,1)$, and so is $W_{s,\delta}( \langle x,y \rangle)$, as its superposition with the non-negative measure $(\rho^{\frac{N-2}{2}} - \rho^{\frac{s}{2}})^2 \, d \rho$.
Hence,
\[
\iint_{\Sigma \times \Sigma} W_{s,\delta} (\langle x,y \rangle) f(x) f(y) \, d \mu_x \, d \mu_y \geq 0.
\]
Since $0 \leq W_{s,\delta}(z) \leq W_s(z)$ and $W_s( \langle x,y \rangle)$ is integrable on $\Sigma \times \Sigma$, the dominated convergence theorem yields the assertion~\eqref{eqn:on-test-functions} upon sending $\delta \downarrow 0$. 
This completes the proof.
\end{proof}

Substituting $\rho = e^{\tau}$ and writing $P_s(\tau,z) = [2 ( \cosh \tau - z)]^{- \frac{N+s}{2}}$, we also obtain the expressions
\begin{equation}\label{eqn:A-B-expansions}
    A_s(z) = \int_{\mathbb{R}} P_s(\tau,z) \, d \tau, \qquad B_s(z) = \int_{\mathbb{R}} e^{\frac{N-2-s}{2} \tau} P_s(\tau,z) \, d \tau.
\end{equation}
\begin{lemma}\label{lemma:kernel-k-s-tau,z-bounds}
    There exists a constant $C$, depending only on the submanifold $\Sigma \subset \bS^{N-1}$, such that
\begin{equation}\label{eqn:Ks-tau-x,y-bounds}
    \begin{aligned}
\int_{\Sigma} P_s ( \tau, \langle x, y \rangle) \, d \mu_y &\leq C |\tau|^{-(2+s)} \qquad & \text{for } \; 0 &\leq |\tau| \leq 1, \\
\int_{\Sigma} P_s(\tau, \langle x, y \rangle) \, d \mu_y &\leq C \exp ( - \tfrac{N+s}{2} |\tau|) \qquad & \text{for } \; & |\tau| \geq 1.
    \end{aligned}
\end{equation}
\end{lemma}
\begin{proof}
To obtain the first bound, we observe that for $0 < |\tau| \leq 1$, we have
\[
\cosh \tau - \langle x,y \rangle = (\cosh \tau - 1) + ( 1 - \langle x,y \rangle) = c ( \tau^2 + d_{\Sigma}(x,y)^2)
\]
near the diagonal $\Sigma \times \Sigma$.
Away from the diagonal, the kernel is uniformly bounded.
Therefore,
\[
P_s ( \tau, \langle x,y \rangle) \leq C( \tau^2 + d_{\Sigma}(x,y)^2 )^{- \frac{N+s}{2}} \quad \text{for } \; 0 \leq |\tau| \leq 1.
\]
Using geodesic polar coordinates on $\Sigma$, we therefore obtain 
\[
\int_{\Sigma} P_s (\tau , \langle x, y \rangle) \, d \mu_y \leq C + \int_0^1 r^{N-3} (\tau^2 + r^2)^{- \frac{N+s}{2}} \, d r \leq C +  |\tau|^{-(2+s)} \int_0^{|\tau|^{-1}} \frac{\rho^{N-3}}{(1+\rho^2)^{\frac{N+s}{2}}} \, d \rho.
\]
For $0 < s < 1$, the last integral is bounded independently of $\tau$, proving the first estimate of~\eqref{eqn:Ks-tau-x,y-bounds}.

For the second bound, we consider $|\tau| \geq 1$ and $x,y \in \bS^{N-1}$, so $\la x,y \rg \leq 1$ and
\begin{align*}
& \cosh \tau - \langle x,y \rangle \geq \cosh |\tau| - 1 = \tfrac{1}{2} ( e^{|\tau|} + e^{- |\tau|} - 2) = \tfrac{1}{2} e^{|\tau|} (1 - e^{- |\tau|})^2 \geq c e^{|\tau|} \\
& \implies P_s(\tau, \langle x, y \rangle) = 2^{- \frac{N+s}{2}} ( \cosh \tau - \langle x, y \rangle)^{- \frac{N+s}{2}} \leq C e^{- \frac{N+s}{2} |\tau|} \qquad \text{for } \; |\tau| \geq 1,
\end{align*}
uniformly in $x,y \in \Sigma$.
This completes the proof of our assertion.
\end{proof}

\begin{lemma}\label{lemma:ps-identity}
On an $s$-minimal cone $E = C(\Omega) \subset \mathbb{R}^N$ with $\Sigma = \partial \Omega \subset \bS^{N-1}$, we form the expression
\begin{equation}\label{eqn:ps(x)}
    p_s(x) := \int_{\Sigma} B_s( \langle x, y \rangle ) (1 - \langle \nu_x , \nu_y \rangle) \, d \mu_y,
\end{equation}
which defines an absolutely convergent integral.
Then, it holds that
    \begin{equation}\label{eqn:ps(x)-nux}
        \textup{PV} \int_{\Sigma} B_s ( \langle x,y \rangle) (\nu_x - \nu_y) \, d \mu_y = p_s(x) \nu_x
    \end{equation}
    for every $x \in \Sigma$.
    If $\Sigma$ is $C^2$, then the left-hand side exists pointwise.
\end{lemma}
\begin{proof}
The integral defining $p_s(x)$ is absolutely convergent because
\[
B_s ( \langle x, y \rangle) \leq C \, d_{\Sigma} (x,y) ^{- (N+s-1)}, \qquad 1 - \langle \nu_x , \nu_y \rangle = O ( d_{\Sigma}(x,y)^2),
\]
so the integrability of~\eqref{eqn:ps(x)} follows from Corollary~\ref{cor:local-integrability}.
To prove the identity for $p_s(x)$, we fix $a \in \mathbb{R}^N$ and set $h_a(X) = \langle a, \nu_X \rangle$.
The translated hypersurfaces $M_{\varepsilon} = M + \varepsilon a$ are again $s$-minimal, and $H_s[M_{\varepsilon}](X + \varepsilon a) = H_s[M](X)= 0$.
Moreover, up to tangential reparametrization, the family of diffeomorphisms $\varepsilon \mapsto M_{\varepsilon}$ has variation field the constant vector $a$, whose normal component along $M$ is $h_a \nu$.
Since $H_s[M_{\varepsilon}] = 0$, differentiating at $\varepsilon = 0$ gives $\mathcal{J}_{s,M} h_a = 0$, so $h_a$ is a Jacobi field.
Thus,
\[
\textup{PV} \int_M \frac{\langle a, \nu_Y - \nu_X \rangle}{|X-Y|^{N+s}} \, d \sigma_Y + \langle a, \nu_X \rangle \int_M \frac{1 - \langle \nu_X, \nu_Y \rangle}{|X-Y|^{N+s}} \, d \sigma_Y =0
\]
for every $a \in \mathbb{R}^N$.
Because this property holds for every $a$, we equivalently find
\[
\textup{PV} \int_M \frac{\nu_X - \nu_Y}{|X-Y|^{N+s}} \, d \sigma_Y = \nu_X \int_M \frac{1 - \langle \nu_X, \nu_Y \rangle}{|X-Y|^{N+s}} \, d \sigma_Y.
\]
We now take $(X,Y) = (x,\rho y)$ with $x,y \in \Sigma$ and $\rho>0$.
Since $\nu_{e^u y} = \nu_y$ and $d \sigma_Y = \rho^{N-2} \, d \rho \, d \mu_y$, radial integration gives $\int_0^{\infty} \rho^{N-2} |x- \rho y|^{-(N+s)} \, d \rho = B_s ( \langle x, y \rangle)$.
Thus, the right-hand side equals $p_s(x) \nu_x$, while the left-hand side becomes the asserted principal value integral of~\eqref{eqn:ps(x)-nux}.

When the link $\Sigma$ is $C^2$, then the above integral exists pointwise.
Indeed, since $C(\Omega)$ is $s$-stationary with $C^2$ link, the bootstrap regularity arguments of~\cites{barrios-figalli-valdinoci , figalli-valdinoci } imply that $\Sigma$ is smooth.
We fix a point $x$ and take a geodesic coordinate $v$ centered at $x$.
Then, the first-order term in $\nu_x - \nu_Y$ is odd in $v$, so it cancels by the symmetry of the principal value integral.
Concretely, we can expand
\[
\nu_{y(v)} = \nu_x + D \nu_x[v] + O(|v|^2), \qquad 1 - \langle x, y(v) \rangle = \tfrac{1}{2} |v|^2 + O (|v|^3), \qquad d \mu_{y(v)} = (1 + O (|v|^2)) \, dv.
\]
The leading part of $B_s ( \langle x, y(v) \rangle) ( \nu_x - \nu_{y(v)})$ is an even radial kernel multiplied by the odd function $- D \nu_x[v]$, so it vanishes under the symmetric principal value integration.
After removing this odd term, the remainder is bounded by $C |v|^{- (N+s-3)}$, with $v \in \bR^{N-2}$, so the corresponding radial integral is bounded by $C \int_0^{\ve} r^{-s} \, dr < \infty$.
The truncated integral in the Euclidean region $|x-\rho y| > \ve$, in a neighborhood of $(\rho,y) = (1,x)$, is comparable with the integral over the region $d_{\Sigma}(x,y) > \ve$, with difference tending to zero as $\ve \to 0$ by the above expansion.
This proves our result.
\end{proof}
In what follows, we also record the weak formulation of the identity~\eqref{eqn:ps(x)-nux}, namely
\begin{equation}\label{eqn:double-integral-B-eta}
\frac{1}{2} \iint_{\Sigma \times \Sigma} B_s ( \langle x,y \rangle) \, \langle \eta(x) - \eta(y), \nu_x - \nu_y \rangle \, d \mu_x \, d \mu_y = \int_{\Sigma} p_s(x) \langle \eta(x), \nu_x \rangle \, d \mu_x
\end{equation}
for every smooth vector field $\eta : \Sigma \to \mathbb{R}^N$.
This property follows from the pointwise identity upon multiplying by $\eta(x)$, integrating in $x$, and symmetrizing in $x$ and $y$.
The double integral in this formulation is absolutely convergent because both differences vanish to first order on the diagonal.
This equality remains valid for links $\Sigma$ with Lipschitz normal vector $x \mapsto \nu_x$.

\begin{lemma}\label{lemma:pv-spherical}
    Let $E = C(\Omega)$ be an $s$-minimal cone in $\bR^N$ whose link has $C^2$ boundary $\Sigma = \partial \Omega \subset \bS^{N-1}$.
    We define the kernel $\hat{\cB}_s(z)$ by
    \[
\hat{\cB}_s(z) := \int_0^{\infty} \rho^{N-1} (1+\rho^2 - 2 \rho z)^{- \frac{N+s}{2}} \, d \rho
\]
Then, for every $x \in \Sigma$, it holds that
\[
\textup{PV}_{\bS^{N-1}} \int_{\bS^{N-1}} ( \mathbf{1}_{\Omega^c}(y) - \mathbf{1}_{\Omega}(y)) \, \hat{\cB}_s( \langle x,y \rangle ) \, d S_y = 0.
\]
Here, $\textup{PV}_{\bS^{N-1}}$ denotes the principal value integral computed by truncations on the sphere.
\end{lemma}
\begin{proof}
The cone $E = C(\Omega)$ has $\mathbf{1}_{E^c}(\rho y) - \mathbf{1}_E(\rho y) = \mathbf{1}_{\Omega^c}(y) - \mathbf{1}_{\Omega}(y)$, and letting
\[
\hat{\cB}_{s,\ve}(z) := \int_0^{\infty} \mathbf{1}_{ \{ 1 + \rho^2 - 2 \rho z > \ve^2 \} } \rho^{N-1} (1 + \rho^2 - 2 \rho z)^{- \frac{N+s}{2}} \, d \rho 
\]
shows that the ambient Euclidean truncation $H_{s,\ve}[E]$ of the $s$-mean curvature satisfies
\[
H_{s,\ve}[E](x) = \int_{\bR^N \setminus B_{\ve}(x)} \frac{\mathbf{1}_{E^c}(X) - \mathbf{1}_E(X)}{|X-x|^{N+s}} \, dX = \int_{\bS^{N-1}} ( \mathbf{1}_{\Omega^c}(y) - \mathbf{1}_{\Omega}(y)) \hat{\cB}_{s,\ve}( \langle x,y \rangle) \, dS_y.
\]
Here, we expanded $|x- \rho y|^2 = 1 + \rho^2 - 2 \rho \langle x,y \rangle$ and $d X = \rho^{N-1} \, d \rho \, dS_y$ for $x,y \in \bS^{N-1}$.
Let $r := d_{\bS^{N-1}}(x,y)$, so using $1 + \rho^2 - 2 \rho \cos r = ( \rho - \cos r)^2 + \sin^2 r$ and $t = \frac{\rho - \cos r}{\sin r}$ shows that
\begin{align*}
    \hat{\cB}_s(\cos r) &= ( \sin r)^{-(N-1+s)} \int_{- \cot r}^{\infty} ( \cos r + t \sin r)^{N-1}(1+t^2)^{- \frac{N+s}{2}} \, dt, \\
    \hat{\cB}_s(\cos r) &\sim c_{N,s} r^{-(N-1+s)} \quad \text{as } \; r \downarrow 0, \qquad c_{N,s} = \int_{\bR} (1+t^2)^{- \frac{N+s}{2}} \,dt = \sqrt{\pi} \frac{\Gamma( \frac{N+s-1}{2})}{\Gamma( \frac{N+s}{2})}.
\end{align*}
as in~\cite{davila-delpino-wei}*{\S 11}.
Hence, $\hat{\cB}_s(\cos r) \leq C_{N,s} r^{-(N-1+s)}$ for sufficiently small $r$.

We study the radial difference kernel between the spherical and Euclidean truncations, namely
\[
\mathfrak{B}_{\ve}(r) := \hat{\cB}_{s,\ve} ( \cos r) - \mathbf{1}_ { \{ r > \ve \} } \hat{\cB}_s(\cos r).
\]
Moreover, let $\sigma(y) := \mathbf{1}_{\Omega^c}(y) - \mathbf{1}_{\Omega}(y)$.
For small $r$, we have
\[
\inf_{\rho \geq 0} |x - \rho y|^2 =\inf_{ \rho \geq 0} (1 + \rho^2 - 2 \rho \cos r) = \sin^2 r,
\]
which shows that $\mathfrak{B}_{\ve}(r)=0$ whenever $r \geq \arcsin \ve$, for $\ve < \frac{1}{10}$.
Also, $|\mathfrak{B}_{\ve}(r)| \leq \hat{\cB}_s(\cos r)$.
Let $\nu \in T_x \bS^{N-1}$ be the unit normal to $\Sigma$ at $x$, with the orientation chosen consistently with $\Omega$, and let $\sigma_0( \exp_x(r\omega)) := \textup{sgn} \la \omega, \nu \rg$ in geodesic polar coordinates $y = \exp_x( r \omega)$ centered at $x$.
The measure satisfies $d S_y = (\sin r)^{N-2} \, dr \, d S_{\omega}$, and because the kernel $\mathfrak{B}_{\ve}(r)$ depends only on $r$, we have
\begin{equation}\label{eqn:sigma-0-1}
\int_{\bS^{N-1}} \sigma_0(y) \mathfrak{B}_{\ve} (d_{\bS^{N-1}}(x,y)) \, dS_y = 0,
\end{equation}
because the tangent half-space at $x$ divides $\bS^{N-2}$ into two equal hemispheres. 
Moreover, the $C^2$ regularity of $\Sigma$ implies that, in a geodesic ball of sufficiently small radius $r$,
\begin{equation}\label{eqn:sigma-2}
\cH^{N-2} \bigl( \{ \omega \in \bS^{N-2} : \sigma(\exp_x (r \omega)) \neq \sigma_0 (\exp_x(r \omega)) \} \bigr) \leq C(\Sigma) \, r
\end{equation}
Indeed, in $B^{\bS}_r(x)$ the $C^2$ hypersurface $\Sigma$ is the graph of a $C^2$ function $u$ over $T_x \Sigma$ satisfying $u(0) = 0$ and $Du(0)= 0$, with $|u(\xi)| \leq C |\xi|^2$, and $d S_y = (\sin r)^{N-2} \, d r \, d S_{\omega}$.
We therefore obtain
\begin{align*}
    \left| \int_{\bS^{N-1}} \sigma(y) \, \mathfrak{B}_{\ve}(r) \, d S_y \right| &= \left| \int_{\bS^{N-1}} (\sigma(y) - \sigma_0(y)) \, \mathfrak{B}_{\ve}(r) \, d S_y \right| \leq C(\Sigma) \int_0^{\arcsin \ve} r \hat{\cB}_s(\cos r) (\sin r)^{N-2} \, d r \\
    &\leq C(N,\Sigma,s) \int_0^{\arcsin \ve} r^{-s} \, dr \leq \tilde{C}(N,\Sigma,s) \, \ve^{1-s}.
\end{align*}
In the last step, we used the bound $\hat{\cB}_s (\cos r) \leq C_{N,s} r^{-(N-1+s)}$ for small $r$.
Consequently,
\[
\lim_{\ve \downarrow 0} \int_{\bS^{N-1}} ( \mathbf{1}_{\Omega^c} (y) - \mathbf{1}_{\Omega}(y)) \, \hat{\cB}_{s,\ve}( \cos r) \, dS_y = \lim_{\ve \downarrow 0} \int_{\bS^{N-1}} ( \mathbf{1}_{\Omega^c}(y) - \mathbf{1}_{\Omega}(y)) \, \mathbf{1}_{ \{ d_{\bS}(x,y) > \ve\} } \hat{\cB}_s(\cos r) \, dS_y.
\]
We may therefore express $H_s[E](x)$ as
\begin{align*}
    &\lim_{\ve \downarrow 0} \int_{\bS^{N-1}} ( \mathbf{1}_{\Omega^c}(y) - \mathbf{1}_{\Omega}(y)) \int_0^{\infty} \mathbf{1}_{ \{ |x - \rho y| > \ve \} } \rho^{N-1} |x- \rho y|^{-(N+s)} \, d \rho \, d S_y \\
    &= \lim_{\ve \downarrow 0} \int_{ \{ d_{\bS}(x,y) > \ve \} } ( \mathbf{1}_{\Omega^c}(y) - \mathbf{1}_{\Omega}(y)) \hat{\cB}_s ( \la x,y \rg) \, dS_y = \textup{PV}_{\bS^{N-1}} \int_{\bS^{N-1}} ( \mathbf{1}_{\Omega^c}(y) - \mathbf{1}_{\Omega}(y) ) \hat{\cB}_s (\langle x, y \rangle) \, dS_y.
\end{align*}
Since $E$ is $s$-minimal, we have $H_s[E](x) = 0$, so the above expression equals zero, as claimed.
\end{proof}

\begin{proposition}\label{prop:angular-}
For every $\bar{s} \in (0,1)$ and $\Lambda \geq 1$, there is a constant $C(N, \bar{s}, \Lambda)$ such that the following property holds.
For $s \in ( 0 , \bar{s})$, let $E = C(\Omega)$ be an $s$-minimal cone in $\bR^N$ whose link has $C^2$ boundary $\Sigma = \partial \Omega \subset \bS^{N-1}$.
Moreover, suppose that $\cH^{N-2}(\Sigma \cap B_r(x)) \leq \Lambda r^{N-2}$ for every $x \in \Sigma$ and $r>0$.
Then, we have the bound
\[
\bigl| 2 |\Omega| - |\bS^{N-1}| \bigr| \leq C(N,\bar{s}, \Lambda) \, \Bigl( \frac{\Lambda}{\cH^{N-2}(\Sigma)} \Bigr)^{\frac{1}{2}} \, s \, \Bigl( \iint_{\Sigma \times \Sigma} B_s( \la x,y \rg) \, |\nu_x - \nu_y|^2 \, d \mu_x \, d \mu_y \Bigr)^{\frac{1}{2}}.
\]
\end{proposition}
\begin{proof}
We proceed in four steps, expressing the left-hand side in terms of appropriate kernels.

\smallskip \noindent \textbf{Step 1:}
Let us write $\mu = \cH^{N-2} \mres \Sigma$.
First, we show that every $x\in \Sigma$ satisfies
\[
    \int_{\Sigma} B_s ( \langle x, y \rangle)  \langle x, \nu_y \rangle \, d \mu_y = 0
\]
and the integral converges absolutely.
Since $\la y, \nu_y \rangle = 0$, we have $\la x, \nu_y \rg = \la x-y, \nu_y \rg$.
Because $\Sigma$ is $C^2$, for $d_{\bS^{N-1}}(x,y)$ small, the distance from $x$ to the tangent space $T_y \Sigma$ is quadratic in $d_{\Sigma}(x,y)$, so
\[
|\la x, \nu_y \rg| = |\la x-y , \nu_y \rg| \leq C(\Sigma) \, d_{\Sigma}(x,y)^2.
\]
On the other hand, $B_s ( \langle x, y \rangle) \leq C(\Sigma)\,  d_{\Sigma}(x,y)^{-(N-1+s)}$ as proved above.
Thus, applying Corollary~\ref{cor:local-integrability} with $m=N-2$ shows that this expression is integrable on $\Sigma$, for $s \in [0,1)$.

Next, we observe that the kernel $\hat{\cB}_s(z)$ introduced in Lemma~\ref{lemma:pv-spherical} satisfies the differential identity
\begin{equation}\label{eqn:bs-diff-eq}
(1-z^2) B'_s(z) - (N-1) z B_s(z) = s \hat{\cB}_s(z).
\end{equation}
Indeed, we have
\begin{align*}
    & \frac{d}{d \rho} \bigl[ \rho^{N-1}(\rho-z) (1+\rho^2 - 2 \rho z)^{- \frac{N+s}{2}} \bigr] \\
    &= (N+s)(1-z^2) \rho^{N-1}(1+\rho^2 - 2 \rho z)^{- \frac{N+s}{2}-1} - (1+\rho^2 - 2 \rho z)^{- \frac{N+s}{2}} \rho^{N-2} [ (N-1) z + s \rho ].
\end{align*}
For $s>0$, we observe that
\[
\rho^{N-1} (\rho-z) (1 + \rho^2  - 2 \rho z)^{- \frac{N+s}{2}} \to 0 \quad \text{as } \; \rho \downarrow 0 \quad \text{and as } \; \rho \to \infty. 
\]
Hence, integrating this identity in $\rho$ and using 
\[
B'_s(z) = (N+s) \int_0^{\infty} \rho^{N-1} ( 1+ \rho^2 - 2 \rho z)^{- \frac{N+s+2}{2}} \, d \rho
\]
proves the differential identity~\eqref{eqn:bs-diff-eq}.
We now fix $x \in \Sigma$ and consider the vector field $Y_x(y) = B_s ( \la x, y \rg) (x - \la x,y \rg y)$.
Because $\nabla_{\bS^{N-1}} \la x, y \rg = x - \la x,y \rg y$, we obtain
\begin{align*}
    |\nabla_{\bS^{N-1}} \la x, y \rg|^2 &= 1 - \la x, y \rg^2, \qquad \Delta_{\bS^{N-1}} \la x, y \rg = - (N-1) \la x,y \rg, \\
    \textup{div}_{\bS^{N-1}} Y_x &= (1-z^2) B'_s(z) - (N-1) z B_s(z) = s \hat{\cB}_s(z). 
\end{align*}
For $\nu$ the outward unit conormal of $\Omega$, the function $\mathbf{1}_{\Omega^c}(y) - \mathbf{1}_{\Omega}(y)$ has distributional derivative
\[
D_{\bS^{N-1}} ( \mathbf{1}_{\Omega^c} - \mathbf{1}_{\Omega}) = 2 \, \nu_{\Omega} \, \cH^{N-2} \mres \Sigma,
\]
so applying Green's theorem~\cite{maggi}*{Proposition~19.22 and Lemma~22.11} to the vector field $( \mathbf{1}_{\Omega^c} - \mathbf{1}_{\Omega}) Y_x(y)$ on the domain $G_{\ve} := \bS^{N-1} \setminus \bar{B}^{\bS}_{\ve}(x)$, with outward conormal $n_{\ve}$, gives
\begin{equation}\label{eqn:greens-theorem}
    \begin{split}
    & s \int_{G_{\ve}} (\mathbf{1}_{\Omega^c} - \mathbf{1}_{\Omega})(y) \hat{\cB}_s( \la x, y \rg ) \, dS_y = \int_{G_{\ve}} ( \mathbf{1}_{\Omega^c} - \mathbf{1}_{\Omega}) \textup{div}_{\bS^{N-1}} Y_x \, d S \\
    &= - 2 \int_{\Sigma \cap G_{\ve}} \la Y_x(y), \nu_y \rg \, d \mu_y + \int_{\partial B^{\bS}_{\ve}(x)} ( \mathbf{1}_{\Omega^c}(y) - \mathbf{1}_{\Omega}(y)) \la Y_x(y), n_{\ve}(y) \rg \, d \cH^{N-2}(y).     \end{split}
\end{equation}
Because $\la y, \nu_y \rg = 0$, we have $\la Y_x(y), \nu_y \rg = B_s( \la x,y \rg) \, \la x, \nu_y \rg$.
Writing $r = d_{\bS^{N-1}}(x,y)$, we find
\begin{equation}\label{eqn:boundary-term-G-eps}
\nabla_{\bS^{N-1}} \la x,y \rg = - \sin r \, \partial_r, \qquad \la Y_x, n_{\ve} \rg = B_s(\cos r) \sin r \quad \text{on } \; \partial B^{\bS}_{\ve}(x),
\end{equation}
because the outward conormal $n_{\ve}$ of $G_{\ve}$ is $- \partial_r$ there.
Moreover, we observe that
\begin{equation}\label{eqn:spherical-separation}
\left| \int_{\partial B^{\bS}_{\ve}(x)} (\mathbf{1}_{\Omega^c}(y) - \mathbf{1}_{\Omega}(y)) \, d \cH^{N-2}(y) \right| \leq C(\Sigma) \, \ve^{N-1}.
\end{equation}
This property follows as in Lemma~\ref{lemma:pv-spherical}, by comparing the function $\sigma(y) := \mathbf{1}_{\Omega^c}(y) - \mathbf{1}_{\Omega}(y)$ with the tangent hemisphere function $\sigma_0( \exp_x(r \omega)) := \textup{sgn} \la \omega, \nu \rg$ in geodesic polar coordinates centered at $x$, which satisfies the bounds~\eqref{eqn:sigma-0-1} and~\eqref{eqn:sigma-2}.
Using $\int_{\partial B^{\bS}_{\ve}(x)} \sigma_0(y) \, d \cH^{N-2} = 0$, we therefore obtain
\[
\left| \int_{\partial B^{\bS}_{\ve}(x)} \sigma(y) \, d \cH^{N-2}  \right| = \left| \int_{\partial B^{\bS}_{\ve}(x)} (\sigma(y) - \sigma_0(y)) \, d \cH^{N-2} \right| \leq C(\Sigma) \ve \cdot \cH^{N-2}(\partial B^{\bS}_{\ve}(x)) \leq C \ve^{N-1}
\]
for $\ve>0$ sufficiently small.
This proves~\eqref{eqn:spherical-separation}.
Moreover, $B_s(\cos r) \leq C (1 - \cos r)^{- \frac{N+s-1}{2}} \leq C r^{-(N+s-1)}$ for $r$ small.
Since $r = \ve$ on $\partial B^{\bS}_{\ve}(x)$, we can estimate the boundary term of~\eqref{eqn:greens-theorem} as
\begin{align*}
    & \left| s \int_{G_{\ve}} ( \mathbf{1}_{\Omega^c} - \mathbf{1}_{\Omega})(y) \, \hat{\cB}_s(\la x,y \rg) \, d S_y + 2 \int_{\Sigma \cap G_{\ve}} \la Y_x(y), \nu_y \rg \, d \mu_y \right| \\
    &\leq B_s(\cos \ve) \sin \ve \; \Bigl| \int_{\partial B^{\bS}_{\ve}(x)} (\mathbf{1}_{\Omega^c} - \mathbf{1}_{\Omega}) \, d \cH^{N-2} \Bigr| \leq C(\Sigma) \, \ve^{-(N+s-1)} \cdot \ve \cdot \ve^{N-1} = C \ve^{1-s}
\end{align*}
by using the computations of~\eqref{eqn:boundary-term-G-eps} and~\eqref{eqn:spherical-separation}.
We now send $\ve \downarrow 0$ and use the principal value formulation of the $s$-stationary equation from Lemma~\ref{lemma:pv-spherical} to see that the first term approaches zero as $\ve \downarrow 0$.
We recall that $\la Y_x(y) , \nu_y \rg = B_s( \la x,y \rg) \la x, \nu_y \rg$ and pass to the limit in Green's theorem,  using the absolute convergence established here; this proves the desired integral identity for $0<s<1$.

\smallskip \noindent \textbf{Step 2:}
Next, we extend the kernel $B_s(z)$ to $s=0$ by
\[
B_0(z) := \int_0^{\infty} \rho^{N-2} (1 + \rho^2 - 2 \rho z)^{-\frac{N}{2}} \, d \rho.
\]
We claim that there is a constant $C(N,\bar{s})$ such that
\begin{equation}\label{eqn:C_n-s-Bs-B0}
    |B_s ( \la x,y \rg) - B_0( \la x,y \rg )| \leq C s \, ( 1 + |\log |x-y||) \, |x-y|^{-(N-1+s)}
\end{equation}
for $x \neq y \in \bS^{N-1}$.
Let $r = \sqrt{2(1-z)} = |x-y|$, so $1 + \rho^2 - 2 \rho z = (\rho-1)^2 + \rho r^2$, and will produce constants $c,C>0$, depending only on $N, \bar{s}$, such that for every $(s,z) \in [0,\bar{s}] \times (0,2]$, it holds that
\begin{equation}\label{eqn:two-kernel-bounds}
    c r^{-(N-1+s)} \leq B_s(z) \leq C r^{-(N-1+s)}, \qquad |\partial_s B_s(z)| \leq C ( 1 + |\log r|) \, r^{-(N-1+s)}.
\end{equation}
Indeed, we have
\[
\partial_s B_s(z) = - \frac{1}{2} \int_0^{\infty} \rho^{N-2} \log (1 + \rho^2 - 2 \rho z) \, (1 + \rho^2 - 2 \rho z)^{- \frac{N+s}{2}} \, d \rho
\]
Suppose first that $0 < r \leq 1$. 
We split the $\rho$-integral into the regions $\{ |\rho-1| \leq \frac{1}{2} \} \cup \{ |\rho-1| > \frac{1}{2} \}$.
On the first region, $\rho \asymp 1$ and $1 + \rho^2 - 2 \rho z \asymp (\rho-1)^2 + r^2$, so for $\rho-1 = rt$, we obtain
\begin{align*}
\int_{|\rho-1| < \frac{1}{2}} (1 + \rho^2 - 2 \rho z)^{- \frac{N+s}{2}} \, d \rho &\leq C r^{1 - N-s} \int_{\bR}(1+t^2)^{- \frac{N+s}{2}} \, dt \leq C r^{-(N-1+s)}, \\
| \log (1 + \rho^2 - 2 \rho z) | &\leq C(N, \bar{s}) \, ( |\log r| + \log (2+t^2) ).
\end{align*}
We therefore obtain
\begin{align*}
    \int_{|\rho-1| \leq \frac{1}{2}} |\log (1+\rho^2 - 2 \rho z)| \, (1+\rho^2 - 2 \rho z)^{- \frac{N+s}{2}} \, d \rho &\leq C r^{-(N-1+s)} \int_{\bR} \frac{1 + |\log r| + \log (2+t^2)}{(1+t^2)^{\frac{N+s}{2}}} \\
    &\leq C(1 + |\log r|) r^{-(N-1+s)}.
\end{align*}
On the region $\frac{1}{2} < |\rho-1| < 10$, the denominator is uniformly bounded below.
For $\rho \to \infty$, we find
\[
\rho^{N-2}(1+\rho^2 - 2 \rho z)^{- \frac{N+s}{2}} \asymp \rho^{-2-s}, \qquad \rho^{N-2} |\log (1 + \rho^2 - 2 \rho z)| \, (1+\rho^2 - 2 \rho z)^{- \frac{N+s}{2}} \leq C (1 + \log \rho) \rho^{-2-s}.
\]
Thus, the complementary contributions to the integral representation of $\partial_s B_s(z)$ are bounded by a constant, and $1 \leq Cr^{-(N-1+s)}$ for $r \leq 2$.
This proves the second bound~\eqref{eqn:two-kernel-bounds}.

For the first bound, on $B_s(z)$, we restrict the defining integral for $B_s(z)$ from~\eqref{eqn:spherical-kernels} to the interval $|\rho-1| \leq \frac{r}{2} \leq \frac{1}{2}$, where $\rho^{N-2} \geq c$ and $1+\rho^2 - 2 \rho z \leq Cr^2$.
Because this interval has length $2r$, we obtain $B_s(z) \geq c r (r^2)^{- \frac{N+s}{2}} = c r^{-(N-1+s)}$ for $r \leq 1$.
When $r \in [1,2]$, equivalently $z \in [ - 1, \frac{1}{2}]$, there is no kernel singularity, so $B_s(z)$ and $\partial_s B_s(z)$ are continuous functions on the compact set $[0, \bar{s}] \times [ - 1, \frac{1}{2}]$, and $B_s(z)> 0$ there.
Therefore, $0 < c \leq B_s(z) \leq C$ and $|\partial_s B_s(z)| \leq C$, while the quantities $r^{-(N-1+s)}$ and $1 + |\log r|$ are bounded above and below by uniform positive constants.

Finally, we integrate the second bound of~\eqref{eqn:two-kernel-bounds} in $s$ to obtain, for $r \in (0,2]$,
\[
|B_s(z) - B_0(z)| = \left| \int_0^s \partial_s B_s(z) \, ds \right| \leq Cs (1 + |\log r|) \, r^{-(N-1+s)}.
\]
The bound is obtained as in our previous argument, for $r \in (0,1]$, and extended to $[1,2]$ by compactness and $|B_s(z) - B_0(z)| \leq Cs$, after enlarging the uniform constants.
Finally, recalling that $z = \la x, y \rg$, with $x \neq y \in \bS^{N-1}$ and $r = \sqrt{2(1-z)} = |x-y|$, we obtain the desired inequality~\eqref{eqn:C_n-s-Bs-B0}.

\smallskip \noindent \textbf{Step 3:} We now suppose that $\Omega \subset \bS^{N-1}$ is $C^2$ and write $\Sigma = \partial \Omega$.
For $s=0$, integrating the above identity shows that the primitive appearing in the computation leading to~\eqref{eqn:bs-diff-eq} satisfies
\[
\lim_{\rho \to \infty} \rho^{N-1} (\rho-z) (1+\rho^2 - 2 \rho z)^{- \frac{N}{2}} = 1, \qquad \implies \qquad (1-z^2) B'_0(z) - (N-1) z B_0(z) = 1.
\]
Hence, the vector field $Y_x(y) = B_0( \langle x,y \rangle ) (x - \langle x,y \rangle y)$ satisfies $\textup{div}_{\bS^{N-1}} Y_x = 1$.
Applying Green's theorem as above, we see that the contribution from $\partial B_{\ve}(x)$ goes to zero as $\ve \downarrow 0$.
Thus,
\begin{equation}\label{eqn:difference-and-B-kernel}
|\bS^{N-1}| - 2 \, |\Omega| = - 2 \int_{\Sigma} B_0 ( \langle x,y \rangle ) \, \langle x, \nu_y \rangle \, d \mu_y , \qquad \text{for every } \; x \in \Sigma.
\end{equation}
The absolute convergence of the integral is obtained by the same argument as above.

\smallskip \noindent \textbf{Step 4:} We now combine the above steps to complete the proof.
First, observe that
\begin{equation}\label{eqn:Hn-2Sigma}
    \cH^{N-2}(\Sigma) \, \bigl( 2\, |\Omega| - |\bS^{N-1}| \Bigr) =  \iint_{\Sigma \times \Sigma} (B_s - B_0) ( \la x, y \rg ) \, \la x-y , \nu_x - \nu_y \rg \, d \mu_x \, d \mu_y.
\end{equation}
For this, we subtract the $s>0$ and $s=0$ integral identities from Steps 1 and 3 to obtain the relation
\begin{align*}
|\Omega| - \tfrac{1}{2} |\bS^{N-1}| &= \int_{\Sigma} (B_0 - B_s) (\la x,y \rg) \, \la x, \nu_y \rg \, d \mu_y \\
\implies \cH^{N-2}(\Sigma) \, \bigl( |\Omega| - \tfrac{1}{2} |\bS^{N-1}| \bigr) &= \iint_{\Sigma \times \Sigma} (B_0 - B_s) (\la x, y\rg) \, \la x, \nu_y \rg \, d \mu_x \, d \mu_y 
\end{align*}
by integrating in $x$.
Interchanging $x$ and $y$ in the second equality, averaging the two resulting expressions, and using $\la x, \nu_x \rg = \la y, \nu_y \rg =0$, we arrive at integration against the term
\[
\la x, \nu_y \rg + \la y, \nu_x \rg = \la x- y , \nu_y - \nu_x \rg = - \la x - y, \nu_x - \nu_y \rg.
\]
This produces the equality~\eqref{eqn:Hn-2Sigma}.
Using the estimates~\eqref{eqn:C_n-s-Bs-B0} and~\eqref{eqn:two-kernel-bounds}  from Step 2, including the kernel lower bound $B_s (\la x,y \rg ) \geq c(N,s) |x-y|^{-(N-1+s)}$, we obtain, for every $x \neq y$, 
\begin{align*}
B_s^{-1}  |B_s - B_0|^2 |x-y|^2 
%&\leq C(N, \bar{s}) \, s^2 \, (1 + |\log |x-y||)^2 |x-y|^{-2(N-1+s)} |x-y|^2 |x-y|^{N-1+s} \\
&\leq C(N, \bar{s}) \, s^2 \, ( 1 + |\log |x-y||)^2 |x-y|^{-(N-3+s)}.
\end{align*}
Thus, we can apply Lemma~\ref{lemma:potential-upper-growth} with $(m,\alpha, \beta_0,q) = (N-2, N-3+s, 1-\bar{s},2)$ to obtain
\[
\sup_{x \in \Sigma} \int_{\Sigma \cap B_1(x)} \frac{|B_s - B_0|^2  |x-y|^2}{B_s} \, d\mu_y \leq C(N,\bar{s}) \, s^2 \Lambda.
\]
We also observe that the bound $\mu(B_r(x)) \leq \Lambda r^{N-2}$ implies that $\mu(\Sigma) \leq C(N) \Lambda$ because $\mu( \{ y\}) = 0$ has no atoms, so $\mu(\Sigma) = \mu(B_2(x)) \leq 2^{N-2} \Lambda$.
Therefore, $\Lambda \leq \Lambda + \cH^{N-2}(\Sigma) \leq C(N) \Lambda$.
On $\{ |x-y| \geq 1 \}$, the integrand is bounded by $C(N,\bar{s}) s^2$, so its contribution to the integral gives
\begin{equation}\label{eqn:Bs-B0-double-bound}
    \begin{split}
    & \sup_{x \in \Sigma} \int_{\Sigma} \frac{|B_s - B_0|^2 |x-y|^2}{B_s} \, d\mu_y \leq C(N,\bar{s}) s^2 (\Lambda + \cH^{N-2}(\Sigma)) \\
    & \implies \iint_{\Sigma \times \Sigma} \frac{|B_s - B_0|^2 |x-y|^2}{B_s} \, d\mu_x \, d\mu_y \leq C(N,\bar{s}) s^2 \Lambda \, \cH^{N-2}(\Sigma)
    \end{split}
\end{equation}
upon bounding $\Lambda + \cH^{N-2}(\Sigma) \leq C(N) \Lambda$.
Finally, we integrate this inequality over $x \in \Sigma$ and use
\begin{align*}
    & \left| \iint (B_s - B_0) ( \la x, y \rg ) \, \la x-y , \nu_x - \nu_y \rg \, d \mu_x \, d \mu_y \right| \\
    & \leq \Bigl( \iint B_s( \la x,y \rg) \, |\nu_x - \nu_y|^2 \, d \mu_x \, d \mu_y \Bigr)^{\frac{1}{2}} \Bigl( \iint \frac{| (B_s- B_0) ( \la x, y \rg) |^2 \, |x-y|^2}{B_s ( \la x,y \rg)} \, d \mu_x \, d \mu_y \Bigr)^{\frac{1}{2}}
\end{align*}
by the Cauchy-Schwarz inequality.
Combined with the relation~\eqref{eqn:Hn-2Sigma}, this gives
\[
\cH^{N-2}(\Sigma) \, \bigl| 2 |\Omega| - |\bS^{N-1}| \, \bigr| \leq C s \, \cH^{N-2}(\Sigma)^{\frac{1}{2}} \Lambda^{\frac{1}{2}}  \Bigl( \iint_{\Sigma \times \Sigma} B_s( \la x,y \rg) \, |\nu_x - \nu_y|^2 \, d \mu_x \, d \mu_y \Bigr)^{\frac{1}{2}}.
\]
The desired inequality follows after cancelling the factor $\cH^{N-2}(\Sigma)>0$.
\end{proof}
The kernel $B_0(z)$ introduced above will be important for studying limits of stable $s$-minimal surfaces as $s \downarrow 0$.
We observe that the substitution $(z,\rho) = (\cos \alpha, \frac{\sin \theta}{\sin (\theta + \alpha)})$ for $\alpha \in (0,\pi)$ and $\theta \in (0,\pi-a)$ makes $1+ \rho^2 - 2 \rho \cos \alpha = \tfrac{\sin^2 \alpha}{\sin^2(\theta+\alpha)}$ and $d \rho =\tfrac{\sin \alpha}{\sin^2(\theta+\alpha)} \, d \theta$.
Consequently,
\begin{equation}\label{eqn:B0-N(z)}
    B_0^{(N)}(z) = (1 - z^2)^{- \frac{N-1}{2}} \int_0^{\arccos(-z)} \sin^{N-2} \theta \, d \theta.
\end{equation}
In the salient dimensions $3 \leq N \leq 6$, we can compute $B_0(z)$ explicitly by a recursive formula:
\begin{align*}
    B_0^{(3)}(z) &= \frac{1}{1-z}, \qquad & B_0^{(4)}(z) &= \frac{\arccos(-z) + z \sqrt{1-z^2}}{2 (1-z^2)^{\frac{3}{2}}}, \\
    B^{(5)}_0(z) &= \frac{2-z}{3(1-z)^2}, \qquad & B^{(6)}_0(z) &= \frac{3 \arccos(-z) + z(5-2z^2) \sqrt{1-z^2}}{8 (1-z^2)^{\frac{5}{2}}}.
\end{align*}

\subsection{The stability inequality on the link}\label{section:stability-inequality}

We now obtain a reduction of the stability inequality for $s$-minimal surfaces that is valid in every dimension and for every $0 < s < 1$.

\begin{lemma}\label{lemma:quadratic-transformation}
We express $\phi(e^t x) = e^{- \frac{N-s-2}{2} t} \psi(t,x)$ with $\psi \in C_c^{\infty} ( \mathbb{R} \times \Sigma)$.
Then, the stability form $Q_{s,M}[\phi]$ of the cone $M = \partial E$ satisfies
\begin{align*}
    Q_{s,M}[\phi] &= \frac{1}{2} \int_{\mathbb{R}} \int_{\mathbb{R}} \iint_{\Sigma \times \Sigma} P_s(\tau, \langle x, y \rangle) [ \psi(t,x) - \psi(t+\tau,y) ]^2 \, d \mu_x \, d \mu_y \, d \tau \, dt \\
    & \quad + \int_{\mathbb{R}} \int_{\Sigma} \psi(t,x)^2 \int_{\Sigma} [ W_s( \langle x,y \rangle) - B_s ( \langle x,y \rangle) (1 - \langle \nu_x, \nu_y \rangle) ] \, d \mu_y \, d \mu_x \, dt.
\end{align*}
Moreover, all the integrals in this representation are finite and
\[
V_s(x) := \int_{\Sigma} [ W_s( \langle x,y \rangle) - B_s( \langle x,y \rangle) (1- \langle \nu_x, \nu_y \rangle )] \, d \mu_y \in L^{\infty}(\Sigma).
\]
\end{lemma}
\begin{proof}
In the coordinates $(X,Y) = (e^t x, e^u y)$, we have
\[
d \sigma_X \, d \sigma_Y = e^{(N-1)(t+u)} \, dt \, du \, d \mu_x \, d \mu_y, \qquad |X-Y|^{N+s} = e^{\frac{N+s}{2}(t+u)} [ 2(\cosh(t-u) - \langle x,y \rangle) ]^{\frac{N+s}{2}}.
\]
Thus, the measure and the denominator contribute a factor of $e^{\frac{N-2-s}{2}(t+u)} P_s (t - u, \langle x, y \rangle)$, and the first summand in the expression $Q_{s,M}[\phi]$ becomes
\[
\frac{1}{2} \int_{\bR} \int_{\bR} \iint_{\Sigma \times \Sigma} P_s (\tau, \langle x , y \rangle) [ e^{\frac{N-2-s}{2} \tau} \psi(t,x)^2 + e^{- \frac{N-2-s}{2} \tau} \psi(t+\tau,y)^2 - 2 \psi(t,x) \psi(t+\tau,y) ] \, d \mu_x \, d \mu_y \, d \tau \, dt.
\]
By interchanging $x$ and $y$, replacing $\tau$ by $- \tau$, and translating the $t$-variable, we see that the two diagonal terms have the same integral.
Consequently, the above expression equals
\begin{align*}
    & \frac{1}{2} \int_{\bR} \int_{\bR} \iint_{\Sigma \times \Sigma} P_s ( \tau, \langle x,y \rangle) [ \psi(t,x) - \psi(t+\tau,y) ]^2 \, d \mu_x \, d \mu_y \, d \tau \, dt \\
    & \quad + \int_{\bR} \int_{\Sigma} \psi(t,x)^2 \int_{\bR} \int_{\Sigma} P_s (\tau, \langle x, y \rangle) ( e^{\frac{N-2-s}{2} \tau} - 1) \, d \mu_y \, d \tau \, d \mu_x \, dt.
\end{align*}
Using the integral expressions for $A_s(z), B_s(z)$ from~\eqref{eqn:A-B-expansions}, we express the second term above as
\[
\int_{\bR} \int_{\Sigma} \psi(t,x)^2 \int_{\Sigma} W_s( \langle x, y \rangle) \, d \mu_y \, d \mu_x \, dt.
\]
Finally, the second term in the expression of $Q_{s,M}[\phi]$ transforms into
\[
\int_M \phi(X)^2 \int_M \frac{1 - \langle \nu_X, \nu_Y \rangle}{|X-Y|^{N+s}} \, d \sigma_Y \, d \sigma_X = \int_{\bR} \int_{\Sigma} \psi(t,x)^2 \int_{\Sigma} B_s( \langle x, y \rangle) (1 - \langle \nu_x, \nu_y \rangle) \, d \mu_y \, d \mu_x \, dt.
\]
This produces the desired expression for $Q_{s,M}[\phi]$.
Now, the estimates established above show that
\[
\int_{\Sigma} |W_s( \langle x,y \rangle) - B_s( \langle x,y \rangle) (1 - \langle \nu_x, \nu_y \rangle)| \, d \mu_y \leq C.
\]
This proves that $V_s \in L^{\infty}(\Sigma)$.
Next, the second term in the above expression for $Q_{s,M}[\phi]$ satisfies
\[
|  \psi(t,x) - \psi(t+\tau,y) | \leq C( |\tau| + d_{\Sigma}(x,y))
\]
near $\tau =0$ and $y=x$, because $\psi$ is smooth.
In the $N-1$ variables consisting of $\tau$ and local coordinates on $\Sigma$, the corresponding radial integrand is bounded by $C r^{N-2} r^2 r^{-(N+s)} \, dr = C r^{-s} \, dr$, hence integrable for $s \in (0,1)$.
For large $|\tau|$, the bound $P_s(\tau, \langle x, y \rangle) \leq C e^{- \frac{N+s}{2} |\tau|}$ and the fact that $\psi$ is compactly supported in the $t$-variable implies the desired integrability of the second term.
\end{proof}

For a function $g$ on the link $\Sigma$, symmetrizing the stability operator $Q_{s,M}[g]$ produces a form
\begin{align*}
q_s[g] &= \frac{1}{2} \iint_{\Sigma \times \Sigma} A_s ( \langle x, y \rangle) ( g(x) - g(y))^2 \, d \mu_x \, d \mu_y + \int_{\Sigma} V_s(x) g(x)^2 \, d \mu_x, \qquad \text{where} \\
V_s(x) &:= \int_{\Sigma} [ W_s( \langle x,y \rangle) - B_s ( \langle x,y \rangle) (1 - \langle \nu_x, \nu_y \rangle) ] \, d \mu_y \in L^{\infty}(\Sigma)
\end{align*}
as in Lemma~\ref{lemma:quadratic-transformation}.
By Corollary~\ref{cor:local-integrability}, the first integral in the definition of $q_s[g]$ is also finite, due to $A_s(\la x,y \rg) \leq C \, d_{\Sigma}(x,y)^{-(N+s-1)}$ and $( g(x) - g(y))^2 \leq \textup{Lip}(g)^2 \, d_{\Sigma}(x,y)^2$.
The second term is finite because $V_s \in L^{\infty}(\Sigma)$.
In terms of the expression~\eqref{eqn:A-B-expansions}, we may further write
\begin{equation}\label{eqn:mellin-transform}
q_s [g] = \frac{1}{2} \int_{\bR} \iint_{\Sigma \times \Sigma} P_s (\tau, \langle x, y \rangle) (g(x) - g(y))^2 \, d \mu_x \, d \mu_y \, d \tau + \int_{\Sigma} V_s(x) g(x)^2 \, d \mu_x.
\end{equation}
Thus, $g \mapsto q_s[g]$ is a well-defined operator on smooth functions on $\Sigma$. 
\begin{lemma}\label{lemma:extend-to-moduli}
    Suppose that $q_s[f] \geq 0$ for every smooth real-valued function $f$ on $\Sigma$.
    Then, for every smooth complex-valued function $g$ on $\Sigma$, we have $q_s[|g|] \geq 0$.
\end{lemma}
We note that the expression $q_s [ |g|] $ is well-defined because $|g|$ is Lipschitz.
\begin{proof}
    For $\ve > 0$, we define the function $g_{\ve} := ( |g|^2 + \ve^2)^{\frac{1}{2}}$.
    The function $g_{\ve}$ is smooth and real-valued,, so $q_s[g_{\ve}] \geq 0$.
    Moreover, because $r \mapsto (r^2 + \ve^2)^{\frac{1}{2}}$ defines a $1$-Lipschitz map on $[0,\infty)$, we can bound the terms in the integrand by
    \[
    |g_{\ve}(x) - g_{\ve}(y)| \leq | |g(x)| - |g(y)| | \leq |g(x) - g(y)|.
    \]
    In addition, we have $g_{\ve} \to |g|$ uniformly on $\Sigma$, and hence
    \[
    A_s ( \langle x, y \rangle) ( g_{\ve}(x) - g_{\ve}(y))^2 \leq A_s ( \langle x,y \rangle) |g(x) - g(y)|^2.
    \]
    For $g \in C^2$, the right-hand side is integrable on $\Sigma \times \Sigma$, so dominated convergence implies
    \[
    \iint_{\Sigma \times \Sigma} A_s ( \langle x,y \rangle) (g_{\ve}(x) - g_{\ve}(y))^2 \, d \mu_x \, d \mu_y \to \iint_{\Sigma \times \Sigma} A_s ( \langle x, y \rangle) ( |g(x)| - |g(y)|)^2 \, d \mu_x \, d \mu_y.
    \]
    On the other hand, writing $g_{\ve}^2 = |g|^2 + \ve^2$, we can bound the potential term as
    \[
    \int_{\Sigma} V_s g_{\ve}^2 \, d \mu - \int_{\Sigma} V_s |g|^2 \, d\mu = \ve^2 \int_{\Sigma} V_s \, d\mu \to 0.
    \]
    Combining these properties, we deduce that $q_s [g_{\ve}] \to q_s [|g|]$.
    Finally, the fact that $q_s[g_{\ve}] \geq 0$ for every $\ve > 0$ shows that $q_s [ |g|] \geq 0$, as claimed.
\end{proof}
\begin{lemma}\label{lemma:fourier-decomposition}
    Given a function $\psi \in C_c^{\infty}( \bR \times \Sigma)$, consider the Fourier transform
    \[
    \hat{\psi}(\xi,x) = \int_{\bR} e^{- i t \xi} \, \psi(t,x) \, dt.
    \]
    Then, the function $\phi(e^t x) = e^{- \frac{N-2-s}{2} t} \psi(t,x)$ satisfies $Q_{s,M}[\phi] = (2 \pi)^{-1} \int_{\bR} \fq_{s,\xi} [ \hat{\psi}(\xi,\cdot)] \, d \xi$, where
    \[
    \fq_{s,\xi}[g] = \frac{1}{2} \int_{\bR} \iint_{\Sigma \times \Sigma} P_s (\tau,\langle x, y \rangle) |g(x) - e^{i \xi \tau}g(y)|^2 \, d \mu_x \, d \mu_y \, d \tau + \int_{\Sigma} V_s(x) |g(x)|^2 \, d \mu_x.
    \]
\end{lemma}
\begin{proof}
    For every fixed triple $(\tau,x,y)$, Plancherel's theorem gives
    \[
    \int_{\bR} |\psi(t,x) - \psi(t+\tau,y)|^2 \, dt = \frac{1}{2 \pi} \int_{\bR} |\hat{\psi}(\xi,x) - e^{i \xi \tau} \hat{\psi}(\xi,y)|^2 \, d \xi.
    \]
    We work with the transformed expression for $Q_{s,M}[\phi]$ obtained in Lemma~\ref{lemma:quadratic-transformation}.
    Because the first term has a non-negative integrand, we can apply Tonelli's theorem to integrate this identity against the kernel $P_s(\tau, \langle x,y \rangle) \, d \tau \, d \mu_x \, d \mu_y$.
    Moreover, for the potential term, we have $V_s \in L^{\infty}(\Sigma)$, so Plancherel's theorem gives
    \[
    \int_{\bR} \int_{\Sigma} V_s(x) |\psi(t,x)|^2 \, d \mu_x \, dt = \frac{1}{2 \pi} \int_{\bR} \int_{\Sigma} V_s(x) |\hat{\psi}(\xi,x)|^2 \, d \mu_x \, d \xi.
    \]
    The desired decomposition follows from combining these identities.
\end{proof}

\begin{proposition}\label{prop:qsg}
An $s$-minimal cone $E = C(\Omega)$ is stable if and only if every smooth real-valued function $g$ on $\Sigma$ satisfies $q_s[g] \geq 0$.
\end{proposition}
\begin{proof}
We note that the function $\fq_{s,\xi}[g]$ of Lemma~\ref{lemma:fourier-decomposition} satisfies $\fq_{s,\xi} [g] \geq q_s [|g|]$ for every complex-value function $g$: for all $z, w \in \bC$ and $\theta \in \bR$, we have $||z| - |w|| \leq |z - e^{i \theta} w|$, hence
\[
| |g(x)| - |g(y)| |^2 \leq | g(x) - e^{i \xi \tau} g(y) |^2, \qquad \text{for every } \; x, y, \xi, \tau.
\]
Integrating this relation and using $\int_{\bR} P_s(\tau, \la x, y \rg ) \, d \tau = A_s( \la x,y \rangle)$, we obtain
\[
\frac{1}{2} \int_{\bR} \iint P_s(\tau, \la x, y \rg ) \, |g(x) - e^{i \xi \tau} g(y)|^2 \, d \mu_x \, d \mu_y \, d \tau \geq \frac{1}{2} \iint A_s (\la x,y \rg) \, ( |g(x)| - |g(y)|)^2 \, d \mu_x \, d \mu_y
\]
as integrals over $\Sigma = \partial \Omega$.
Moreover, the potential terms in the functions $q_{s,\xi}[g]$ and $q_s [ |g|]$, proving our assertion.
Consequently, if $q_s[f] \geq 0$ for every smooth real-valued function $f$, then Lemma~\ref{lemma:extend-to-moduli} shows that $q_s [|g|] \geq 0$ for every $g \in C^{\infty}(\Sigma)$, thus also 
\[
\fq_{s,\xi}[ \hat{\psi}(\xi,\cdot) ] \geq q_s [ |\hat{\psi}(\xi,\cdot)|] \geq 0, \qquad \text{for every } \; \psi \in C_c^{\infty} ( \bR \times \Sigma).
\]
Applying Lemma~\ref{lemma:fourier-decomposition}, we conclude that $Q_{s,M}[\phi] = (2 \pi)^{-1} \int_{\bR} \fq_{s,\xi} [ \hat{\psi}(\xi,\cdot)] \, d \xi \geq 0$ for every $\phi \in C_c^{\infty}(\Sigma)$, hence the cone $E = C(\Omega)$ is stable. 

Conversely, suppose that $Q_{s,M}[\phi] \geq 0$ for every $\phi \in C_c^{\infty}(\Sigma)$.
For each sufficiently large $R$, we can choose functions $\chi_R \in C_c^{\infty}(\mathbb{R})$ satisfying
\[
0 \leq \chi_R \leq 1, \qquad \chi_R = 1 \quad \text{on } \; [-R,R], \qquad \chi_R = 0 \quad \textup{outside } \; [-R-1,R+1],
\]
and $\int_{\mathbb{R}} |\chi'_R|^2 \, dt \leq C$, for a uniform constant $C$.
We define $\psi_R(t,x) = \chi_R(t) g(x)$ and $\phi_R(e^t x) = e^{- \frac{N-2-s}{2} t} \psi_R(t,x)$, so $\phi_R \in C_c^{\infty}( M \setminus \{ 0 \})$ and $Q_{s,M}[\phi_R] \geq 0$ for every finite $R$, by stability.
Write
\[
I_R := \int_{\mathbb{R}} \chi^2_R \, dt = 2R + O(1), \qquad c_R(\tau) =: I_R^{-1} \int_{\mathbb{R}} \chi_R(t) \chi_R(t+\tau) \, dt.
\]
Because translations preserve the $L^2$-norm, we can record the identity
\begin{equation}\label{eqn:record-tau}
   2 I_R ( 1 - c_R(\tau)) = \int_{\bR} |\chi_R(t+\tau) - \chi_R(t)|^2 \, dt.
\end{equation}
Moreover, the fundamental theorem of calculus and the Cauchy-Schwarz inequality give the bound
\[
    \| \chi_R(\cdot + \tau) - \chi_R \|_{L^2} = \Bigl\| \int_0^{\tau} \chi'_R(\cdot + s) \, ds \Bigr\|_{L^2} \leq \int_{\min \{ 0 , \tau \} }^{ \max \{ 0, \tau \} } \| \chi'_R(\cdot + s) \|_{L^2} \, ds = |\tau| \cdot \| \chi'_R \|_{L^2}.
\]
Therefore, $\| \chi_R( \cdot + \tau) - \chi_R \|_{L^2} \leq C |\tau|$, and combining these bounds implies that
\begin{equation}\label{eqn:cr-tau}
    0 \leq 1 - c_R(\tau) \leq (2 I_R)^{-1} \tau^2 \| \chi'_R \|^2_{L^2} \leq C R^{-1} \tau^2.
\end{equation}
Using the representation formula for $Q_{s,M}$ from Lemma~\ref{lemma:quadratic-transformation}, we obtain an integral inequality over $\Sigma$,
\begin{align*}
    \frac{Q_{s,M}[\phi_R]}{I_R} &= \frac{1}{2} \int_{\bR} \iint P_s(\tau,\la x,y \rg) [ g(x)^2 + g(y)^2 - 2 c_R(\tau) g(x) g(y)] \, d \mu_x \, d \mu_y \, d \tau + \int V_s(x) g(x)^2 \, d \mu_x.
\end{align*}
On the other hand, we recall that, with integrals taken over $\Sigma$,
\begin{align*}
    q_s[g] = \frac{1}{2} \int_{\bR} \iint P_s(\tau, \la x,y \rg) \, [ g(x)^2 + g(y)^2 - 2 g(x) g(y) ] \, d \mu_x \, d \mu_y \, d \tau + \int V_s(x) g(x)^2 \, d \mu_x.
\end{align*}
Subtracting the two relations and recalling the integral estimates for $P_s$, from Lemma~\ref{lemma:kernel-k-s-tau,z-bounds}, together with the bounds~\eqref{eqn:record-tau} and~\eqref{eqn:cr-tau}, we obtain
\begin{align*}
\bigl| I_R^{-1} Q_{s,M}[\phi_R] - q_s[g] \bigr| &= \left| \int_{\bR} \iint_{\Sigma \times \Sigma} P_s(\tau, \la x,y \rg) (1 - c_R(\tau)) g(x) g(y) \, d \mu_x \, d \mu_y \, d \tau \right| \\
&\leq C R^{-1} \|g\|^2_{L^{\infty}} \int_{\bR} \tau^2 \sup_{\Sigma} \int_{\Sigma} P_s ( \tau, \la x, y \rg ) \, d \mu_y \, d \tau \\
& \leq C_g R^{-1} \Bigl( \int_0^1 \tau^{-s} \, d \tau + \int_1^{\infty} \tau^2 e^{- \frac{N+s}{2} \tau} \, d \tau).
\end{align*}
The constant $C_g$ depends only on $N,s,\Sigma, g$, while the last line contains the sum of two finite terms, which are bounded in terms of $N,s$.
We therefore obtain
\[
|I_R^{-1} Q_{s,M}[\phi_R] - q_s[g]| \leq C_g R^{-1}, \qquad I_R^{-1} Q_{s,M}[\phi_R] \to q_s[g] \quad \text{as } \; R \to \infty.
\]
Since $I_R> 0$ and $Q_{s,M}[\phi_R] \geq 0$, by stability, passing to the limit as $R \to \infty$ shows that $q_s[g] = \lim_{R \to \infty} I_R^{-1} Q_{s,M}[\phi_R] \geq 0$.
This establishes the converse and completes the proof.
\end{proof}

\section{Compactness of stable \texorpdfstring{$s$}{s}-minimal cones for \texorpdfstring{$s\sim0$}{s near 0}}\label{sec:compactness}

Next, we will establish area estimates and a compactness result for stable $s$-minimal surfaces as the fractional parameter $s$ tends to zero.
Area estimates for stable $s$-minimal surfaces were also obtained by Cinti-Serra-Valdinoci in~\cite{cinti-serra-valdinoci}*{Theorem 1.1}, with a constant that depends on the parameter $s$ and may become infinite as $s \downarrow 0$.
In our situation, it is crucial to obtain a bound that remains bounded in this regime, so we need to estimate terms more carefully.
We follow~\cite{savin-valdinoci}*{Lemma 1} and~\cite{cinti-serra-valdinoci}*{\S 2} and perform a key reduction to remove the far-field term.

Let $K: \bR^N \to [0,\infty]$ be a non-negative even function and let $U \subset \bR^N$ be an open set.
We define the $K$-perimeter of the measurable set $E$ in $U$ by 
\[
\textup{Per}_K(E;U) := \iint_{(E \cap U) \times E^c} K ( x-y) \, dx \, dy + \iint_{(E \setminus U) \times (E^c \cap U)} K ( x-y) \, dx \, dy.
\]
For $K_s(z) = |z|^{-N-s}$, this definition recovers the fractional perimeter $\textup{Per}_s(E;U)$.
A set $E$ with $P_K(E;U) < \infty$ is \textbf{stable under rearrangements} in $U$ if for any given vector field $X = X(x,t) \in C_c^2(U \times (-1,1) ; \bR^N)$ and $\ve>0$, there is a $t_{\ve} > 0$ such that, with $\Phi_t$ the integral flow of $X$, we have
\[
\min \{ \textup{Per}_K ( \Phi_t(E) \cup E ; U) \, , \, \textup{Per}_K( \Phi_t(E) \cap E ; U) \} \geq \textup{Per}_K(E ; U) - \ve t^2
\]
for all $t \in (-t_{\ve},t_{\ve})$; see~\cite{cinti-serra-valdinoci}*{Definition 1.6} and~\cite{caselli}*{Definition A}.
Stability under rearrangements is hereditary, in the sense that if $E$ is stable in a set $U$, then it is also stable in every $U' \subseteq U$.
Moreover, this property is preserved under translation and rescaling.
\begin{lemma}\label{lemma:stable-under-flows}
Let $E$ be stable under rearrangements in $B_4$, with respect to a kernel $K$ with $K(z) \geq \kappa>0$ for $|z| \leq 2$.
Suppose that, for every $v \in \bS^{N-1}$, there exists a smooth flow $\Phi^v_t$, compactly supported in $B_4$, with the property that $\Phi^v_t(E) \cap B_1 = (E+tv) \cap B_1$, for small $|t|$, and
\[
\textup{Per}_K( \Phi^v_t(E) ; B_4) + \textup{Per}_K ( \Phi^v_{-t}(E) ; B_4)  - 2 \, \textup{Per}_K(E;B_4) \leq \Lambda t^2,
\]
for some $\Lambda \geq 0$ independent of $v$.
Then, $\mathbf{1}_E \in \textup{BV}(B_1)$, namely $E$ has finite perimeter in $B_1$, and
\[
\textup{Per}(E;B_1) \leq C_N ( 1 + \kappa^{- \frac{1}{2}} \Lambda^{\frac{1}{2}}).
\]
\end{lemma}
\begin{proof}
This result follows from combining~\cite{cinti-serra-valdinoci}*{Lemmas 2.4 and 2.5}.    
\end{proof}

\begin{lemma}\label{lemma:rearrangement-stable}
Consider an open set $U \subseteq \bR^N$.
If $E$ is a local $s$-perimeter minimizer in $U$, then $E$ is stable under rearrangements in $U$.
If $E$ is $s$-stationary in $U$, $\partial E \cap U$ is $C^2$, and $Q_{s,\partial E}[\phi] \geq 0$ for every $\phi \in C_c^{\infty}(U)$, then $E$ is stable under rearrangements in $U$.
\end{lemma}
\begin{proof}
The fact that local $s$-perimeter minimizers are stable under rearrangements follows from minimality, because $E \cap \Phi_t(E)$ and $E \cup \Phi_t(E)$ agree with $E$ outside the support of the flow, hence
\[
\textup{Per}_s(E;U) \leq \textup{Per}_s (E \cup \Phi_t(E) ; U) \qquad \text{and} \qquad \textup{Per}_s(E;U) \leq \textup{Per}_s(E \cap \Phi_t(E) ; U).
\]
For sets with $\partial E \cap U$ of class $C^2$, \cite{cabre-cinti-serra}*{Remark 3.2} proves that rearrangement and second variation stability are equivalent, using the identities for $h = \la X(\cdot, 0), \nu_E \rangle$,
\begin{align*}
\liminf_{t \to 0 } \frac{\textup{Per}_s(E \cup \Phi_t(E);U) - \textup{Per}_s(E;U)}{t^2} &= 2\, Q_{s, \partial E}[h_+], \\
\liminf_{t \to 0} \frac{\textup{Per}_s(E \cap \Phi_t(E);U) - \textup{Per}_s(E;U)}{t^2} &= 2\, Q_{s, \partial E}[h_-].
\end{align*}
Because the functions $h_{\pm}$ are Lipschitz, they can be approximated in the norm induced by the quadratic form $Q_{s,\partial E}$ by smooth functions, as discussed in Lemma~\ref{lemma:extend-to-moduli}.
The local singularity is integrable by Corollary~\ref{cor:local-integrability} because $|h_{\pm}(x) - h_{\pm}(y)|^2 \leq C |x-y|^2$.
For a $C^2$ boundary, the potential term in the expression of Lemma~\ref{lemma:quadratic-transformation} has the same integrability due to $|\nu_x - \nu_y| \leq C |x-y|$.
\end{proof}

The following Lemma will be the key ingredient that allows us to upgrade the arguments of~\cite{cinti-serra-valdinoci}*{\S 2} to obtain a perimeter bound with a uniformly bounded constant as $s \downarrow 0$.
\begin{lemma}\label{lemma:bv-translation}
    Let $U \Subset V \subseteq \bR^N$ be open sets and consider a vector $v \in \bR^N$ with the property that
    \[
    U + [0,v] := \{ x + tv : x \in U , \; 0 \leq t \leq 1 \} \Subset V.
    \]
    Then, for every function $u \in \textup{BV}_{\textup{loc}}(\bR^N)$, it holds that
    \[
    \int_U |u(x+v) - u(x)| \, dx \leq |v| \, |Du|(V).
    \]
    In particular, if $E$ is a Caccioppoli set, then $\mathbf{1}_E \in \textup{BV}_{\textup{loc}}(\bR^N)$ and
    \[
    \int_U |\mathbf{1}_E(x+v) - \mathbf{1}_E(x)| \, dx \leq |v| \, \textup{Per}(E;V).
    \]
\end{lemma}
\begin{proof}
    When $u \in C^{\infty}(\bR^N)$ is a smooth function, the fundamental theorem of calculus gives
    \begin{align*}
        \Bigl| u(x+v) - u(x) \Bigr| &= \Bigl| \int_0^1 \la \nabla u(x+tv) , v \rangle \, dt \Bigr| \leq |v| \int_0^1 |\nabla u(x+tv)| \, dt, \\
        \int_U |u(x+v) - u(x)| \, dx &\leq |v| \int_0^1 \int_U |\nabla u(x+tv)| \, dx \, dt = |v| \int_0^1 \int_{U+tv} |\nabla u(y)| \, dy \, dt \\
        &\leq |v| \int_V |\nabla u(y)| \, dy.
    \end{align*}
    For $u \in \textup{BV}_{\textup{loc}}(\bR^N)$, we can choose a smooth approximation $u_k \to u$ in $L^1_{\textup{loc}}(V)$ such that $\int_V |\nabla u_k| \to |Du|(V)$, by using a strict approximation on an intermediate open set $W$ with $(U + [0,v]) \Subset W \subseteq V$.
    The desired bound then follows upon applying the smooth estimate and passing to the limit of BV functions.
    Finally, if $E$ is a Caccioppoli set, then the function $\mathbf{1}_E \in \textup{BV}_{\textup{loc}}(\bR^N)$, and $|D \mathbf{1}_E|(V) = \textup{Per}(E;V)$.
    Thus, the above result proves the claim.
\end{proof}

We now prove the uniform perimeter estimate for stable $s$-minimal surfaces.
The key inequality~\eqref{eqn:per-symmetrization-bound} follows the argument of~\cite{cinti-serra-valdinoci}*{Lemma 2.1}, but we use Lemma~\ref{lemma:bv-translation} to bound terms more carefully.
We produce a volume-preserving deformation that improves their result, enabling us to replace the kernel $K^*$ used in the perimeter bound of~\cite{cinti-serra-valdinoci} by $P_{K^{\sharp}_s}$, where $K^{\sharp}_s(z) = |z|^{-N-s} \min \{ 1, |z|^{-1}\}$.
This removes the far-field effect and the $\frac{1}{s}$ degeneracy in the resulting perimeter bound.

\begin{proposition}\label{prop:bound-s}
For every $N \geq 2$ and $\bar{s} \in (0,1)$, there exists a constant $C(N, \bar{s}) > 0$ with the following property.
For $s \in (0, \bar{s})$, suppose that $E \subset \mathbb{R}^N$ is $s$-minimal and stable under rearrangements in $B_{8R}(x_0)$.
Then, $E$ satisfies
\[
\textup{Per} (E; B_R(x_0)) \leq C(N, \bar{s}) R^{N-1},
\]
where $\textup{Per}(E;U)$ denotes the classical perimeter of the Caccioppoli set $E$ in $U \subseteq \bR^N$.
\end{proposition}
\begin{proof}
    We fix some $v \in \bS^{N-1}$ and work in the unit scale, with $x_0 = 0$.
    We choose a $\varphi \in C_c^{\infty}(B_4)$ with $\varphi = 1$ on $B_2$ and define the antisymmetric tensor $A_{ij}(x)$ and the vector field $X_i(x)$ by
    \[
    \textstyle{A_{ij}(x) = \frac{1}{N-1} \, \varphi(x) (v_i x_j - v_j x_i), \qquad X_i(x) = \sum_{j=1}^N \partial_j A_{ij}(x).}
    \]
    Because $A_{ij} = - A_{ji}$ is antisymmetric, $X$ is divergence-free, since $\textup{div} \, X = \sum_{i,j} \partial_i \partial_j A_{ij} = 0$.
    Moreover, in the region where $\varphi=1$, we obtain $X=v$ on $B_2$ and $X=0$ near $\bR^N \setminus B_4$, due to
    \[
    X_i = \tfrac{1}{N-1} \textstyle{ \sum_j \partial_j (v_i x_j - v_j x_i) } = \tfrac{1}{N-1} \sum_j (v_i - v_j \delta_{ij}) = \tfrac{1}{N-1}(N v_i - v_i) = v_i.
    \]
    Let $\Phi^v_t$ be the flow generated by the vector field $X$, so $\det D \Phi^v_t \equiv 1$ because $X$ is divergence-free.
    Since $\Phi^v_t$ is the identity in a neighborhood of $\partial B_4$, we have $\Phi^v_t(B_4) = B_4$.
    Also, for $|t|$ small, we can write $\Phi^v_t(x) = x + tv$ for $x \in B_1$ and $|t|<\ve$.
    This also gives $\Phi^v_t(E) \cap B_1 = (E+tv) \cap B_1$, because both the forward and inverse trajectories of points of $B_1$ remain in $B_2$.

    We consider the kernel $K_s( z) = |z|^{-(N+s)}$ and let $\hat{\cA} := \bR^{2N} \setminus (B^c_4 \times B^c_4)$.
    Then,
    \[
    \textup{Per}_s ( \Phi^v_t(E) ; B_4) = \frac{1}{2} \iint_{\hat{\cA}} |\mathbf{1}_E(x) - \mathbf{1}_E(y)|^2 K_s ( \Phi^v_t(x) - \Phi^v_t(y)) \, dx \, dy.
    \]
    For fixed $x,y$, we set $w_t := \Phi^v_t(x) - \Phi^v_t(y)$ and $f(t) := K_s(w_t)$.
    Then,
    \[
    \dot{w}_t = X(\Phi^v_t(x)) - X(\Phi^v_t(y)), \qquad \ddot{w}_t = D X (\Phi^v_t(x)) X(\Phi^v_t(x)) - D X (\Phi^v_t(y)) X(\Phi^v_t(y)).
    \]
    Consequently, the smoothness of $X$ implies, uniformly for small $t$, the bound $|\dot{w}_t| + |\ddot{w}_t| \leq C \min \{ 1, |z| \}$, where $z := x-y$.
    We now define the integral kernel
    \[
    K^{\sharp}_s(z) := |z|^{-N-s} \min \{ 1, |z|^{-1} \}, \qquad \text{with } \quad \int_{\bR^N} K^{\sharp}_s(z) \min \{ 1, |z| \} \, dz < \infty.
    \]
    This satisfies $|D^{\ell} K_s(\zeta)| \leq C(N, \bar{s}) \, |\zeta|^{-N-s-\ell}$ for $\ell \in \{1,2\}$, where $C(N, \bar{s})$ is independent of $s$.

    Now, the map $\Phi^v_t$ is uniformly bi-Lipschitz for small $t$, so we have $|w_t| \asymp |z|$ for $|t| < \ve$.
    Moreover,
    \begin{align*}
    & f''(t) = D^2 K_s(w_t) [ \dot{w}_t, \dot{w}_t] + \la D K_s(w_t) , \ddot{w}_t \rangle, \qquad |f''(t)| \leq C |z|^{-N-s} \min \{ 1, |z|^{-1} \}, \\
    & \implies |K_s(w_t) + K_s(w_{-t}) - 2 K_s(z)| \leq C(N, \bar{s}) \, t^2 |z|^{-N-s} \min \{ 1, |z|^{-1} \} = C(N, \bar{s})  \, t^2 K^{\sharp}_s(z).
    \end{align*}
    Having obtained this bound, we may now argue as in~\cite{savin-valdinoci}*{Lemma 1} and~\cite{cinti-serra-valdinoci}*{Lemmas 2.4-2.5} to prove the desired $s$-independent estimate.
    Indeed, using the invariance of the measure  integrating the above symmetric kernel estimate, with $K^{\sharp}_s(z) = |z|^{-N-s} \min \{ 1, |z|^{-1} \}$, gives
    \begin{equation}\label{eqn:per-symmetrization-bound}
    \textup{Per}_s (\Phi^v_t(E); B_4) + \textup{Per}_s ( \Phi^v_{-t}(E);B_4) - 2 \, \textup{Per}_s(E;B_4) \leq C(N, \bar{s}) \, t^2 \textup{Per}_{K^{\sharp}_s} (E;B_4).
    \end{equation}
    Here, we used the invariance of the measure and the relative interaction domain, due to $\det D \Phi^v_t = 1$ and $\Phi^v_t(B_4) = B_4$.
    We note that $\int_{\bR^N} K^{\sharp}_s(z) \min \{1, |z| \} \, dz < \infty$, so all the considerations therein are applicable.
    Moreover, for $|z| \leq 2$, the kernel $K_s(z)$ satisfies the dimensional lower bound $K_s(z) \geq 2^{-N-s} \geq 2^{-2N}$.
    We may therefore apply Lemma~\ref{lemma:stable-under-flows} with $\kappa = 2^{-2N}$ and $\Lambda = C(N, \bar{s}) \, \textup{Per}_{K^{\sharp}_s}(E;B_4)$ to obtain, for $C_N$ a dimensional constant, 
    \[
    \textup{Per}(E;B_1) \leq C_N ( 1 + \textup{Per}_{K^{\sharp}_s} (E;B_4)^{\frac{1}{2}} ).
    \]
    Next, we claim that $E$ has finite perimeter in the ball $B_6$.
    This property follows from~\cite{cabre-cinti-serra}*{Theorem 1.1}, as we only need any finite bound; for completeness, we present a self-contained argument.
    First, observe that $K^{\sharp}_s \leq K_s$, so $\textup{Per}_{K^{\sharp}_s}(E;B) \leq \textup{Per}_s(E;B) < \infty$ for every ball $B \Subset B_8$.
    We cover $\bar{B}_6$ by a finite collection of smaller balls $B_{\theta}(x_{\alpha})$, of fixed sufficiently small radius $\theta \in (0, \frac{1}{50})$, such that $B_{4 \theta} (x_{\alpha}) \Subset B_8$, with $x_{\alpha} \in \bar{B}_6$, and $B_1(x_{\alpha}) \subset B_7 \subset B_8$,
    We consider the rescaled set $\tilde{E}_{\alpha} := \theta^{-1}(E - x_{\alpha})$, which is stable under rearrangements in the ball $B_4$ and satisfies
    \[
    \textup{Per}_{K^{\sharp}_s}(\tilde{E}_{\alpha}; B_4) \leq \textup{Per}_s ( \tilde{E}_{\alpha} ; B_4) = \theta^{s-N} \textup{Per}_s ( E; B_{4 \theta}(x_{\alpha})) < \infty.
    \]
    Thus, our earlier argument applies to the set $\tilde{E}_{\alpha}$, producing a bound as in~\eqref{eqn:per-symmetrization-bound},
    uniformly in $v \in \bS^{N-1}$.
    Using Lemma~\ref{lemma:stable-under-flows} again, with $\kappa = 2^{-2N}$ and $\Lambda_{\alpha} = C(N,\bar{s}) \textup{Per}_{K_s^{\sharp}}(\tilde{E}_{\alpha};B_4) < \infty$ independent of $v$, we see that $\tilde{E}_{\alpha}$ has finite perimeter in $B_1$.
    Rescaling the classical perimeter, we obtain
    \[
    \textup{Per}(E; B_{\theta}(x_{\alpha})) = \theta^{N-1} \textup{Per} ( \tilde{E}_{\alpha}; B_1) < \infty, \qquad \text{for every } \; x_{\alpha} \in \bar{B}_6.
    \]
    After refining the above cover, we can arrange $\bar{B}_6 \subset \bigcup_{\alpha=1}^M B_{\theta/2}(x_{\alpha})$, with $\mathbf{1}_E \in \textup{BV}(B_{\theta}(x_{\alpha}))$ for each $\alpha$.
    Let $\{ \zeta_{\alpha} \}$ be a smooth partition of unity subordinate to this cover, with $\textup{spt} \, \zeta_{\alpha} \Subset B_{\theta}(x_{\alpha})$ and $\sum_{\alpha=1}^M \zeta_{\alpha}=1$ on a neighborhood of $\bar{B}_6$.
    Then, every $\varphi \in C_c^1(B_6; \mathbb{R}^N)$ with $|\varphi| \leq 1$ satisfies
    \[
    \textup{div} \, \varphi = \sum_{\alpha=1}^M \textup{div}(\zeta_{\alpha} \varphi), \qquad \Bigl| \int_{B_6} \mathbf{1}_E \textup{div} \, \varphi \, dx \Bigr| \leq \sum_{\alpha=1}^M \Bigl| \int \mathbf{1}_E \textup{div}(\zeta_{\alpha} \varphi) \, dx \Bigr| \leq \sum_{\alpha=1}^M |D \mathbf{1}_E| (B_{\theta}(x_{\alpha}) ). 
    \]
    Taking the supremum over all $\varphi \in C_c^1(B_6 ; \mathbb{R}^N)$, we conclude that $\mathbf{1}_E \in \textup{BV}(B_6)$ due to
    \[
    |D \mathbf{1}_E|(B_6) \leq \sum_{\alpha=1}^M |D \mathbf{1}_E| ( B_{\theta} (x_{\alpha})) < \infty, \qquad \implies \qquad \textup{Per}(E;B_6) < \infty
    \]
    as claimed.
    To bound the modified interaction by the classical perimeter, we claim that
    \[
    P_{K^{\sharp}_s}(E;B_4) \leq C(N, \bar{s}) (1 + \textup{Per}(E; B_6)).
    \]
    Indeed, for $|z|<1$, we can apply Lemma~\ref{lemma:bv-translation}, since $E$ has finite perimeter in $B_6$, to bound
    \[
    \int_{B_5} |\mathbf{1}_E(x+z) - \mathbf{1}_E(x)| \, dx \leq |z| \, \textup{Per}(E;B_6)
    \]
    with $U = B_5$ and $V = B_6$, since $B_5 + [0,z] \subseteq B_6$. 
    Moreover, setting $y = x+z$ in the integral expression $\iint_{\hat{\cA}} |\mathbf{1}_E(x) - \mathbf{1}_E(y)|^2 K^{\sharp}_s(x-y) \, dx \, dy$ and expanding $|\mathbf{1}_E(x) - \mathbf{1}_E(y)|^2 = |\mathbf{1}_E(x) - \mathbf{1}_E(y)|$ for characteristic functions, we obtain
    \[
    P_{K^{\sharp}_s}(E;B_4) = \frac{1}{2} \int_{\bR^N} K^{\sharp}_s(z) \int_{\bR^N} \mathbf{1}_{ \{ x \in B_4 \} \cup \{ x+z \in B_4 \} } |\mathbf{1}_E(x+z) - \mathbf{1}_E(x)| \, dx \, dz.
    \]
    For $|z|<1$, the condition $x \in B_4$ or $x+z \in B_4$ implies that $x \in B_5$.
    Thus,
    \[
    \int_{\bR^N} \mathbf{1}_{ \{ x \in B_4 \} \cup \{ x+z \in B_4 \}  } |\mathbf{1}_E(x+z) - \mathbf{1}_E(x)| \, dx \leq \int_{B_5} |\mathbf{1}_E(x+z) - \mathbf{1}_E(x)| \, dx \leq |z| \, \textup{Per}(E;B_6)
    \]
    by applying Lemma~\ref{lemma:bv-translation} again.
    We now observe that
    \begin{align*}
        \int_{|z|<1} |z| \, K^{\sharp}_s(z) \, dz &= |\bS^{N-1}| \int_0^1 r^{N-1} \cdot r \cdot  r^{-(N+s)} \, dr = \frac{1}{1-s} |\bS^{N-1}| \leq \frac{1}{1-\bar{s}} |\bS^{N-1}|, \\
        \int_{|z|>1} K^{\sharp}_s(z) \, dz &= |\bS^{N-1}| \int_1^{\infty} r^{N-1} \cdot r^{-N-s-1} \, dr = \frac{1}{1+s} |\bS^{N-1}| \leq |\bS^{N-1}|.
    \end{align*}
    We therefore obtain, in the regions $|z|<1$ and for $|z| \geq 1$,
    \begin{align*}
        C \, \textup{Per}(E;B_6) &\int_{|z|<1} |z| \cdot |z|^{-(N+s)} \, dz \leq \tfrac{C}{1- \bar{s}} \textup{Per}(E;B_6) \quad  & \text{for } \; |z| &< 1, \\
         |\mathbf{1}_E(x+z) - \mathbf{1}_E(x)| &\leq 1, \qquad | \{ x \in B_4 \} \cup \{ x+z \in B_4 \} | \leq 2 |B_4| \quad &\text{for } \; |z| &\geq 1.
    \end{align*}
    Thus, the far-field contribution is bounded by $C \int_{|z| \geq1}|z|^{-N-1-s} \, dz \leq C$.
    Hence, we conclude that
    \[
    \textup{Per}(E;B_1) \leq C ( 1 + \textup{Per}(E;B_6)^{\frac{1}{2}} ).
    \]
    Applying this result to the rescaled domain $E_{z,r} = r^{-1}(E-z)$, we therefore obtain
    \begin{align*}
        & r^{1-N} \textup{Per}(E;B_r(z)) \leq C + C [ r^{1-N} \textup{Per}(E; B_{6r}(z))]^{\frac{1}{2}} , \\
        & \implies \rho^{1-N} \textup{Per}(E; B_{\rho/6}(z)) \leq C + C [ \rho^{1-N} \textup{Per}(E;B_{\rho}(z))]^{\frac{1}{2}} \leq \delta \rho^{1-N} \textup{Per}(E;B_\rho(z)) + C_{\delta}.
    \end{align*}
    In the second step, we set $\rho = 6r$, absorbed constants, and then applied Young's inequality to produce a constant $C_{\delta}$ depending only on $N,\delta, \bar{s}$. 
    We may now employ Simon's covering~\cite{simon-covering-argument}, presented in general form in~\cite{cinti-serra-valdinoci}*{Lemma 3.1}, to see that the above inequality implies the desired perimeter estimate $\textup{Per}(E;B_R(x_0)) \leq C ( N, \bar{s}) R^{N-1}$, where $C$ depends only on $\bar{s}$ and on the dimension $N$, through the covering number.
    This completes the proof.
\end{proof}

Proposition~\ref{prop:bound-s} implies that stable $s$-minimal sets satisfy a uniform perimeter estimate as $s \downarrow 0$.
This situation differs sharply from the regime $s \uparrow 1$, where the nonlocal interaction produces a term $\frac{1}{1-s}$ and leads to a different behavior from classical minimal surfaces.
\begin{theorem}\label{thm:compactness-thm}
Consider a sequence of $s_j \downarrow 0$ and cones $E_j = C(\Omega_j) \subset \mathbb{R}^N$, with $\Omega_j \subset \bS^{N-1}$ and $\varnothing \neq \Sigma_j = \partial \Omega_j \subset \bS^{N-1}$.
Suppose that each $\Sigma_j$ is a closed embedded $C^2$ hypersurface and each $E_j$ is stationary and stable for the $s_j$-perimeter.
Then, after passing to a subsequence, there exist a Caccioppoli set $\Omega \subset \bS^{N-1}$ with $|\Omega| = \frac{1}{2} |\bS^{N-1}|$, a finite positive Radon measure $\ell$ on $\bS^{N-1}$, and an $\bR^N$-valued measure $\mathbf{q}$ on $\bS^{N-1}$ with the following convergence properties.
\begin{enumerate}[(i)]
    \item We have $\mathbf{1}_{\Omega_j} \to \mathbf{1}_{\Omega}$ in strict $\textup{BV}(\bS^{N-1})$ and $\partial[\![ \Omega_j ]\!] \to \partial [\![ \Omega]\!]$ in the flat topology.
    \item Write $\mu_j := \cH^{N-2} \mres \Sigma_j$ and let $\nu_j$ denote the outward spherical conormal to $\Sigma_j$, with $\nu_{\Omega}$ the outward conormal to $\partial^* \Omega$.
    Then, we have the convergence of measures $\mu_j \xrightharpoonup{*} |D \mathbf{1}_{\Omega}|$ and
    \[
    (x, \nu_j(x))_{\#} \mu_j \xrightharpoonup{*} (x, \nu_{\Omega}(x))_{\#} |D \mathbf{1}_{\Omega}|, \qquad D \mathbf{1}_{\Omega_j} = - \nu_j \mu_j \xrightharpoonup{*} D \mathbf{1}_{\Omega} = -\nu_{\Omega} \, |D \mathbf{1}_{\Omega}|
    \]
    where the $(x, \nu_j(x))_{\#} \mu_j$ are measures on the normal bundle of $\bS^{N-1}$.
    Moreover, it holds that
        \[
\iint_{\partial^* \Omega \times \partial^* \Omega} B_0 ( \langle x, y \rangle ) \, |\nu_{\Omega}(x) - \nu_{\Omega}(y)|^2 \, d |D \mathbf{1}_{\Omega}(x)| \, d |D \mathbf{1}_{\Omega} |(y) < \infty.
\]  
    \item As in Lemma~\ref{lemma:ps-identity}, we define
    \[
    p_j(x) := \int_{\Sigma_j} B_{s_j} ( \langle x, y \rangle )(1 - \langle \nu_j(x), \nu_j(y) \rangle) \, d\mu_j(y).
    \]
    Then, we have $p_j \mu_j \xrightharpoonup{*} \ell$ and $p_j \nu_j \mu_j \xrightharpoonup{*} \mathbf{q}$.
    Additionally, we have
    \[
    \iint_{\partial^* \Omega \times \partial^* \Omega} f(x) B_0(\la x,y \rg)\,  (1 - \la \nu(x), \nu(y) \rg) \, d |D \mathbf{1}_{\Omega}|(x) \, d |D \mathbf{1}_{\Omega}|(y) \leq \int_{\bS^{N-1}} f \, d \ell
    \]
    for every non-negative continuous function $f$.  
  \item We have $|\mathbf{q}| \leq \ell$, and for every smooth function $\Phi: \bS^{N-1} \to \bR^N$, it holds that
    \[
    \iint_{\partial^* \Omega \times \partial^* \Omega} B_0( \langle x,y \rangle) \, \langle \Phi(x) - \Phi(y) , \nu_{\Omega}(x) - \nu_{\Omega}(y) \rangle \, d  |D \mathbf{1}_{\Omega}(x)| \, d |D \mathbf{1}_{\Omega}(y)| =2 \int_{\bS^{N-1}} \langle \Phi , d \mathbf{q} \rangle.
    \]
    In addition, for every smooth function $\Psi : \bS^{N-1} \to \bigwedge^2 \bR^N$, we have
    \[
    \iint_{\partial^* \Omega \times \partial^* \Omega} B_0 ( \langle x, y \rangle ) \, \langle \Psi(x) - \Psi(y) , \nu_{\Omega}(x) \wedge \nu_{\Omega}(y) \rangle \,d  |D \mathbf{1}_{\Omega}(x)| \, d |D \mathbf{1}_{\Omega}(y)| = 0.
    \]
    \item The limiting stability form $q_0[g]$ satisfies $q_0[g] \geq 0$ for every smooth real-valued function $g \in C^{\infty}(\bS^{N-1})$.
    That is, for $\Sigma = \partial^* \Omega$, we have
    \begin{align*}
    \int_{\bS^{N-1}} g^2 \, d \ell &\leq \frac{1}{2} \iint_{\partial^* \Omega \times \partial^* \Omega} A_0( \langle x,y \rangle) (g(x) - g(y))^2 \, d  |D \mathbf{1}_{\Omega}(x)| \, d |D \mathbf{1}_{\Omega}(y)| \\
    & \quad + \iint_{\partial^* \Omega \times \partial^* \Omega} W_0( \la x,y \rg) g(x)^2 \, d  |D \mathbf{1}_{\Omega}(x)| \, d |D \mathbf{1}_{\Omega}(y)|.
    \end{align*}
    Moreover, each term in this expression is finite.
\end{enumerate}
\end{theorem}
\begin{proof}
    Working in polar coordinates on each conical domain $E_j$, we write $d \cH^{N-1}_{\partial E_j} = r^{N-2} \, dr \, d \mu_j$, so
    \[
    \cH^{N-1} ( \partial E_j \cap \{ a < |X| < b \}) = \tfrac{1}{N-1} ( b^{N-1}- a^{N-1}) \, \mu_j(\Sigma_j).
    \]
    We can cover a fixed annulus $\{ \frac{3}{4} < |X| < \frac{5}{4} \}$ by finitely many balls and apply Proposition~\ref{prop:bound-s} for each $E_j$ to obtain uniform bounds on the measures $\mu_j = \cH^{N-2} \mres \Sigma_j$ as
    \begin{equation}\label{eqn:mu-j-Sigma-j}
    \mu_j(\bS^{N-1}) = \cH^{N-2}(\Sigma_j) \leq C(N,\bar{s}), \qquad \mu_j(B_r(x)) \leq C(N,\bar{s}) r^{N-2}.
    \end{equation}
    For each fixed $x \in \Sigma_j$ and small $r \in (0,r_0)$, the radial slab 
    \[
    \{ \rho y : y \in \Sigma_j \cap B_{cr}(x) \, : \; 1 - cr < \rho < 1 + cr \}
    \]
    is contained in $B_{Cr}(x)$, and has $(N-1)$-measure bounded below by $cr \, \mu_j(\Sigma_j \cap B_{cr}(x))$.
    Applying Proposition~\ref{prop:bound-s}, we obtain $cr \, \mu_j(\Sigma_j \cap B_{cr}(x)) \leq C r^{N-1}$, and hence $\mu_j(\Sigma_j \cap B_r(x)) \leq \tilde{C} r^{N-2}$ after changing the constants.
    Also, after passing to a tail of the sequence, we can assume that $s_j \leq \frac{1}{4}$ in what follows.
    We now prove the properties $(i)$-$(iv)$ in steps.
    
\smallskip \noindent \textbf{Step 1: Property $(i)$.}
We record the uniform bound $\mu_j(B_r(x)) \leq C(N,\bar{s}) r^{N-2}$ by~\eqref{eqn:mu-j-Sigma-j}.
By Lemma~\ref{lemma:kernel-k-s-tau,z-bounds}, the angular kernels satisfy $B_s(\la x,y \rg) \asymp |x-y|^{-(N-1+s)}$, while Lemma~\ref{lemma:ws-useful-bound} shows that 
\[
W_s( \la x,y \rg) \leq C_N ( 1 + |\log |x-y||) \, |x-y|^{-(N-3+s)}.
\]
Thus, we can apply Lemma~\ref{lemma:potential-upper-growth} with $(m,\alpha,\beta_0,q) = (N-2, N-3+s,1-\bar{s},1)$ and $\bar{s} \leq \frac{1}{4}$ to bound
\begin{equation}\label{eqn:Ws-uniform-bound}
    \sup_j\sup_{x\in\bS^{N-1}} \int_{\Sigma_j\cap B_r(x)} W_{s_j}(\la x,y\rg)\,d\mu_j(y) \leq C r^{\frac34}(1+|\log r|) \leq C r^{\frac14}, \qquad \text{for} \quad 0<r\leq1.
\end{equation}
Testing the stability inequality with the constant function $g \equiv 1$, so $q_{s_j}[1] \geq 0$, we obtain
    \begin{align*}
    \iint_{\Sigma_j \times \Sigma_j} B_{s_j}(\la x,y \rg) \, |\nu_j(x) - \nu_j(y)|^2 \, d\mu_j(x) \, d\mu_j(y) &\leq 2 \, \iint_{\Sigma_j \times \Sigma_j} W_{s_j}(\la x,y \rg) \,d \mu_j(x) \, d \mu_j(y) \\
    &\leq 2 \, \cH^{N-2}(\Sigma_j) \sup_{x \in \bS^{N-1}}\int_{\Sigma_j} W_{s_j}(\la x, y \rg) \, d\mu_j(y),
    \end{align*}
so~\eqref{eqn:Ws-uniform-bound} implies that $\iint B_s |\nu_j(x) - \nu_j(y)|^2 \, d\mu_j \, d\mu_j \leq C \cH^{N-2}(\Sigma_j)$.
Then, Proposition~\ref{prop:angular-} yields
    \begin{equation}\label{eqn:omega-j-inequality}
        \bigl| 2|\Omega_j| - |\bS^{N-1}| \bigr| \leq C(N, \bar{s}) \Bigl( \frac{\Lambda}{\cH^{N-2}(\Sigma_j)} \Bigr)^{\frac{1}{2}} \cH^{N-2}(\Sigma_j)^{\frac{1}{2}} s_j \leq C(N, \bar{s}) \, s_j.
    \end{equation}
    because $\Lambda = C(N,\bar{s})$ here, and $\cH^{N-2}(\Sigma_j) \leq C(N,\bar{s})$ uniformly in $s_j < \frac{1}{4}$ due to~\eqref{eqn:mu-j-Sigma-j}.
    
    Moreover, plugging $\nu_j \in C^1(\Sigma_j)$ into the weak formulation~\eqref{eqn:double-integral-B-eta} of Lemma~\ref{lemma:ps-identity}, we obtain
    \begin{equation}\label{eqn:uniform-pj-bound}
    \int_{\Sigma_j} p_j \, d\mu_j = \frac{1}{2} \iint_{\Sigma_j \times \Sigma_j} B_{s_j} ( \langle x,y \rangle) \, |\nu_j(x) - \nu_j(y)|^2 \, d \mu_j(x) \, d \mu_j(y) \leq C.
    \end{equation}
    We denote by $\cN$ the spherical normal bundle
    \[ 
    \mathcal{N} := \{ (x, v ) \in \bS^{N-1} \times \bS^{N-1} : \la x, v \rg = 0 \}
    \]
    so the $\eta_j := (x, \nu_j(x))_{\#} \mu_j$ define measures on $\cN$ with $(\pi_x)_{\#} \eta_j = \mu_j$.
    We deduce from the above considerations that the measures $\ell_j := p_j \mu_j$ and $\mathbf{q}_j := p_j \nu_j \, \mu_j$ satisfy uniform bounds $\ell_j ( \bS^{N-1}) \leq C$, $|\mathbf{q}_j|(\bS^{N-1}) \leq C$, and $\eta_j(\cN) \leq C$.
    Therefore, we can extract subsequential limit measures
    \[
    \mu_j \xrightharpoonup{*} \mu, \qquad \eta_j \xrightharpoonup{*} \eta, \qquad \ell_j \xrightharpoonup{*} \ell, \qquad \mathbf{q}_j \xrightharpoonup{*} \mathbf{q}, \qquad \textup{spt} \, \eta \subseteq \cN,
    \]
    which satisfy $(\pi_x)_{\#} \eta = \mu$.
    Moreover, the bound~\eqref{eqn:mu-j-Sigma-j} implies that $\mu(B_r(x)) \leq C(N,\bar{s}) r^{N-2}$ for every $x \in \bS^{N-1}$ and $r>0$: for every $j$, either $B_r(x) \cap \Sigma_j = \varnothing$, in which case $\mu_j(B_r(x)) = 0$, or we can find some $z_j \in B_r(x) \cap \Sigma_j$.
    Then, $B_r(x) \subset B_{2r}(z_j)$ and a non-negative cutoff function gives
    \[
    \mu_j(B_r(x)) \leq \mu_j(B_{2r}(z_j)) \leq C r^{N-2} \implies \mu(B_r(x)) \leq \limsup_{j \to \infty} \mu_j(B_{2r}(x)) \leq C r^{N-2}
    \]
    for a constant $C = C(N,\bar{s})$.
    This property also shows that $\mu (\{ x\}) = 0$ for every $x$, and hence
    \begin{equation}\label{eqn:mu-otimes-mu-measure}
    (\mu \otimes \mu) ( \{ (x,y) : x=y\}) = 0, \qquad (\eta \otimes \eta) ( \{ ( (x,v) , (y,w) ) : x=y \} ) = 0.
    \end{equation}
    The uniform bounds~\eqref{eqn:mu-j-Sigma-j} and~\eqref{eqn:omega-j-inequality} imply that the sets $\Omega_j$ have uniformly bounded measure and uniformly bounded perimeter, so the BV compactness theorem~\cite{maggi}*{Proposition 12.15 and Theorem 12.26} allows us to extract a Caccioppoli set $\Omega \subset \bS^{N-1}$ and a subsequence of $s_j \downarrow 0$ with $\mathbf{1}_{\Omega_j} \to \mathbf{1}_{\Omega}$ in $L^1(\bS^{N-1})$.
    Because $|\mathbf{1}_{\Omega_j} - \mathbf{1}_{\Omega}| = \mathbf{1}_{\Omega_j \triangle \Omega}$ a.e., we obtain
    \begin{align*}
        \bigl| |\Omega_j| - |\Omega| \bigr| &= \left| \int_{\bS^{N-1}} (\mathbf{1}_{\Omega_j} - \mathbf{1}_{\Omega}) \, d \cH^{N-1} \right| \leq \int_{\bS^{N-1}} |\mathbf{1}_{\Omega_j} - \mathbf{1}_{\Omega}| \, d \cH^{N-1} = |\Omega_j \triangle \Omega| \\
        &= \| \mathbf{1}_{\Omega_j} - \mathbf{1}_{\Omega}\|_{L^1(\bS^{N-1})} \to 0.
    \end{align*}
    Combining this property with~\eqref{eqn:omega-j-inequality} gives $|\Omega| = \frac{1}{2} |\bS^{N-1}|$.
    Finally, the current $T_j := [\![ \Omega_j]\!] - [\![ \Omega]\!]$ is represented by the $\bZ$-valued multiplicity $\mathbf{1}_{\Omega_j} - \mathbf{1}_{\Omega}$, with mass $\mathbf{M}(T_j) = |\Omega_j \triangle \Omega| \to 0$ and $\partial [\![ \Omega_j]\!] - \partial [\![\Omega]\!] = \partial T_j$.
    The flat norm of $\partial T_j$ is bounded by $\mathbf{M}(T_j) \to 0$, so $\partial [\![\Omega_j]\!] \to \partial [\![\Omega]\!]$ in the flat topology. 
    This completes part $(i)$ except for the strict BV convergence, which we prove in Step 2.

    \smallskip \noindent \textbf{Step 2: Properties $(ii)$ and $(iii)$.}
    To study the limit measure $\ell$ in terms of $\eta$, let $f \geq 0$ be a continuous function on $\bS^{N-1}$ and let $\chi_{\delta}(x,y)$ be a smooth cutoff function satisfying
    \begin{equation}\label{eqn:chi-delta-cutoffs}
    0 \leq \chi_{\delta} \leq 1, \qquad \chi_{\delta} = 0 \quad \text{if } \; |x-y| \leq \delta, \qquad \chi_{\delta} = 1 \quad \text{if } \; |x-y| \geq 2 \delta.
    \end{equation}
    We can arrange a continuous family $\{ \chi_{\delta} \}_{0<\delta \leq 1}$ of such cutoff functions that is ordered decreasingly pointwise in $\delta>0$, with $\chi_{\delta}(x,y) \uparrow 1$ as $\delta \downarrow 0$ for every $x \neq y$.
    For fixed $\delta>0$, the kernels converge uniformly to $B_0$ on $\textup{spt} \, \chi_{\delta}$, as showed in Proposition~\ref{prop:angular-}.
    Moreover, $\eta_j \otimes \eta_j \xrightharpoonup{*} \eta \otimes \eta$.
    It follows that
    \begin{align*}
    &  \iint_{\cN \times \cN} \chi_{\delta}(x,y) f(x) \, B_0 ( \la x, y \rg ) (1 - \la v,w\rg) \, d \eta(x,v) \, d \eta(y,w) \\
    &= \lim_{j \to \infty} \iint_{\cN \times \cN} \chi_{\delta}(x,y) f(x) \, B_{s_j}(\la x,y \rg) (1 - \la v,w \rg) \, d \eta_{j}(x,v) \, d \eta_j(y,w) \\
    &\leq \lim_{j \to \infty} \int_{\bS^{N-1}} f \, d \ell_j = \int_{\bS^{N-1}} f \, d \ell.
    \end{align*}
    In the third step, we used the expression of the measure $\ell_j = p_j \, \mu_j$ and the equality of Lemma~\ref{lemma:ps-identity}.
    Because the functions $\{ \chi_{\delta} \}_{\delta>0}$ are ordered decreasingly in $\delta$, with $\chi_{\delta} \uparrow 1$ on $\{ x \neq y \}$, we can apply the monotone convergence theorem together with~\eqref{eqn:mu-otimes-mu-measure} to obtain
    \[
    \iint_{\cN \times \cN} f(x) B_0(\la x,y \rg) (1 - \la v,w \rg) \, d \eta(x,v) \, d \eta(y,w) \leq \int_{\bS^{N-1}} f \, d \ell.
    \]
    In particular, taking $f \equiv 1$ and using $|v-w|^2 = 2( 1 - \la v, w \rg)$ shows that
    \begin{equation}\label{eqn:B0-integral-finite}
    \iint_{\cN \times \cN} B_0(\la x,y \rg) \, |v-w|^2 \, d \eta(x,v) \, d \eta(y,w) \leq 2 \, \ell(\bS^{N-1}) < \infty.
    \end{equation}
    This establishes $(ii)$.
    Next, we disintegrate $\eta$ with respect to $\mu$ and write $d \eta(x,v) = d \eta_x(v) \,  d\mu(x)$, where $\eta_x$ is a probability measure for $\mu$-a.e.
    Let $b(x) := \int_{\bS^{N-1}} v \, d \eta_x(v)$ be the barycenter function.
For every smooth tangential vector field $V$ on $\bS^{N-1}$, the convergence $\eta_j= (x, \nu_j(x))_{\#} \mu_j \xrightharpoonup{*} \eta$ gives
\begin{align*}
    \int_{\bS^{N-1}} \la V(x), b(x) \rg \, d \mu(x) &= \int_{\cN} \la V(x), v \rg \, d\eta(x,v) = \lim_{j \to \infty} \int_{\Sigma_j} \la V , \nu_j \rg \,d \mu_j \\
    &= \lim_{j \to \infty} \int_{\Omega_j} \textup{div}_{\bS^{N-1}} V \, d S = \int_{\Omega} \textup{div}_{\bS^{N-1}} V \, dS
\end{align*}
by the divergence theorem and $\mathbf{1}_{\Omega_j} \to \mathbf{1}_{\Omega}$ in $L^1(\bS^{N-1})$.
Thus, the distributional derivative satisfies 
\begin{equation}\label{eqn:DS1bmu}
- D_{\bS^{N-1}} \mathbf{1}_{\Omega} = b \, \mu, \qquad \implies \qquad |D_{\bS^{N-1}} \mathbf{1}_{\Omega}| = |b| \, \mu \leq \mu
\end{equation}
due to $|b(x)| \leq \int |v| \, d \eta_x(v) = 1$.
Moreover, using $\mu = (\pi_x)_{\#} \eta$, we find
\begin{align*}
    \iint_{\cN \times \cN} B_0 (\la x,y \rg) \, |v-w|^2 \, d \eta(x,v) \, d \eta(y,w) = 2 \iint_{\bS^{N-1} \times \bS^{N-1}} B_0(\la x,y \rg) \, (1 - \la b(x), b(y) \rg) \, d\mu(x) \, d\mu(y)
\end{align*}
Define the non-negative Radon measure $\lambda := (1 - |b|) \mu$ on $\bS^{N-1}$.
The computation of Proposition~\ref{prop:angular-} and~\eqref{eqn:B0-N(z)} shows that $B_0(z) \geq c_N (1-z)^{- \frac{N-1}{2}}$, so $B_0(\la x,y\rg) \geq c_N |x-y|^{-(N-1)}$ due to $z = \la x,y \rg$ and $1-z = \frac{|x-y|^2}{2}$ for $x,y \in \bS^{N-1}$.
Bounding $(1-r_1)(1-r_2) \leq 1 - r_1 r_2$ for $r_1, r_2 \in [0,1]$, we obtain
\begin{align*}
& 1 - \la b(x), b(y) \rg \geq 1 - |b(x)| \cdot |b(y)| \geq (1 - |b(x)|) \, (1 - |b(y)|) \\
& \implies c_N \iint_{\bS^{N-1} \times \bS^{N-1}} |x-y|^{-(N-1)} \, d \lambda(x) \, d \lambda(y) \leq \int_{\bS^{N-1} \times \bS^{N-1}} B_0( \la x,y \rg) \, d \lambda(x) \, d \lambda(y) < \infty
\end{align*}
due to~\eqref{eqn:B0-integral-finite} and the above integral inequality.
By Lemma~\ref{lemma:radon-Sm}, the finiteness of this integral forces $\lambda=0$, so $|b(x)|=1$ $\mu$-a.e. and~\eqref{eqn:DS1bmu} shows that $|D_{\bS^{N-1}} \mathbf{1}_{\Omega}|=\mu$.
Since $\eta_x$ is a probability measure supported on $\bS^{N-1}$, we can use $b(x) = \int_{\bS^{N-1}} v \, d \eta_x(v)$ and expand
\begin{align*}
    \int |v-b(x)|^2 \, d \eta_x(v) = \int |v|^2 \, d \eta_x(v) - 2 \Bigl\langle b(x), \int v \, d \eta_x(v) \Bigr\rangle + |b(x)|^2 = 1 - |b(x)|^2 = 0. 
\end{align*}
Thus, $\eta_x = \delta_{b(x)}$ holds $\mu$-a.e., and the identity $- D_{\bS^{N-1}} \mathbf{1}_{\Omega} = b \, \mu$ implies that $b = \nu_{\Omega}$, $|D \mathbf{1}_{\Omega}|$-a.e.
Finally, for every continuous every field $\Phi$,
\[
\int \la \Phi(x), \nu_j \rg \, d \mu_j = \int_{\cN} \la \Phi(x), v \rg \, d \eta_j \to \int_{\cN} \la \Phi(x), v \rg \, d \eta = \int \la \Phi, \nu_{\Omega} \rg \, d\mu = - \int \la \Phi, d D \mathbf{1}_{\Omega} \rg
\]
as integrals over $\bS^{N-1}$.
Thus, $D \mathbf{1}_{\Omega_j} = - \nu_j \mu_j \xrightharpoonup{*} D \mathbf{1}_{\Omega}$, and $\eta(x) = (x, \nu_{\Omega}(x))_{\#} |D \mathbf{1}_{\Omega}|$ by the above argument and Reshetnyak's continuity theorem~\cite{reshetnyak}.
Then, substituting into~\eqref{eqn:B0-integral-finite} gives
\[
\iint_{\partial^* \Omega \times \partial^* \Omega} B_0(\la x,y \rg) \, |\nu_{\Omega}(x) - \nu_{\Omega}(y)|^2 \, d |D \mathbf{1}_{\Omega}(x)| \, d |D \mathbf{1}_{\Omega}(y)| < \infty
\]
so the integral is finite.
Since the sphere is compact, the convergence $\mu_j \xrightharpoonup{*} |D \mathbf{1}_{\Omega}|$ also implies the convergence of total masses, hence $\Omega_j \to \Omega$ in strict BV.
This completes the proof of $(ii)$.

\smallskip \noindent
\textbf{Step 3: Property $(iv)$.}
For every smooth function $\Phi: \bS^{N-1} \to \bR^N$ and each $j$, we have
\[
\frac{1}{2} \iint_{\Sigma_j \times \Sigma_j} B_{s_j} ( \la x,y \rg) \la \Phi(x) - \Phi(y) , \nu_j(x) - \nu_j(y) \rg \, d \mu_j(x) \, d \mu_j(y) = \int_{\bS^{N-1}} \la \Phi, d \mathbf{q}_j \rg
\]
due to Lemma~\ref{lemma:ps-identity}.
We first control the integrals uniformly near the diagonal.
Recall that $B_{s_j}(\la x,y \rg) \, |x-y|^2 \leq C(N, \bar{s}) \, |x-y|^{-(N-3+s_j)}$, so we use the uniform bound $\mu_j(B_r(x)) \leq C r^{N-2}$ and apply Lemma~\ref{lemma:potential-upper-growth} with $(m,\alpha,\beta_0,q) = (N-2,N-3+s_j, 1-\bar{s},0)$ to obtain
\begin{equation}\label{eqn:sup-j-sup-x}
    \sup_j \sup_{x \in \bS^{N-1}} \int_{\Sigma_j \cap B_r(x)} B_{s_j}( \la x,y\rg) \, |x-y|^2 \, d\mu_j(y) \leq C r^{1- \bar{s}}, \qquad 0 < r \leq 1.
\end{equation}
Moreover,~\eqref{eqn:uniform-pj-bound} implies that $\iint_{\Sigma_j \times \Sigma_j} B_{s_j} (\la x,y \rg) \, |\nu_j(x) - \nu_j(y)|^2 \, d\mu_j(x) \, d\mu_j(y) \leq C$.
Hence, the Cauchy-Schwarz inequality and the Lipschitz bound $|\Phi(x)-\Phi(y)|\leq \|D\Phi\|_{L^\infty}|x-y|$ yield
\begin{equation}\label{eqn:pass-to-send-r-to-zero}
\sup_j \left| \iint_{|x-y|<r} B_{s_j}(\la x,y \rg) \, \la \Phi(x) - \Phi(y), \nu_j(x) - \nu_j(y) \rg \, d \mu_j(x) \, d \mu_j(y) \right| \leq C(N, \bar{s}, \Phi) \, r^{\frac{1-\bar{s}}{2}}.
\end{equation}
We use the continuous family $\{ \chi_{\delta} \}_{0 < \delta \leq 1}$ satisfying~\eqref{eqn:chi-delta-cutoffs} from Step 2.
For fixed $\delta>0$, the kernels $\chi_{\delta}(x,y) B_{s_j}(\la x,y \rg)$ converge uniformly to $\chi_{\delta}(x,y) B_0( \la x,y \rg)$, while $\eta_j \otimes \eta_j \xrightharpoonup{*} \eta \otimes \eta$.
Since Step 2 gives $\eta = ( x, \nu_{\Omega}(x))_{\#} \mu$ for $\mu = |D \mathbf{1}_{\Omega}|$, we obtain
\begin{align*}
& \frac{1}{2} \iint \chi_{\delta}(x,y)B_{s_j}(\la x,y\rg) \la\Phi(x)-\Phi(y),v-w\rg \,d\eta_j(x,v)d\eta_j(y,w)\\
& \qquad \longrightarrow \frac{1}{2} \iint \chi_{\delta}(x,y)B_0(\la x,y\rg) \la\Phi(x)-\Phi(y),\nu_\Omega(x) - \nu_\Omega(y)\rg \,d\mu(x)d\mu(y).
\end{align*}
The limiting integral without the cutoff is absolutely convergent by Cauchy-Schwarz, because 
\[
\iint B_0(\la x,y \rg) \, |\nu_{\Omega}(x) - \nu_{\Omega}(y)|^2 \, d\mu(x) \, d\mu(y) < \infty, \qquad B_0( \la x,y \rg) |\Phi(x) - \Phi(y)|^2 \leq C(N,\bar{s},\Phi) \, |x-y|^{-(N-3)}
\]
produces a kernel that is integrable against $\mu \otimes \mu$, by Lemma~\ref{lemma:potential-upper-growth}, since $\mu(B_r(x)) \leq C r^{N-2}$.
Thus, if $I_j^\Phi$ and $I^\Phi$ denote the corresponding full integrals and $I_{j,\delta}^\Phi$, $I_{\delta}^\Phi$ their integrals with the cutoff $\chi_\delta$, then~\eqref{eqn:pass-to-send-r-to-zero} gives $\sup_j |I_j^\Phi-I_{j,\delta}^\Phi| \leq C(N,\bar{s},\Phi)\delta^{\frac{1-\bar{s}}{2}}$, while $I_{j,\delta}^\Phi \to I_\delta^\Phi$ for every fixed $\delta>0$ and $I_\delta^\Phi \to I^\Phi$ as $\delta\downarrow0$ by absolute convergence.
It follows that $I_j^\Phi \to I^\Phi$.

We now use the convergence of measures $\mathbf{q}_j \xrightharpoonup{*} \mathbf{q}$ and the property $\eta(x) = ( x, \nu_{\Omega}(x))_{\#} \mu$, with $\mu = |D \mathbf{1}_{\Omega}|$.
Combining the above properties and writing $\Sigma = \partial^* \Omega$, we obtain
\[
\frac{1}{2} \iint_{\Sigma \times \Sigma} B_0( \la x,y \rg) \, \la \Phi(x) - \Phi(y), \nu_{\Omega}(x) - \nu_{\Omega}(y) \rg \, d \mu(x) \, d \mu(y) = \int_{\bS^{N-1}} \la \Phi, d \mathbf{q} \rg.
\]
Recalling that $\mathbf{q}_j = p_j \nu_j \mu_j$, we find $|\mathbf{q}_j| =p_j \mu_j = \ell_j$, hence
\[
\left| \int_{\bS^{N-1}} \la \Phi, d \mathbf{q} \rg \right| = \lim_{j \to \infty} \left| \int_{\bS^{N-1}} \la \Phi, \nu_j \rg \, d \ell_j \right| \leq \lim_{j \to \infty} \int_{\bS^{N-1}} |\Phi| \, d \ell_j = \int_{\bS^{N-1}} |\Phi| \, d \ell
\]
for every continuous vector field $\Phi$.
Therefore, $|\mathbf{q}| \leq \ell$.

Finally, to derive the wedge product identity, we use the identity of Lemma~\ref{lemma:ps-identity} to write
\begin{align*}
    & \textup{PV} \int_{\Sigma_j} B_{s_j}(\la x,y \rg ) \, (\nu_j(x) - \nu_j(y)) \,d \mu_j(y) = p_j(x) \nu_j(x) \\
    & \implies \textup{PV} \int_{\Sigma_j} B_{s_j}(\la x,y \rg) \, \nu_j(x) \wedge \nu_j(y) \, d \mu_j(y) = 0
\end{align*}
after wedging with $\nu_j(x)$.
Testing with $\Psi \in C^{\infty}(\bS^{N-1} , \bigwedge^2 \bR^N)$ and symmetrizing gives
\[
\frac{1}{2} \iint_{\Sigma_j \times \Sigma_j} B_{s_j}(\la x,y \rg) \, \la \Psi(x) - \Psi(y), \nu_j(x) \wedge \nu_j(y) \rg \, d\mu_j(x) \, d \mu_j(y) = 0.
\]
Because $|\nu_j(x) \wedge \nu_j(y)| \leq |\nu_j(x) - \nu_j(y)|$, the near-diagonal contribution to the above integral from the region $|x-y|<r$ satisfies a bound as in~\eqref{eqn:pass-to-send-r-to-zero}, uniformly in $j$, by the same Cauchy-Schwarz estimate.
Thus, using the cutoffs $\chi_\delta$ and $\eta_j \otimes \eta_j \xrightharpoonup{*} \eta \otimes \eta$ exactly as above, we may first pass to the limit as $j \to \infty$ away from the diagonal and then send $\delta \downarrow 0$ to conclude that
\[
\iint_{\partial^*\Omega \times \partial^* \Omega} B_0 ( \la x,y \rg) \, \la \Psi(x) - \Psi(y) , \nu_{\Omega}(x) \wedge \nu_{\Omega}(y) \rg \, d \mu(x) \, d\mu(y) = 0.
\]
This completes the proof of the property $(iv)$.

\smallskip \noindent \textbf{Step 4: Property $(v)$.}
We finally pass to the limit in the stability inequality.
We again write $\mu = |D \mathbf{1}_{\Omega}|$ for brevity.
For $g \in C^{\infty}(\bS^{N-1};\bR)$ and each $j$, we define the expression
\[
\tilde{Q}_j[g] := \frac{1}{2} \iint_{\Sigma_j \times \Sigma_j} A_{s_j}(\la x,y \rg) \, [g(x) - g(y)]^2 \, d\mu_j(x) \, d\mu_j(y) + \iint_{\Sigma_j \times \Sigma_j} W_{s_j}(\la x,y \rg) \, g(x)^2 \, d \mu_j(x) \, d\mu_j(y).
\]
The inequality $q_s[g] \geq 0$ of Proposition~\ref{prop:qsg} amounts to $\tilde{Q}_j[g] \geq \int_{\bS^{N-1}} g^2 \, d \ell_j$ in our notation.
We claim that $\tilde{Q}_j[g] \to \tilde{Q}_0[g]$ as $s_j \downarrow 0$, where, as integrals over $\bS^{N-1} \times \bS^{N-1}$,
\[
\tilde{Q}_0[g] = \frac{1}{2} \iint A_0(\la x,y \rg) \, [g(x) - g(y)]^2 \, d\mu(x) \, d \mu(y) + \iint W_0(\la x,y \rg ) \, g(x)^2 \, d\mu(x) \, d \mu(y).
\]
Indeed, using $|g(x)-g(y)| \leq \|Dg\|_{L^\infty}|x-y|$ and the kernel bounds from Lemma~\ref{lemma:ws-useful-bound} and $A_s(\la x,y \rg) \leq C r^{-(N-1+s)}$, we can apply Lemma~\ref{lemma:potential-upper-growth} with $(m,\alpha,\beta_0,q) = (N-2,N-3+s_j, 1-\bar{s},1)$ to obtain
\[
\sup_j \iint_{|x-y|<r} A_{s_j}(\la x,y\rg)[g(x)-g(y)]^2\,d\mu_j(x)d\mu_j(y) + \sup_j\iint_{|x-y|<r} W_{s_j}(\la x,y\rg)g(x)^2\,d\mu_j(x)d\mu_j(y)
\to0
\]
as $r\downarrow0$.
For fixed $\delta>0$, after inserting the cutoff $\chi_\delta$ from~\eqref{eqn:chi-delta-cutoffs}, the kernels $A_{s_j}$ and $W_{s_j}$ converge uniformly to $A_0$ and $W_0$ away from the diagonal, while $\mu_j\otimes\mu_j\xrightharpoonup{*} \mu \otimes \mu$.
Hence, the cutoff forms converge as $j\to\infty$ as in Step 3.
The same potential estimate with $s=0$ shows that both terms in $\tilde Q_0[g]$ are finite and that their cutoff errors vanish as $\delta\downarrow0$.
Therefore, $\tilde{Q}_j[g] \to \tilde{Q}_0[g]$.
On the other hand, since $\ell_j \xrightharpoonup{*} \ell$, we also have $\int_{\bS^{N-1}} g^2 \, d \ell_j \to \int_{\bS^{N-1}} g^2 \, d \ell$.
Passing to the limit in the inequality $\tilde{Q}_j [g] \geq \int_{\bS^{N-1}} g^2 \, d \ell_j$, we therefore obtain
\[
\frac{1}{2} \iint A_0 (\la x,y \rg) \, [g(x) - g(y)]^2 \, d \mu(x) \, d\mu(y) + \iint W_0(\la x,y \rg) \, g(x)^2 \, d\mu(x) \, d \mu(y) \geq \int g^2 \, d \ell
\]
as integrals over $\bS^{N-1}$.
Since $\mu = |D \mathbf{1}_{\Omega}|$ and $\Omega$ is a Caccioppoli set, the integrals in the left-hand side can be taken over $\Sigma \times \Sigma$, where $\Sigma = \partial^*\Omega$.
This completes the proof of our Theorem.
\end{proof}

\begin{corollary}\label{cor:limit-Jacobi-inequality}
Consider a sequence of $s_j \downarrow 0$ and stable $s_j$-minimal cones $C(\Omega_j) \subset \bR^N$ with $C^2$ links, converging as in Theorem~\ref{thm:compactness-thm} to a Caccioppoli set $\Omega \subset \bS^{N-1}$ with measure-theoretic outward spherical conormal $\nu_{\Omega}$ and $\Sigma = \partial^* \Omega$.
Then, for every vector $a \in \bR^N$, it holds that
\begin{align*}
&\iint_{\Sigma \times \Sigma} B_0(\la x,y \rg) \, (1 - \la \nu_{\Omega}(x) , \nu_{\Omega}(y) \rg ) \, d |D \mathbf{1}_{\Omega}(x)| \, d |D \mathbf{1}_{\Omega}(y)| \\
&\leq \ell (\bS^{N-1}) \leq \iint_{\Sigma \times \Sigma} W_0( \la x,y \rg ) \, (1 - \la a, \nu_{\Omega}(x) \rg) (1 - \la a, \nu_{\Omega}(y) \rg) \, d |D \mathbf{1}_{\Omega}(x)| \, d |D \mathbf{1}_{\Omega}(y)|.
\end{align*}
\end{corollary}
\begin{proof}
Let us write $\mu = |D \mathbf{1}_{\Omega}|$ and $\nu_x = \nu_{\Omega}(x)$ for brevity. 
Using $|\nu|=1$ a.e., we have $|\nu_x - \nu_y|^2 = 2(1 - \la \nu_x , \nu_y \rg)$.
Thus, the inequality~\eqref{eqn:B0-integral-finite}, with $\eta = (x, \nu_x)_{\#} \mu$, proves the first bound.
For the second bound, set $h_j=\la a,\nu_j\rg$, which is a Jacobi field on the cone $C(\Omega_j)$ by Lemma~\ref{lemma:ps-identity}. 
Applying the weak identity~\eqref{eqn:double-integral-B-eta} first with the constant vector field $\eta\equiv a$ and then with $\eta=h_j a$, we obtain
\[
\int_{\Sigma_j}p_jh_j\,d\mu_j=0, \qquad  \frac{1}{2} \iint_{\Sigma_j\times\Sigma_j}B_{s_j}(\la x,y\rg)[h_j(x)-h_j(y)]^2\,d\mu_j(x)d\mu_j(y) =\int_{\Sigma_j}p_jh_j^2\,d\mu_j.
\]
Using $B_{s_j}=A_{s_j}+W_{s_j}$ and the symmetry of $W_{s_j}$, these identities give
\[
q_{s_j}[1-h_j] = -\int_{\Sigma_j}p_j\,d\mu_j +  \iint_{\Sigma_j\times\Sigma_j} W_{s_j}(\la x,y\rg)(1-h_j(x))(1-h_j(y))\,d\mu_j(x)d\mu_j(y).
\]
Since the links $\Omega_j$ are of class $C^2$, the functions $h_j$ are $C^1$, so we can use them as test functions for the stability inequality after mollifying and taking limits.
    Then, the stability inequality $q_{s_j}[1 - h_j ] \geq 0$ shows that the above expression is non-negative for each $j$.
Since $|1-h_j|\leq1+|a|$, the same cutoff argument as in Step~3 of Theorem~\ref{thm:compactness-thm}, using~\eqref{eqn:Ws-uniform-bound} and
$\eta_j\otimes\eta_j\xrightharpoonup{*}\eta\otimes\eta$, yields
\[
\iint W_{s_j} \,(1-h_j(x))(1-h_j(y))\,d\mu_j(x)d\mu_j(y) \to \iint
W_0 \, (1-\la a,\nu_x\rg)(1-\la a,\nu_y\rg)\,d\mu(x)d\mu(y).
\]
Finally, $p_j\mu_j\xrightharpoonup{*}\ell$ implies $\int_{\Sigma_j}p_j\,d\mu_j\to\ell(\bS^{N-1})$, and passing to the limit completes the proof.
\end{proof}

\begin{corollary}\label{cor:identify-limit-q}
    In the setting of Theorem~\ref{thm:compactness-thm}, we let $\nu_{\Omega}$ denote the measure-theoretic outward spherical conormal of the Caccioppoli set $\Omega \subset \bS^{N-1}$ and consider the norm
    \[
    \| u \|^2_{B_0} := \int |u|^2 \, d |D \mathbf{1}_{\Omega}| + \iint B_0 ( \la x,y \rg)  \, |u(x) - u(y)|^2 \, d |D \mathbf{1}_{\Omega}(x)| \, d |D \mathbf{1}_{\Omega}(y)| .
    \]
    Suppose that there exist Lipschitz vector fields $a_k \in \textup{Lip}( \partial^* \Omega ; \bR^N)$ such that $a_k \to \nu_{\Omega}$ strongly in the norm $\| - \|_{B_0}$.
    Then, we have the measure convergence $p_j \nu_j \mu_j \xrightharpoonup{*} p_0 \nu_{\Omega} \, |D \mathbf{1}_{\Omega}|$, where
    \[
    p_0(x) := \int_{\partial^* \Omega} B_0(\la x,y \rg) \, (1 - \la \nu_{\Omega}(x), \nu_{\Omega}(y) \rg ) \, d\mu(y).
    \]
    In particular, this property holds whenever there are $c, r_0>0$ such that $\cH^{N-2} (\partial^* \Omega \cap B_r(x)) \geq c \, r^{N-2}$ for a dense set of $x \in \partial^* \Omega$ and $r \in (0,r_0)$.
\end{corollary}
\begin{proof}
For brevity, we will denote $\mu := |D \mathbf{1}_{\Omega}|$, $d \mu_x := d |D \mathbf{1}_{\Omega}|(x)$, and $\nu_x := \nu_{\Omega}(x)$, so $\textup{spt} \, \mu = \overline{\partial^* \Omega}$.
By Theorem~\ref{thm:compactness-thm}, it suffices to identify the limit measure $p_j \nu_j \mu_j \xrightharpoonup{*} \mathbf{q}$ as $\mathbf{q} = p_0 \nu \, |D \mathbf{1}_{\Omega}|$.
We first observe that the Lipschitz approximation of $\nu$ can be upgraded to an approximation by smooth ambient vector fields.
By Kirszbraun's theorem~\cite{maggi}*{Theorem 7.2}, each vector field $a_k$ extends to a Lipschitz vector field $\tilde{a}_k : \bR^N \to \bR^N$ with the same Lipschitz constant $L_k$.
Let $\rho_{\ve}$ be the standard mollifier, so $A_{k,\ve} := \rho_{\ve} \ast \tilde{a}_k$ is a smooth vector field on $\bR^N$.
Moreover, on $\partial^*\Omega$,
\begin{equation}\label{eqn:lipschitz-approximation-bounds}
\| A_{k,\ve} - a_k \|_{L^{\infty}} \leq C L_k \ve, \qquad \textup{Lip}(A_{k,\ve} - a_k) \leq 2 L_k.
\end{equation}
By Step 3 of Theorem~\ref{thm:compactness-thm}, we have
\[
\sup_{x \in \textup{spt} \, \mu} \int B_0(\la x,y \rg) \, |x-y|^2 \, d\mu_y< \infty.
\]
Therefore, using the properties~\eqref{eqn:lipschitz-approximation-bounds}, we obtain $|(A_{k,\ve} - a_k)(x) - (A_{k,\ve}- a_k)(y)| \leq 2 L_k |x-y|$ and
\[
\iint B_0( \la x,y \rg) \, |(A_{k,\ve} - a_k)(x) - (A_{k,\ve}- a_k)(y)|^2 \, d\mu_x \, d\mu_y \to 0 \quad \text{as } \; \ve \downarrow 0
\]
by the dominated convergence theorem, because $A_{k,\ve} - a_k \to 0$ uniformly as $\ve \downarrow 0$.
The $L^2$ term in the norm $\|u \|^2_{B_0}$ converges to zero as well, for the same reason, thus $\| A_{k,\ve} - a_k \|_{B_0} \to 0$ as $\ve \downarrow 0$.
We can therefore choose $\ve_k$ so that $\|A_{k, \ve_k} - a_k \|_{B_0} \leq \| a_k - \nu \|_{B_0}$, and write $A_k := A_{k, \ve_k}$.
Then, the triangle inequality implies that
\[
\|A_k - \nu \|_{B_0} \leq \|A_k - a_k \|_{B_0} + \| a_k - \nu \|_{B_0} \leq 2 \, \| a_k - \nu \|_{B_0} \to 0
\]
producing the desired smooth approximation of $\nu_{\Omega}$.

We now fix a smooth vector field $\Phi :\bS^{N-1} \to \bR^N$.
For each $k$, taking $\Psi_k := \Phi \wedge A_k$ defines a smooth $\bigwedge^2 \bR^N$-valued field, so the limiting wedge identity from part $(iii)$ of Theorem~\ref{thm:compactness-thm} implies
\[
\iint B_0(\la x,y \rg) \ \la \Phi(x) \wedge A_k(x) - \Phi(y) \wedge A_k(y), \nu_x \wedge \nu_y \rg \, d \mu_x \, d\mu_y =0 .
\]
For each $k$, the vector field $e_k := A_k - \nu$ satisfies
\begin{align*}
    |\Phi(x) \wedge e_k(x) - \Phi(y) \wedge e_k(y)| &\leq \| \Phi\|_{L^{\infty}} |e_k(x) - e_k(y)| + \textup{Lip}(\Phi) \, |x-y| \, |e_k(y)|.
\end{align*}
We may therefore estimate, as $k \to \infty$,
\begin{align*}
    & C_{\Phi}^{-1} \iint B_0 \, |\Phi(x) \wedge e_k(x) - \Phi(y) \wedge e_k(y)|^2 \, d \mu_x \, d\mu_y \\
    &\leq \iint B_0 \, |e_k(x) - e_k(y)|^2 \, d\mu_x \, d\mu_y + \int |e_k(y)|^2 \Bigl( \int B_0 (\la x,y \rg) \, |x-y|^2 \, d\mu_x \Bigr) \, d\mu_y \leq C_{\Phi} \|e_k \|^2_{B_0} \to 0.
\end{align*}
On the other hand, $|\nu_x \wedge \nu_y| \leq |\nu_x - \nu_y|$, so Cauchy-Schwarz and the finiteness of
\[
\iint_{\partial^* \Omega \times \partial^* \Omega} B_0( \la x,y \rg) \, |\nu_x - \nu_y|^2 \, d\mu_x \, d\mu_y < \infty,
\]
by part $(ii)$ of Theorem~\ref{thm:compactness-thm}, implies that the wedge functional $\Phi \mapsto \Phi \wedge \nu$ is continuous with respect to the $B_0$-seminorm.
Consequently,
\[
\iint B_0 (\la x,y \rg) \, \la \Phi(x) \wedge \nu_x - \Phi(y) \wedge \nu_y, \nu_x \wedge \nu_y \rg \, d \mu_x \, d\mu_y = 0.
\]
We recall that, for arbitrary vectors $\alpha, \beta \in \bR^N$ and unit vectors $u,v \in \bS^{N-1}$, it holds that
\begin{align*}
    &\la \alpha \wedge u, u \wedge v \rg = \la \alpha, u \rg \la u,v \rg - \la \alpha, v \rg, \qquad \la \beta \wedge v, u \wedge v \rg = \la \beta , u \rg - \la \beta, v \rg \la u, v \rg, \\
    & \implies \la \alpha \wedge u - \beta \wedge v, u \wedge v \rg = \la \alpha - \beta, u -v \rg - (1 - \la u,v\rg) \, ( \la \alpha, u \rg + \la \beta, v \rg).
\end{align*}
Using the above vanishing and this formula with $(\alpha, \beta, u,v) = (\Phi(x), \Phi(y), \nu_x, \nu_y \rg$, we find
\begin{align*}
\iint B_0 \, \la \Phi(x) - \Phi(y), \nu_x - \nu_y \rg \, d\mu_x \, d\mu_y &= \iint B_0 \, (1 - \la \nu_x, \nu_y \rg) \, ( \la \Phi(x), \nu_x \rg + \la \Phi(y), \nu_y \rg) \, d\mu_x \, d\mu_y \\
&= 2 \iint \la \Phi(x), \nu_x \rg \left[ \int B_0 (\la x,y \rg) \, (1 - \la \nu_x, \nu_y \rg) \, d\mu_y \right] \, d \mu_x \\
&= 2 \int p_0 \la \Phi, \nu \rg \, d\mu
\end{align*}
In the last step, we used the fact that the kernel and the factor $1 - \la \nu_x, \nu_y \rg$ are symmetric in $x,y$, so the two terms on the right-hand side are equal.
We then used the definition of $p_0$ above, where
\[
\int p_0 \, d\mu = \frac{1}{2} \iint B_0( \la x,y \rg) \, |\nu_x - \nu_y|^2 \, d\mu_x \, d\mu_y < \infty, \qquad \implies \qquad p_0 \in L^1(\mu).
\]
Finally, the property $(iii)$ of Theorem~\ref{thm:compactness-thm} shows that the left-hand side of this expression equals $2 \int \la \Phi, d \mathbf{q}\rg$ for every smooth vector field.
Combining these properties, we obtain $\mathbf{q} = p_0 \nu \, \mu$.

Regarding the second claim, it is clear that having $\mu(B_r(x)) = \cH^{N-2} (\partial^* \Omega \cap B_r(x)) \geq c\, r^{N-2}$ for a dense set of $x\in \partial^* \Omega$ and $r \in (0,r_0)$ implies the same bound for every given $(x,r) \in \overline{\partial^* \Omega} \times (0,r_0)$, possibly with a smaller constant.
Indeed, we can find some $(y,s)$ in the admissible set with $y \in B_{\frac{r}{4}}(x)$ and $s \in ( \frac{r}{3}, \frac{r}{2})$, so $B_s(y) \subset B_r(x)$.
Consequently, 
\[
\mu(B_r(x)) \geq \mu(B_s(y)) \geq c s^{N-2} \geq c \, 3^{2-N} \, r^{N-2}
\]
uniformly on $\overline{\partial^* \Omega}$.
In addition the uniform perimeter bound of Proposition~\ref{prop:bound-s} and its subsequent application to Theorem~\ref{thm:compactness-thm} imply that $\mu(B_r(x)) \leq C r^{N-2}$ for a constant $C = C(N, \bar{s})$.
Also, $B_0(\la x,y \rg) \asymp |x-y|^{1-N}$ as in Proposition~\ref{prop:angular-} and Theorem~\ref{thm:compactness-thm}, so the norm $\|u\|^2_{B_0}$ is equivalent to
\[
\| u\|^2_{L^2(\mu)} + \iint \frac{|u(x) - u(y)|^2}{|x-y|^{N-1}} \, d\mu_x \, d\mu_y,
\]
namely the Besov $B^{1/2}_{2,2}$-norm on $\overline{\partial^* \Omega} = \textup{spt} \, \mu$.
By standard trace and extension theory for such function spaces~\cite{besov}*{Ch. VII, Theorem 1}, the space of Lipschitz functions on $\overline{\partial^* \Omega} = \textup{spt} \, \mu$ is dense in the norm $\| - \|_{B_0}$.
Since $\nu$ has finite component-wise $\|-\|_{B_0}$-norm, by Theorem~\ref{thm:compactness-thm} $(ii)$, we can apply this density result obtain a sequence of Lipchitz vector fields $a_k \in \textup{Lip}( \partial^* \Omega; \bR^N)$ with $\| a_k - \nu \|_{B_0} \to 0$ and conclude by the first part.
This completes the proof.
\end{proof}

\section{Classification of stable links}\label{section:stable-links}

In this section, we establish pinching theorems for the links of stable $s$-minimal cones as $s \downarrow 0$.
Using these results, we will prove Theorem~\ref{thm:main}.
\begin{lemma}\label{lemma:A-B-properties}
    For every $z \in [-1,1)$, the function $B_s(z)$ is increasing in $z$, and $B_s(z) \geq \frac{1}{N(N-1)}$.
    Moreover, the function $\frac{A_0(z)}{B_0(z)}$ is increasing in $z$, with $\frac{A_0(z)}{B_0(z)} \geq \kappa_N := \frac{\Gamma( \frac{N}{2})^2}{\Gamma(N-1)}$.
\end{lemma}
\begin{proof}
    The function $B_s(z)$ is increasing in $z$, with
    \[
    B'_s(z) = (N+s) \int_0^{\infty} (1+\rho^2 - 2 \rho z)^{- \frac{N+s+2}{2}} \rho^{N-1} \, d \rho >0.
    \]
    Therefore, $B_s(z) \geq B_s(-1)$, for which we compute
    \[
    B_s(-1) = \int_0^{\infty} \frac{\rho^{N-2}}{(1+\rho)^{N+s}} = \frac{\Gamma(N-1)\Gamma(1+s)}{\Gamma(N+s)} \geq \frac{\Gamma(N-1)}{\Gamma(N+1)} = \frac{1}{N(N-1)}.
    \]
    To compare $A_0(z), B_0(z)$, we split the integrals at $\rho=1$ and define a measure $\mu_z$ by
    \begin{align*}
    d \mu_z(\rho) &= B_0(z)^{-1} ( 1 + \rho^{N-2}) (1 + \rho^2 - 2 \rho z)^{- \frac{N}{2}} \, d \rho, \\
    \frac{A_0(z)}{B_0(z)} &= 2 \int_0^1 \frac{\rho^{\frac{N-2}{2}}}{1 + \rho^{N-2}} \, d \mu_z(\rho), \qquad \frac{d}{dz} \frac{A_0(z)}{B_0(z)} = \textup{Cov}_{\mu_z} \Bigl( \frac{2 \rho^{\frac{N-2}{2}}}{1 + \rho^{N-2}}, \frac{N \rho}{1 + \rho^2 - 2 \rho z} \Bigr).
    \end{align*}
    Both functions in this covariance are increasing, since $( \frac{x}{1+x^2})' = \frac{1-x^2}{(1+x^2)^2} \geq 0$ for $|x|\leq1$.
    Therefore, Chebyshev's covariance inequality implies that this covariance is non-negative, hence the map $z \mapsto \frac{A_0(z)}{B_0(z)}$ is increasing.
    To compute the minimum at $z=-1$, we evaluate the Beta functions
    \[
    A_0(-1) = \int_0^{\infty} \frac{\rho^{\frac{N-2}{2}}}{(1+\rho)^N} \, d \rho =\frac{\Gamma( \frac{N}{2})^2}{\Gamma(N)}, \qquad B_0(-1) = \int_0^{\infty} \frac{\rho^{N-2}}{(1+\rho)^N} \, d \rho = \frac{\Gamma(N-1)\Gamma(1)}{\Gamma(N)} = \frac{1}{N-1}.
    \]
    The two quantities have ratio $\kappa_N = \frac{\Gamma(\frac{N}{2})^2 (N-1)}{\Gamma(N)} = \frac{\Gamma(\frac{N}{2})^2}{\Gamma(N-1)}$, proving our claim.
\end{proof}

\begin{theorem}\label{thm:uniform-pniching}
For every $N \geq 3$, $\Lambda >1$, and $\bar{s} \in (0,1)$, there exists some $\ve = \ve(N,\Lambda, \bar{s})$ with the following property.
For $s \in (0, \bar{s}]$, let $E = C (\Omega)$ be a stable $s$-minimal cone in $\bR^N$ with $C^2$ link $\Sigma = \partial \Omega \subset \bS^{N-1}$, satisfying $\cH^{N-2}(\Sigma \cap B_r(x)) \leq \Lambda r^{N-2}$ for every $x \in \bS^{N-1}$ and $r>0$, and $\cH^{N-2}(\Sigma) \geq \Lambda^{-1}$.
Moreover, suppose that
\[
\inf_{a \in \bS^{N-1}} \int_{\Sigma} |\nu - a|^2 \, d \mu < \ve
\]
where $\nu$ is the outward conormal to $\Sigma$.
Then, $\Sigma$ is an equator and $E$ is a half-space. 
\end{theorem}
\begin{proof}
Let $a \in \bS^{N-1}$ be such that $D := \| \nu - a \|_{L^2(\Sigma)}$ satisfies $D^2<\ve$.
By the bound $\mu(B_r(x)) \leq \Lambda r^{N-2}$ and the bound on $W_s$ from Lemma~\ref{lemma:ws-useful-bound}, we can apply Lemma~\ref{lemma:potential-upper-growth} with $(m,\alpha,\beta_0,q) = (N-2, N-3+s, 1-\bar{s}, 1)$ and absorb the logarithm into a smaller power to obtain
\begin{equation}\label{eqn:general-bound-with-constant}
\sup_{\substack{F\subset\Sigma\\ \mu(F)\leq t}}\sup_{x\in\Sigma}
\int_F W_s(\la x,y\rg)\,d\mu_y
\leq C(N,\bar s,\Lambda)t^\gamma,
\qquad \gamma:=\frac{1-\bar s}{2(N-2)},
\end{equation}
for $0<t\leq1$.
We write $h(x) := \la a, \nu(x) \rg$ and define $h_{\pm}$ by $h_+ = \max \{ h,0\}$ and $h_- = \max \{ -h,0 \}$, which are Lipschitz because $\Omega$ is of class $C^2$.
As in Lemma~\ref{lemma:extend-to-moduli}, we have $|h_-(x) - h_-(y)| \leq |h(x) - h(y)| \leq C|x-y|$, so the functions $h_{\pm}$ are admissible in the stability inequality, and $q_s [h_{\pm}] \geq 0$.
We take the scalar product of the identity~\eqref{eqn:ps(x)-nux} from Lemma~\ref{lemma:ps-identity} with $a$, test against $h_-$, and compare the resulting identity with $q_s[h_-]\geq0$, using $h = h_+ - h_-$ and $h_+ h_- = 0$.
This gives
\begin{equation}\label{eqn:h-minus-comparison}
\iint_{\{h>0\}\times\{h<0\}}
B_s(\la x,y\rg)h_+(x)h_-(y)\,d\mu_xd\mu_y
\leq
\iint_{\{h<0\}\times\{h<0\}}
W_s(\la x,y\rg)h_-(x)h_-(y)\,d\mu_xd\mu_y .
\end{equation}
Because $|a| = |\nu| = 1$, we have $|\nu - a|^2 = 2 (1 - \la a,\nu \rg) = 2 (1-h)$, hence
\begin{equation}\label{eqn:sigma-integral}
\int_{\Sigma} h_+ \, d\mu \geq \int_{\Sigma} h \, d\mu = \cH^{N-2}(\Sigma) - \tfrac{1}{2} D^2, \qquad \text{and} \qquad \mu ( \{ h < 0 \} ) \leq \tfrac{1}{2} D^2
\end{equation}
because $|\nu-a|^2>2$ on the set $F = \{ h < 0 \}$.
Also, using $0 \leq h_- \leq 1$, we can estimate
\begin{align*}
    & \iint_{ \{ h<0 \} \times \{ h<0 \} } W_s(\la x,y \rg) h_-(x) h_-(y) \, d\mu_x \, d \mu_y \\
    &\leq \Bigl( \sup_{ \substack{ F \subset \Sigma \\ \mu(F) < D^2} } \int_F W_s (\la x,y \rg) \, d\mu_y \Bigr) \int_{ \{ h <0 \} } h_-(x) \, d\mu_x \leq C(N,\bar{s}, \Lambda) D^{2 \gamma}  \int_{ \{ h < 0 \} } h_-(x) \, d\mu_x
\end{align*}
using the bound~\eqref{eqn:general-bound-with-constant} with $F = \{ h<0 \}$ and $\mu(F) \leq \frac{1}{2} D^2$.
Next, Lemma~\ref{lemma:A-B-properties} gives $B_s \geq \frac{1}{N(N-1)}$, so
\[
\iint_{\{h>0\}\times\{h<0\}} B_s h_+(x)h_-(y)\,d\mu_xd\mu_y \geq N^{-2} \left(\int_{\{h>0\}}h_+\,d\mu\right) \left(\int_{\{h<0\}}h_-\,d\mu\right).
\]
Thus, either $h_-=0$ $\mu$-a.e., or after cancelling
$\int_{\{h<0\}}h_-\,d\mu>0$ in~\eqref{eqn:h-minus-comparison}, we must have 
\[
\cH^{N-2}(\Sigma) - \tfrac{1}{2} D^2 \leq \int_{ \{ h > 0 \} } h_+ \, d\mu \leq C(N, \bar{s}, \Lambda) \, D^{2 \gamma}, \qquad \implies \qquad \cH^{N-2}(\Sigma) \leq \tfrac{1}{2} \ve + C \ve^{\gamma}
\]
due to $D^2<\ve$.
Since $\cH^{N-2}(\Sigma) \geq \Lambda^{-1}$, this is impossible for $\ve = \ve(N, \bar{s}, \Lambda)$ sufficiently small.
Hence, $h_- = 0$, so $\la a, \nu \rg \geq 0$ everywhere by continuity.
On the other hand, testing the identity~\eqref{eqn:ps(x)} of Lemma~\ref{lemma:ps-identity} with the constant vector field $a$ and integrating over $\Sigma$ gives
\[
\int_{\Sigma} B_s (\la x,y \rg) \, [h(x) - h(y)] \, d\mu_y = p_s(x) h(x), \qquad \implies \qquad \int_{\Sigma} p_s(x) \la a, \nu_x \rg \, d\mu_x = 0.
\]
Moreover, the property~\eqref{eqn:sigma-integral} shows that $h_+ = \la a, \nu \rg > 0$ on a set of positive $\mu$-measure, so using $p_s, h_0 \geq 0$, we can find some $x_0 \in \Sigma$ with $p_s(x_0) = 0$ and $\la a, \nu_{x_0} \rg > 0$.
We recall that
\[
p_s(x_0) = \int_{\Sigma} B_s( \la x_0,y \rg) \, (1 - \la \nu(x_0), \nu(y) \rg ) \, d\mu_y
\]
where both factors in the integrand are non-negative and $B_s>0$.
Therefore, $\nu(y) =\nu(x_0)$ for $\mu$-a.e. $y$, and since $\Omega$ is of class $C^2$, this forces $\nu \equiv b$ for a fixed vector $b \in \bS^{N-1}$.
Moreover, $b$ is tangent to the sphere at every $x \in \Sigma$, so $\la x, b \rg = 0$ implies that $\Sigma \subset \bS^{N-1} \cap b^{\perp}$ is contained in an $(N-2)$-dimensional equator.
Also, $T_x \Sigma = \{ x,b \}^{\perp} = T_x ( \bS^{N-1} \cap b^{\perp})$ implies that every connected component of $\Sigma$ is open and closed in the equator, so $\Sigma = \bS^{N-1} \cap b^{\perp}$ and $E$ is a half-space.
\end{proof}
\begin{corollary}\label{cor:near-isoperimetric}
    For every $N \geq 3$, there exist $s_N \in (0,1)$ and $\delta_N>0$ with the following property.
    Let $E = C (\Omega)$ be a stable $s$-minimal cone in $\bR^N$ with $C^2$ link $\Sigma = \partial \Omega \subset \bS^{N-1}$, for $s \in (0, s_N)$.
    If $\cH^{N-2}(\Sigma) \leq |\bS^{N-2}| + \delta_N$, then $E$ is a half-space.
\end{corollary}
\begin{proof}
Suppose, for contradiction, that this property fails, so there exist $s_j \downarrow 0$ and $C^2$ spherical domains $\Omega_j$ with non-equatorial $C^2$ links $\Sigma_j = \partial \Omega_j$ such that $\cH^{N-2}(\Sigma_j) \to |\bS^{N-2}|$.
By Theorem~\ref{thm:compactness-thm}, we can extract a subsequential limit $\mathbf{1}_{\Omega_j} \to \mathbf{1}_{\Omega}$ in $L^1(\bS^{N-1})$ for a domain $\Omega$ with $|\Omega| = \frac{1}{2} |\bS^{N-1}|$.
Since each $\Omega_j$ has $C^2$ boundary, the lower semicontinuity of the perimeter under $L^1$ convergence~\cite{maggi}*{Proposition 12.15} implies that
\[
\textup{Per}_{\bS^{N-1}} (\Omega) \leq \liminf_{j \to \infty} \textup{Per}_{\bS^{N-1}}(\Omega_j) = \liminf_{j \to \infty} \cH^{N-2}(\Sigma_j) \leq |\bS^{N-2}|. 
\]
By the spherical isoperimetric inequality, originally due to Schmidt and proved in~\cite{duzaar-fusco} for Caccioppoli sets, we also have $\textup{Per}_{\bS^{N-1}}(\Omega) \geq |\bS^{N-2}|$ with equality precisely for a hemisphere.
Thus, $\Omega$ is a hemisphere with constant outward conormal $a$ and $- D_{\bS^{N-1}} \mathbf{1}_{\Omega} = a \, \cH^{N-2} \mres \partial^* \Omega$.
By Theorem~\ref{thm:compactness-thm}, the Gauss-Green measures satisfy $\nu_j \mu_j \xrightharpoonup{*} a \, \cH^{N-2} \mres \partial^* \Omega$, and since $\cH^{N-2}(\Sigma_j) \to |\bS^{N-2}|$, we obtain
\[
\Bigl\la a, \int_{\Sigma_j} \nu_j \, d \mu_j \Bigr\rg \to |\bS^{N-2}|,\qquad \implies \qquad \int_{\Sigma_j} |\nu_j - a|^2 = 2 \, \cH^{N-2} (\Sigma_j) - 2 \Bigl\la a, \int_{\Sigma_j} \nu_j \, d \mu_j \Bigr\rg \to 0.
\]
Since $\int_{\Sigma_j} |\nu_j - a|^2 \to 0$ and the cones $C(E_j)$ satisfy the properties of Theorem~\ref{thm:uniform-pniching}, we deduce that they are equators for all $j$ sufficiently large.
This produces a contradiction.
\end{proof}

Combining the results~\ref{lemma:A-B-properties} - \ref{cor:near-isoperimetric}, we obtain a conditional flatness result for stable $s$-minimal cones.
\begin{proposition}\label{prop:dimension-descent}
For $N \geq 3$, suppose that every $s=0$ limit link $\Omega \subset \bS^{N-1}$ of stable $s$-minimal cones in $\bR^N$ is a hemisphere up to null sets.
Then, there exists an $s_* > 0$ such that for all $s \in (0,s_*)$, every stable $s$-minimal cone in $\bR^N$ with $C^2$ boundary away from the origin is a half-space. 
\end{proposition}
\begin{proof}
If the result failed, there would exist a sequence of $s_j \downarrow 0$ and non-flat stable $s_j$-minimal cones $E_j = C(\Omega_j) \subset \bR^N$ with $C^2$ links $\Sigma_j = \partial \Omega_j$. 
Then, we can apply Theorem~\ref{thm:compactness-thm} to extract a subsequential limit $\mathbf{1}_{\Omega_j} \to \mathbf{1}_{\Omega}$ in strict $\textup{BV}(\bS^{N-1})$, and every such limit $\Omega$ is, up to null sets, a hemisphere whose equator $\Sigma = \partial^* \Omega = \bS^{N-1} \cap a^{\perp}$ has constant outward conormal $\nu_{\Omega} = a$.
The property $(ii)$ of Theorem~\ref{thm:compactness-thm} implies that the measures $\mu_j := \cH^{N-2} \mres \Sigma_j$ satisfy
\begin{align*}
& \mu_j \xrightharpoonup{*} |D \mathbf{1}_{\Omega}|, \qquad (x, \nu_j(x))_{\#} \mu_j \xrightharpoonup{*} (x, \nu_{\Omega}(x))_{\#} |D \mathbf{1}_{\Omega}| = (x,a)_{\#} ( \cH^{N-2} \mres (\bS^{N-1} \cap a^{\perp})),
\end{align*}
and $\nu_j \mu_j \xrightharpoonup{*} - D \mathbf{1}_{\Omega} = a \, \cH^{N-2} \mres \Sigma$.
Combining with $\mu_j(\bS^{N-1}) \to |D \mathbf{1}_{\Omega}|(\bS^{N-1}) = |\bS^{N-2}|$, we obtain
\begin{align*}
    \int_{\Sigma_j} |\nu_j - a|^2 \, d\mu_j &= 2 \mu_j (\bS^{N-1}) - 2 \Bigl\la a, \int_{\Sigma_j} \nu_j \, d \mu_j \Bigr\rg \to 2 \mu(\bS^{N-1}) - 2|\bS^{N-2}| = 0.
\end{align*}
Thus, for all sufficiently large $j$, the domains $\Omega_j$ satisfy the conditions of Theorem~\ref{thm:uniform-pniching}, so they are hemispheres and the cones $C(\Omega_j)$ are flat for sufficiently large $j$.
This is a contradiction.  
\end{proof}

\subsection{The limit link in three dimensions}

To prove Theorem~\ref{thm:main}, we show that the conditions of Proposition~\ref{prop:dimension-descent} hold for $N=3$.
First, we establish a precise bound for the quantity $\iint_{\Sigma \times \Sigma} W_0 \la \nu_x, \nu_y \rg$ in the stability inequality $(v)$ of Theorem~\ref{thm:compactness-thm} when $\Omega$ is a half-volume Caccioppoli set.
\begin{lemma}\label{lemma:domain-expansion}
    For $N \geq 3$, let $\Omega \subset \bS^{N-1}$ be a Caccioppoli set with $\Sigma := \partial^* \Omega$ and define
    \[
    W_0(z) := \int_0^1 (1 - \rho^{\frac{N-2}{2}})^2  (1 +\rho^2 - 2 \rho z)^{- \frac{N}{2}} \, d \rho.
    \]
    Suppose that $|\Omega| = \frac{1}{2} |\bS^{N-1}|$ and $\iint_{\Sigma \times \Sigma} W_0( \la x,y \rg) \, d |D\mathbf{1}_{\Omega}|(x) \, d | D \mathbf{1}_{\Omega} | (y) < \infty$.
    Then,
    \[
    0 \leq \iint_{\Sigma \times \Sigma} W_0 (\la x,y \rg) \, \la \nu_x, \nu_y \rg \, d |D\mathbf{1}_{\Omega}|(x) \, d | D \mathbf{1}_{\Omega} | (y) \leq w_N \, |\bS^{N-1}|^2
    \]
    where $w_N = \frac{1}{4} N^3$ for $N \geq 4$ and $w_3 = \frac{3}{16 \pi}$.
    Moreover, we have
    \[
    c_N \cH^{N-2}(\Sigma)^2 \leq \iint_{\Sigma \times \Sigma} W_0(\la x,y \rg) \, d\mu_x \, d\mu_y, \qquad \text{where } \; c_N := \int_0^1 (1-\rho^2)^{-1} \bigl( 1 - \rho^{\frac{N-2}{2}} \bigr)^2 \, d \rho.
    \]
\end{lemma}
\begin{proof}
We first perform a general computation for all $N \geq 3$; this includes the general measure bound above.
We then specialize to more precise bounds for $N=3$.

\smallskip \noindent \textbf{Step 1: The general computation.}
We consider the function $u := \mathbf{1}_{\Omega} - \frac{1}{2}$, which is bounded and has zero spherical mean by construction.
Thus, $u \in L^{\infty}(\bS^{N-1}) \subset L^2(\bS^{N-1})$ admits an expansion into $L^2$-orthonormal spherical harmonics, given by $u = \sum_{k \geq 1} \sum_{\alpha} c_{k \alpha} Y_{k\alpha}$.
We write $\mu := |D \mathbf{1}_{\Omega}|$ for brevity, so $D_{\bS^{N-1}} u = D_{\bS^{N-1}} \mathbf{1}_{\Omega} = - \nu \, \mu$ as vector-valued measures.
Since $|\mathbf{1}_{\Omega} - \frac{1}{2}| = \frac{1}{2}$ a.e., we have
\begin{equation}\label{eqn:parseval-bound}
\sum_{k \geq 1, \alpha} |c_{k \alpha}|^2 = \bigl\| \mathbf{1}_{\Omega} - \tfrac{1}{2} \bigr\|^2_{L^2(\bS^{N-1})} = \tfrac{1}{4}  \, |\bS^{N-1}|.
\end{equation}
We denote by $P_{\rho}(z)$ the spherical Poisson kernel, given as in~\cite{axler-bourdon-ramey}*{Ch. 5} by
\[
P_{\rho}(z) := |\bS^{N-1}|^{-1} (1-\rho^2) (1 + \rho^2 - 2 \rho z)^{- \frac{N}{2}}.
\]
This has the property that for every degree-$k$ spherical harmonic $Y_k$ has Poisson extension
\begin{equation}\label{eqn:poisson-extension-Yk}
\int_{\bS^{N-1}} P_{\rho}(\la x,y \rg) \, Y_k(y) \, dS_y = \rho^k Y_k(z)
\end{equation}
given by multiplication with $\rho^k$.
Then, the integral kernel $W_0(z)$ is expressible in the form
\[
W_0(z) = |\bS^{N-1}| \int_0^1 (1-\rho^2)^{-1} \bigl(1 - \rho^{\frac{N-2}{2}} \bigr)^2 P_{\rho}(z) \, d \rho.
\]
Let $Y = H_k|_{\bS^{N-1}}$ denote a normalized homogeneous degree-$k$ harmonic.
The polynomial $x_i H_k - \frac{|x|^2}{2k+N-2} \partial_i H_k$ is harmonic and homogeneous of degree $k+1$, so
\[
k x_i Y = kx_i H_k - \tfrac{k}{2k+N-2} |x|^2 \partial_i H_k + \tfrac{k}{2k+N-2} \partial_i H_k
\]
on the sphere.
This implies the identity
\begin{equation}\label{eqn:Di-Y-identity}
    D_i Y = \partial_i H_k - k x_i Y = \frac{k+N-2}{2k+N-2} \partial_i H_k - k \Bigl( x_i H_k - \frac{|x|^2}{2k+N-2} \partial_i H_k \Bigr).
\end{equation}
The two summands are harmonic, of degrees $(k-1)$ and $(k+1)$, respectively, and satisfy
\[
\sum_i \| \partial_i H_k \|^2_{L^2} = \int_{\bS^{N-1}} |\nabla_{\bR^N} H_k|^2 \, dS = k (2k+N-2)
\]
due to $|\nabla_{\bR^N} H_k|^2 = |\nabla_{\bS^{N-1}} Y|^2 + k^2 Y^2$, with $k(k+N-2)$ the eigenvalue of the spherical Laplacian.
Consequently, the degree-$(k-1)$ part has total $L^2$-norm $\frac{k(k+N-2)^2}{2k+N-2}$, while the $(k+1)$-part has total $L^2$-norm $\frac{k^2(k+N-2)}{2k+N-2}$. 
As in Lemma~\ref{eqn:ws-properties}, we denote by $W_{0,\delta}(z)$ the truncated kernel
\begin{equation}\label{eqn:W-zero-delta}
W_{0,\delta}(z) := |\bS^{N-1}| \int_0^{1-\delta} (1-\rho^2)^{-1} (1 - \rho^{\frac{N-2}{2}})^2 P_{\rho}(z) \, d \rho.
\end{equation}
For every $\delta>0$, the truncated kernel $W_{0,\delta}(z)$ is smooth.
Moreover, we can cover $\bS^{N-1}$ by finitely many charts and apply the smooth strict approximation of BV functions~\cite{ambrosio-fusco-pallara}*{Theorem 3.9} (see also~\cite{maggi}*{Theorem 13.8}) to obtain a sequence of smooth functions $u_j \in C^{\infty}$, with $|u_j| \leq \frac{1}{2}$, such that $u_j \to u \in L^1(\bS^{N-1})$ and $D_{\bS^{N-1}} u_j \xrightharpoonup{*} D_{\bS^{N-1}} u$ as vector-valued measures, while 
\[
\sup_j |D_{\bS^{N-1}} u_j| (\bS^{N-1}) < \infty, \qquad \int_{\bS^{N-1}} |\nabla u_j| \, dS \to |D_{\bS^{N-1}} u| (\bS^{N-1}).
\]
Therefore, $D_i u_j \, dS \xrightharpoonup{*} D_i u$ as measures with totally bounded variation, hence also
\[
(D_i u_j \, dS) \otimes (D_i u_j \, dS) \xrightharpoonup{*} (D_i u) \otimes (D_i u).
\]
Since $|u_j| \leq \frac{1}{2}$ and $u_j \to u$ in $L^1$, we also have $u_j \to u$ in $L^2$.
Then, for $\delta>0$, we can integrate by parts twice in terms of the truncated kernel $W_{0,\delta}$ and the approximation functions $u_j$ to obtain
\begin{align*}
     \sum_{i=1}^N \iint W_{0,\delta}(\la x,y \rg) \, D_i u_j(x) \, D_i u_j(y) \, dS_x \, dS_y
    & \to \sum_{i=1}^N \iint W_{0,\delta}(\la x,y \rg) \, d(D_i u)(x) \, d (D_i u)(y) \\
    &= \iint_{\Sigma \times \Sigma} W_{0,\delta}( \la x,y \rg) \, \la \nu_x, \nu_y \rg \, d\mu_x \, d \mu_y.
\end{align*}
The kernel $W_{0,\delta}$ induces a bilinear form on smooth functions, given by
\[
\cT_{W,\delta} : (f,g) \mapsto \sum_{i=1}^N \iint W_{0,\delta}(\la x,y \rg) \, D_i f(x) \, D_i g(y) \, dS_x \, dS_y = \sum_{i=1}^N \la D_i f, W_{0,\delta} \ast D_i g \rg
\]
which is $O(N)$-invariant.
Therefore, the spaces of spherical harmonics of each degree are mutually orthogonal under the form $\cT_{W,\delta}$, and its restriction to each degree-$k$ space is a scalar multiple of the ordinary $L^2$ product.
The computations~\eqref{eqn:poisson-extension-Yk} - \eqref{eqn:Di-Y-identity} show that evaluating the form $\cT_{W,\delta}$ on a normalized degree-$k$ harmonic produces
\[
\lambda_k(\delta) = \frac{k(k+N-2)}{2k+N-2} \left( (k+N-2) \int_0^{1-\delta} \rho^{k-1} \frac{(1- \rho^{\frac{N-2}{2})})^2}{1-\rho^2} \, d \rho + k \int_0^{1-\delta} \rho^{k+1} \frac{(1 - \rho^{\frac{N-2}{2}})^2}{1-\rho^2} \, d \rho \right).
\]
Consequently, for each $\delta>0$ and smooth function $u_j$, computing $\cT_{W,\delta}(u_j,u_j)$ yields the identity
\begin{equation}\label{eqn:W0,delta-multiplier}
\sum_{i=1}^N \iint W_{0,\delta}(\la x,y \rg) \, D_i u_j(x) \, D_i u_j(y) \, dS_x \, dS_y = |\bS^{N-1}| \sum_{k,\alpha} \lambda_k(\delta) \, |c^{(j)}_{k \alpha}|^2.
\end{equation}
Here, $c^{(j)}_{k,\alpha}$ are the coefficients in the expansion $u_j = \sum_{k \geq 1} \sum_{\alpha} c^{(j)}_{k\alpha} Y_{k\alpha}$ into spherical harmonics.
Also, for every fixed $k$, the expression $\lambda_k(\delta)$ involves a positive integrand, so it increases, as $\delta \downarrow 0$, to
\begin{equation}\label{eqn:tau-k}
\lambda_k = \frac{k(k+N-2)}{2k+N-2} \left( (k+N-2) \int_0^1 \rho^{k-1} \frac{(1- \rho^{\frac{N-2}{2})})^2}{1-\rho^2} \, d \rho + k \int_0^1 \rho^{k+1} \frac{(1 - \rho^{\frac{N-2}{2}})^2}{1-\rho^2} \, d \rho \right).
\end{equation}
Moreover, $\sup_{k \geq 1}\lambda_k < N^3$
because $1 - \rho^{\frac{N-2}{2}} \leq N(1-\rho)$ for $\rho \in [0,1]$, hence 
\[
\int_0^1 \rho^{\ell-1} (1-\rho^2)^{-1} (1-\rho^a)^2 \, d \rho \leq N^2 \int_0^1 \rho^{\ell} (1 - \rho) \, d \rho < N^2 (\ell+1)^{-2}.
\]
Substituting these estimates into the definition of $\lambda_k$, we obtain
\[
\lambda_k \leq N^2 \frac{k(k+N-2)}{2k+N-2} \Bigl(\frac{k+N-2}{k^2} + \frac{k}{(k+2)^2} \Bigr) < N^2 ( 2 + k^{-1}(N-2)) < N^3
\]
as claimed.
Moreover, because $u_k \to u$ in $L^2(\bS^{N-1})$, Parseval's formula gives
\[
\sum_{k,\alpha} |c^{(j)}_{k \alpha} - c_{k \alpha}|^2 \to 0, \qquad \implies \qquad \sum_{k,\alpha} \lambda_k (\delta) \, |c^{(j)}_{k \alpha}|^2 \to \sum_{k,\alpha} \lambda_k(\delta) \, |c_{k \alpha}|^2
\]
due to $\sup_k \lambda_k(\delta) \leq \sup_k \lambda_k < N^3$.
Combining the above equalities, we may therefore pass to the limit as $j \to \infty$ in~\eqref{eqn:W0,delta-multiplier} to obtain
\begin{align*}
&\iint_{\Sigma \times \Sigma} W_{0,\delta}(\la x,y \rg) \, \la \nu_x, \nu_y \rg \, d\mu_x \, d \mu_y = |\bS^{N-1}| \sum_{k \geq 1, \alpha} \lambda_k(\delta) \, |c_{k,\alpha}|^2.
\end{align*}
Next, we observe that 
\[
0 \leq |W_{0,\delta} (\la x,y \rg) \, \la \nu_x, \nu_y \rg| \leq W_{0,\delta}(\la x,y \rg) \leq  W_0(\la x,y \rg)
\]
while $\iint_{\Sigma \times \Sigma} W_0( \la x,y \rg ) \, d \mu_x \, d\mu_y < \infty$, as in Lemma~\ref{eqn:ws-properties}.
Thus, the left-hand side of the above integration converges as $\delta \downarrow 0$ by the dominated convergence theorem.
Also, $\lambda_k(\delta) \uparrow \lambda_k$ as $\delta \downarrow 0$ and $\sup_k \lambda_k<N^3$.
Thus, we can apply the monotone convergence theorem and send $\delta \downarrow 0$ to obtain
\begin{equation}\label{eqn:dominated-convergence-Omega}
    \iint_{\Sigma \times \Sigma} W_0 (\la x,y \rg ) \, \la \nu_x, \nu_y \rg \, d \mu_x \, d \mu_y = |\bS^{N-1}| \sum_{k \geq 1, \alpha} \lambda_k |c_{k\alpha}|^2.
\end{equation}
For $N \geq 4$, we use the bound $\lambda_k < N^3$ and recall that $\sum_{k,\alpha} |c_{k,\alpha}|^2 = \frac{1}{4} |\bS^{N-1}|$ due to~\eqref{eqn:parseval-bound}.
Combining these computations proves the claimed upper bound $w_N |\bS^{N-1}|^2$ for the integral, where $w_N = \frac{1}{4} N^3$.

\smallskip \noindent \textbf{Step 2: The measure bound}.
In terms of the spherical harmonics $Y_{k \alpha}$, the spherical Poisson kernel has the expansion $P_{\rho}(\la x,y \rg) = \sum_{k \geq 0} \sum_{\alpha} \rho^k Y_{k \alpha}(x) Y_{k \alpha}(y)$, cf.~\cite{axler-bourdon-ramey}*{Ch. 5}.
Thus, for every finite Radon measure $\sigma$ on $\bS^{N-1}$, we use the Cauchy-Schwarz inequality to bound
\[
\iint P_{\rho}(\la x,y \rg) \,d\sigma_x \, d\sigma_y = \sum_{k \geq 0} \sum_{\alpha} \rho^k \Bigl| \int Y_{k \alpha} \, d \sigma \Bigr|^2 \geq |\bS^{N-1}|^{-1} \, \sigma(\bS^{N-1})^2
\]
because the normalized degree-zero harmonic on the sphere is $|\bS^{N-1}|^{- \frac{1}{2}}$.
Applying these bounds to the truncated kernel $W_{0,\delta}(z)$, expressed in terms of the spherical Poisson kernel in~\eqref{eqn:W-zero-delta}, we obtain
\[
\iint W_{0,\delta}(\la x,y \rg) \, d \mu_x \, d\mu_y \geq \cH^{N-2}(\Sigma)^2 \int_0^{1 - \delta} (1-\rho^2)^{-1} \bigl(1 - \rho^{\frac{N-2}{2}} \bigr)^2 \, d \rho.
\]
Finally, the assumption that $\iint W_0(\la x,y \rg) \, d \mu_x \, d\mu_y< \infty$, as above, allows to apply the monotone convergence theorem and send $\delta \downarrow 0$, since both sides are strictly decreasing in $\delta$.
We arrive at
\[
\iint W_0(\la x,y \rg) \, d\mu_x \, d\mu_y \geq c_N \cH^{N-2}(\Sigma)^2, \qquad c_N = \int_0^1 (1-\rho^2)^{-1} \bigl( 1 - \rho^{\frac{N-2}{2}} \bigr)^2 \, d \rho.
\]
as desired.
This proves our second claim.

\smallskip \noindent \textbf{Step 3: $N=3$}.
We now estimate the terms $\lambda_k$ more carefully.
We evaluate the terms of~\eqref{eqn:tau-k} as
\[
\lambda_1 = \pi - 4 \log 2 - \tfrac{1}{9}, \qquad \lambda_2 = \tfrac{448}{25} - 3 \pi - 12 \log 2 .
\]
For the remaining degrees, we write $\rho = y^2$ and use the bound
\begin{align*}
&\frac{1-y}{(1+y)(1+y^2)} \leq \frac{- \log y}{4y}, \qquad \text{due to } \; - \log y \geq \frac{2(1-y)}{1+y} \quad \text{and} \quad 1+y^2 \geq 2y, \\
& \implies \int_0^1 \rho^{\ell} \frac{(1 - \sqrt{\rho})^2}{1-\rho^2} \, d \rho = 2 \int_0^1 \frac{y^{2 \ell+1}(1-y)}{(1+y)(1+y^2)} \, dy \leq \frac{1}{2} \int_0^1 y^{2 \ell}(- \log y) = \frac{1}{2(2 \ell+1)^2}.
\end{align*}
For $k \geq 3$, we may therefore bound
\[
\lambda_k \leq \frac{k(k+1)}{2 (2k+1)} \Bigl( \frac{k+1}{(2k-1)^2} + \frac{k}{(2k+3)^2} \Bigr) \leq \frac{38}{225} < \lambda_2 < \lambda_1.
\]
An orthonormal basis of the degree-$1$ spherical harmonics on $\bS^2$ is $Y_i(x) = \sqrt{\frac{3}{4 \pi}} x_i$ for $i=1,2,3$, so
\[
\int_{\bS^2} x \, dS = 0, \qquad \implies \qquad c_{1i} = \sqrt{\frac{3}{4 \pi}} \int_{\bS^2} u(x) x_i \, dS = \sqrt{\frac{3}{4 \pi}} \int_{\Omega} x_i \, dS
\]
For $a \in \bR^3$, form the tangential vector field $X_a(x) = a - \la a,x \rg$.
Since $\textup{div}_{\bS^2} X_a = - 2 \la a,x \rg$ and $\la X_a, \nu \rg = \la a, \nu \rg$ on $\Sigma$, Green's theorem~\cite{maggi}*{Proposition~19.22 and Lemma~22.11} implies
\[
\Bigl \la a, \int_{\Sigma} \nu \, d \mu \Bigr\rg = \int_{\Sigma} \la X_a , \nu \rg \, d \mu = \int_{\Omega} \textup{div}_{\bS^2} X_a \, dS = - 2 \, \Bigl \la a, \int_{\Omega} x \, dS \Bigr \rg.
\]
Therefore, $\int_{\Omega} x \, dS = - \frac{1}{2} \int_{\Sigma} \nu \, d \mu$, and since $|\Omega| = \frac{1}{2} |\bS^2| = 2 \pi$, we take $a = \frac{\int_{\Omega} x \, dS}{|\int_{\Omega} x \, d S |}$ to bound
\begin{align*}
& \left| \int_{\Sigma} \nu \, d\mu \right| = 2 \, \left| \int_{\Omega} x \, dS \right| \leq 2 \int_{ \{ \la a, x \rg > 0 \} } \la a, x \rg \, dS = 2 \pi, \quad \implies \quad \sum_{\alpha} |c_{1 \alpha}|^2 = \frac{3}{4 \pi} \left| \int_{\Omega} x \, dS \right|^2 \leq \frac{3 \pi}{4}.
\end{align*}
Combining this with the property $\sum_{k,\alpha} |c_{k \alpha}|^2 = \pi$ and $\lambda_k < \lambda_2 < \lambda_1$, we conclude that
\begin{align*}
\sum_{k \geq 1, \alpha} \lambda_k |c_{k \alpha}|^2 \leq \lambda_2 \Bigl( \pi - \sum_{\alpha} |c_{1\alpha}|^2 \Bigr) + \lambda_1 \sum_{\alpha} |c_{1\alpha}|^2 \leq \tfrac{\pi}{4} \lambda_2 + \tfrac{3 \pi}{4} \lambda_1.
\end{align*}
Notice that $3 \lambda_1 + \lambda_2 = \frac{1319}{75} - 24 \log 2 < \frac{3}{\pi}$ and $|\bS^2| = 4 \pi$.
Using~\eqref{eqn:dominated-convergence-Omega}, we therefore obtain
\begin{align*}
\iint_{\Sigma \times \Sigma} W_0 ( \la x,y \rg) \, \la \nu_x, \nu_y \rg \, d\mu_x \, d\mu_y &\leq \pi^2 (3 \lambda_1 + \lambda_2) < 3 \pi.
\end{align*}
Since $3 \pi = \frac{3}{16 \pi} |\bS^2|^2$, this completes the proof of our result for $N=3$.
\end{proof}

Finally, we rule out stable $s=0$ links in dimension $N=3$.
Recall the following result.
\begin{lemma}\label{lemma:batcave}
    Let $h \geq 0$ be a measurable function on $(\Sigma, \mu)$, where $\mu$ is a non-atomic Radon measure.
    Then, for every function $f$ satisfying $0 \leq f \leq 1$ and $\int_{\Sigma} f \, d \mu = m$, it holds that
    \[
    \int_{\Sigma} hf \, d\mu \leq \sup_{ \substack{ F \subset \Sigma \\ \mu(F) = m } } \int_F h \, d\mu.
    \]
\end{lemma}
\begin{proof}
Let $t_* := \inf \{ s \geq 0 : \mu ( \{ h > s \}) \leq m \}$.
By the continuity of the measure along monotone level sets, this $t_*$ satisfies
\[
\mu ( \{ h > t_* \}) \leq m \leq \mu ( \{ h \geq t_* \}), \qquad \implies \qquad 0 \leq m - \mu ( \{ h > t_* \}) \leq \mu ( \{ h = t_* \}).
\]
Since $\mu$ is non-atomic, we can find a measurable set $E \subset \{ h = t_* \}$ with $\mu(E) = m - \mu ( \{ h > t_* \})$, so the set $F := \{ h> t_* \} \cup E$ has measure $m$ and $\{ h > t_* \} \subset F \subset \{ h \geq t_* \}$.
By construction, $(h-t)( f - \mathbf{1}_F) \leq 0$ pointwise and $\int_{\Sigma} (f - \mathbf{1}_F) = 0$.
Therefore,
\[
\int_{\Sigma} hf \, d \mu - \int_F h \, d\mu = \int_{\Sigma} (h-t) (f -\mathbf{1}_F) \, d\mu \leq 0
\]
so the left-hand side is bounded by the supremum of $\int_F h \, d \mu$ over sets $F$ with $\mu(F) = m$.
\end{proof}

The computations of Lemma~\ref{lemma:domain-expansion} and Proposition~\ref{prop:stable-links-s=0} imply a uniform bound on the volume of limiting stable links: for every $N \geq 4$, there exist $a_N, b_N$ such that every Caccioppoli set $\Omega \subset \bS^{N-1}$ arising as the spherical limit of stable $s$-minimal cones in $\bR^N$ as $s \downarrow 0$, with $\Sigma = \partial^* \Omega$, satisfies
\[
\cH^{N-2}(\Sigma)^2 \leq a_N \Bigl( \int_{\Sigma} \nu \, d \mu_{\Sigma} \Bigr)^2 + b_N.
\]

\begin{proposition}\label{prop:stable-links-s=0}
For $N=3$, the only possible Caccioppoli set $\Omega \subset \bS^2$ arising as the spherical limit of stable $s$-minimal cones in $\bR^3$ is a hemisphere up to a null set.
\end{proposition}
\begin{proof}
For brevity, let $\mu := |D \mathbf{1}_{\Omega}|$ and $\Sigma = \partial^* \Omega$, so $\cH^1(\Sigma) = |D \mathbf{1}_{\Omega}|(\bS^2)$.
We also write $L := \cH^1(\Sigma)$ and proceed in a number of steps that combine the bounds of Corollary~\ref{cor:limit-Jacobi-inequality} and Lemma~\ref{lemma:domain-expansion}.

\smallskip \noindent \textbf{Step 1: Bounds on the kernel $W_0$.}
The limiting kernel $W_0$ is given by
\[
W_0(z) = \int_0^1 (1 - \sqrt{\rho})^2 (1 + \rho^2 - 2 \rho z)^{- \frac{3}{2}} \, d \rho.
\]
First, we observe the inequalities
\begin{equation}\label{eqn:W0(z)-N=3}
0 \leq W_0(z) \leq \frac{1}{4} \textup{arcsinh} \Bigl( \frac{\sqrt{2}}{\sqrt{1-z}} \Bigr), \qquad 0 \leq W_0( \la x,y \rg) \leq \frac{1}{4} \textup{arcsinh} \, \Bigl( \frac{2}{|x-y|} \Bigr)
\end{equation}
Indeed, the change of variables $u = \frac{1 - \rho}{\sqrt{\rho}}$ shows that
\begin{align*}
W_0(z) &= \int_0^{\infty} \frac{u^2}{2 \sqrt{1 + \frac{u^2}{4}} \bigl( 1 + \sqrt{1 + \frac{u^2}{4}} \bigr) (u^2 + 2(1-z))^{\frac{3}{2}} } \, du \\
&\leq 2 \int_0^{\infty} \frac{u^2}{(u^2+8) (u^2 + 2(1-z))^{\frac{3}{2}}} \, du, \qquad \text{due to } \; \frac{1}{2 \sqrt{1 + \frac{u^2}{4}} \bigl( 1 + \sqrt{1 + \frac{u^2}{4}} \bigr)} \leq \frac{2}{u^2+8}.
\end{align*}
Let $q := \sqrt{1 - \frac{1}{4}(1-z)} \in [ \frac{1}{\sqrt{2}}, 1)$, so writing $u = \sqrt{2(1-z)} \, \sinh t$ and $v = \tanh t$ produces
\[
\frac{1}{4} \int_0^1 \frac{v^2}{1-q^2v^2} \, dv = \frac{1}{4} \frac{\textup{arctanh} \, q - q}{q^3}, \qquad \text{where} \quad q \in \bigl[ \tfrac{1}{\sqrt{2}}, 1 \bigr).
\]
We claim that the function $H(q) := \textup{arcsinh} \Bigl( \frac{1}{\sqrt{2(1-q^2)}} \Bigr) - \frac{\textup{arctanh} \, q - q}{q^3}$ is non-negative.
Indeed,
\begin{align*}
    q H'(q) &= \frac{q^2}{(1-q^2) \sqrt{3-2q^2}} - \frac{1}{1-q^2} + \frac{3 ( \textup{arctanh} \, q - q)}{q^3} \\
    &\geq 1 + \frac{3q^2}{5} + \frac{3 q^4}{7} - \frac{q^3+3}{\sqrt{3-2q^2} \, ( \sqrt{3-2q^2}+q^2)} \geq 1 + \frac{3q^2}{5} + \frac{3q^4}{7} - \frac{2}{2-q^2}
\end{align*}
because $\textup{arctanh} \, q - q \geq \frac{q^3}{3} + \frac{q^5}{5} + \frac{q^7}{7}$ and $\sqrt{3-2x} \geq \frac{1}{2}(3-x)$ for $x = q^2 \in [ \frac{1}{2}, 1]$.
Finally, we collect 
\[
1 + \frac{3x}{5} + \frac{3x^2}{7} - \frac{2}{2-x} = \frac{x(7+9x - 15x^2)}{35(2-x)} \geq 0 \qquad \text{for } \; x \in \bigl[ \tfrac{1}{2}, 1 \bigr].
\]
Therefore, $H'(q) \geq 0$, and $H( \frac{1}{\sqrt{2}}) = 2 - (2 \sqrt{2} - 1) \, \textup{arcsinh}(1) > 3 - 2 \sqrt{2} > 0$ shows that $H>0$ for $q \in [ \frac{1}{\sqrt{2}}, 1]$.
Recalling that $1-q^2 = \frac{1-z}{4}$, where $z = \la x,y \rg$ has $|x-y|^2 = 2(1-z)$, we obtain
\[
W_0(z) \leq \frac{1}{4} \frac{\textup{arctanh}(q) - q}{q^3} \leq \frac{1}{4} \textup{arcsinh} \Bigl( \frac{1}{\sqrt{2(1-q^2)}} \Bigr) = \frac{1}{4} \textup{arcsinh} \Bigl( \frac{\sqrt{2}}{\sqrt{1-z}} \Bigr).
\]

\noindent \textbf{Step 2: Bounds on length and flux.}
As in the argument of Corollary~\ref{cor:limit-Jacobi-inequality}, we recall that the limiting measures in every dimension $N$ satisfy
\begin{equation}\label{eqn:first-bound-B0}
\iint_{\Sigma \times \Sigma} B_0(\la x,y \rg) \, (1 - \la \nu_x, \nu_y \rg) \, d\mu_x \, d \mu_y \leq \ell( \bS^{N-1}) \leq \iint_{\Sigma \times \Sigma} W_0(\la x,y \rg) \, d\mu_x \, d\mu_y
\end{equation}
as a consequence of the limiting stability inequality for the limit link $\Sigma = \partial^* \Omega$.
The bounds on $A_0, B_0$ from Lemma~\ref{lemma:A-B-properties}, imply that $A_0(z) \geq \frac{\Gamma(\frac{N}{2})^2}{\Gamma(N-1)} B_0(z)$ and $B_0(z) \geq B_0(-1) = \frac{1}{N-1}$, so $A_0(z) \geq \frac{\Gamma( \frac{N}{2})^2}{\Gamma(N)}$.
For $N=3$, this makes $A_0 (z) \geq \frac{\pi}{4} B_0(z)$, hence
\[
\iint_{\Sigma \times \Sigma} A_0(\la x,y \rg) \, (1 - \la \nu_x, \nu_y \rg) \, d \mu_x \, d\mu_y \geq \frac{\pi}{4} \iint_{\Sigma \times \Sigma} B_0(\la x,y \rg) \, (1 - \la \nu_x, \nu_y \rg) \, d\mu_x \, d\mu_y.
\]
Next, we decompose $B_0 = A_0 + W_0$, whereby
\allowdisplaybreaks{
\begin{align*}
    & \iint_{\Sigma \times \Sigma} W_0(\la x,y \rg) \, d\mu_x \, d\mu_y \\
    &= \iint_{\Sigma \times \Sigma} W_0(\la x,y\rg) \, \la \nu_x, \nu_y \rg) \, d\mu_x \, d\mu_y +  \iint_{\Sigma \times \Sigma} (B_0 - A_0)(\la x,y \rg) \, (1 - \la \nu_x, \nu_y \rg ) \, d \mu_x \, d\mu_y \\
    &\leq \iint_{\Sigma \times \Sigma} W_0(\la x,y\rg) \, \la \nu_x, \nu_y \rg) \, d\mu_x \, d\mu_y + \Bigl( 1 - \frac{\pi}{4} \Bigr) \iint_{\Sigma \times \Sigma} B_0( \la x,y \rg) \, (1 - \la \nu_x, \nu_y \rg) \, d \mu_x \, d \mu_y .
\end{align*}}
Combined with the inequality~\eqref{eqn:first-bound-B0} and the bound of Lemma~\ref{lemma:domain-expansion}, this shows that
\[
\iint_{\Sigma \times \Sigma} W_0(\la x,y\rg) \, d\mu_x \, d\mu_y \leq \frac{4}{\pi} \iint_{\Sigma \times \Sigma} W_0(\la x,y \rg) \, \la \nu_x, \nu_y \rg \, d\mu_x \, d\mu_y \leq \frac{4}{\pi} \cdot 3 \pi = 12.
\]
On the other hand, the second bound of Lemma~\ref{lemma:domain-expansion} implies that
\[
c_3 L^2 \leq \iint_{\Sigma \times \Sigma} W_0(\la x,y \rg) \, d\mu_x \, d\mu_y, \qquad c_3 := \int_0^1 \frac{(1 - \sqrt{\rho})^2}{1 - \rho^2} \, d \rho  = \frac{1}{2} ( \pi - 4 \log 2).
\]
Therefore, $L^2 \leq \frac{24}{\pi - 4 \log 2} < 66$, so $L < \frac{41}{5}$.
Since $|\Omega| = \frac{1}{2} |\bS^{N-1}| = 2 \pi$, the spherical isoperimetric inequality of~\cite{duzaar-fusco} shows that $L = \cH^1( \partial^* \Omega) \geq 2 \pi$.
In addition, since $B_0(\la x,y \rg) \geq B_0(-1) = \frac{1}{N-1}$ by Lemma~\ref{lemma:A-B-properties}, we have $B_0 \geq \frac{1}{2}$ for $N=3$.
Then, for $V := \int_{\Sigma} \nu \, d\mu$ the flux vector, we find
\begin{equation}\label{eqn:B0-lower-bound-n-1}
 \iint B_0( \la x,y \rg) \, (1 - \la \nu_x, \nu_y \rg) \, d \mu_x \, d \mu_y \geq \frac{1}{2} \iint (1 - \la \nu_x, \nu_y \rg) \, d\mu_x \, d\mu_y = \frac{L^2 - |V|^2}{2}.   
\end{equation}
This quantity is also bounded from above by $12$, due to~\eqref{eqn:first-bound-B0}.
Hence, $L^2 - |V|^2 \leq 24$.
Since $L \geq 2 \pi$, this yields $|V|^2 \geq ( 2 \pi)^2 - 24 > 15 > ( \frac{18}{5} )^2$.
In particular, $|V| \neq 0$.
Next, we prove
\begin{equation}\label{eqn:14-r}
    \mu(\Sigma \cap B_r(p)) \leq 12 r \quad \text{for every } \; p \in \bS^2, \quad 0 < r \leq 2. 
\end{equation}
For $r \geq \frac{7}{10}$, this property is clear, since $\mu(\Sigma \cap B_r(p)) \leq \mu(\Sigma) = \cH^1(\Sigma) < \frac{41}{5} < 12 \cdot\frac{7}{10}$.
For $r \in (0, \frac{7}{10}]$, we recall that $B_0(z) = \frac{1}{1-z}$, so $\la x,y \rg = 1 - \frac{|x-y|^2}{2}$ gives $B_0(\la x,y \rg) = \frac{2}{|x-y|^2}$ and
\[
x,y \in \Sigma \cap B_r(p) \implies |x-y| \leq 2r, \quad \implies \quad B_0(\la x,y \rg) \geq \tfrac{1}{2} r^{-2} \quad \text{on } \; (\Sigma \cap B_r(p)) \times (\Sigma \cap B_r(p)).
\]
Then, the left-hand side of~\eqref{eqn:first-bound-B0} is bounded from below by
\begin{align*}
\iint_{\Sigma \times \Sigma} B_0(\la x,y \rg) \, (1 -\la \nu_x, \nu_y \rg) \, d \mu_x \, d\mu_y &\geq \frac{1}{2r^2} \iint_{(\Sigma \cap B_r(p)) \times (\Sigma \cap B_r(p))} (1 - \la \nu_x, \nu_y \rg) \, d\mu_x \, d \mu_y \\
&= \frac{1}{2r^2} \left( \mu(\Sigma \cap B_r(p))^2 - \left| \int_{\Sigma \cap B_r(p)} \nu \, d \mu \right|^2 \right).
\end{align*}
On the other hand, applying Green's theorem~\cite{maggi}*{Proposition~19.22 and Lemma~22.11} to $\Omega \cap B_r(z)$ with the tangential vector field $X_a := a - \la a,x \rg x$ as above shows that
\[
\left| \int_{\Sigma \cap B_r(p)} \nu \, d \mu \right| \leq 2 \pi r + 2 \pi r^2
\]
for a.e. $r \in (0,1]$.
Indeed, the new boundary lying on $\partial B_r(z)$ has length at most $2 \pi r$, while the divergence term over the spherical cap is at most $2 \pi r^2$, due to $|X_a| \leq 1$ and $|\textup{div}_{\bS^2} X_a| \leq 2$.
The estimate for every $r$ follows upon approximating by above.

Combining these computations in the inequality~\eqref{eqn:first-bound-B0} and recalling the above upper bound $\iint W_0(\la x,y \rg) \, d\mu_x \, d\mu_y \leq 12$, we deduce that, for every $r \leq \frac{7}{10}$,
\[
\mu(\Sigma \cap B_r(p))^2 \leq \left| \int_{\Sigma \cap B_r(p)} \nu \, d \mu \right|^2 + 24 r^2 \leq (2r)^2 \, [ \pi^2(1+r)^2 + 6 ] \leq (2r)^2 \, \bigl[ ( \tfrac{17}{10})^2\pi^2 + 6 \bigr] \leq (12 r)^2.
\]
Here, we used $(\frac{17}{10})^2 \pi^2 + 6 < 35 < 6^2$.
This proves the claimed inequality~\eqref{eqn:14-r}.

\smallskip \noindent \textbf{Step 3: An integral estimate.}
We claim that for every $x \in \Sigma$ and every measurable set $F \subset \Sigma$ with $\mu(F) = t$, it holds that
\begin{equation}\label{eqn:mu-bound-W0}
    \sup_{x \in \Sigma} \int_F W_0(\la x,y \rg) \, d \mu_y \leq \tfrac{1}{4} \, t \, [ 1 + \log (1 + 48 t^{-1}) \, ].
\end{equation}
Indeed, let us fix $x \in \Sigma$ and write $m(r) := \mu(F \cap B_r(x)) \leq \min \{ t, 12 r\}$.
Recall that $W_0( \la x,y \rg) \leq \frac{1}{4} \textup{arcsinh} \bigl( \frac{2}{|x-y|} \bigr)$ because of the bound~\eqref{eqn:W0(z)-N=3}.
Because the function $r \mapsto \textup{arcsinh} \, \frac{2}{r}$ is positive and strictly decreasing, arguing as in Lemma~\ref{lemma:batcave} shows that the integral $\int_F W_0(\la x,y \rg) \, d \mu_y$ is maximized by the radial distribution $d m_*(r) = 12 \, \mathbf{1}_{(0,t/12)}(r) \, dr$.
We therefore obtain
\begin{align*}
\int_F W_0(\la x,y \rg) \, d \mu_y &\leq \int_F \frac{1}{4} \textup{arcsinh} \Bigl( \frac{2}{|x-y|} \Bigr) \, d \mu_y \leq 3 \int_0^{\frac{t}{12}}\textup{arcsinh}  \, \frac{2}{r} \, dr \\
&= \frac{t}{4} \textup{arcsinh} \, \frac{24}{t} + 6 \, \textup{arcsinh} \, \frac{t}{24}.
\end{align*}
Writing $r = \frac{t}{24}$, for $r \geq 0$ we observe the inequalities $r \, \textup{arcsinh} \, \tfrac{1}{r} \leq r \log (1 + 2r^{-1})$ and $\textup{arcsinh} \, r \leq r$.
The inequality~\eqref{eqn:mu-bound-W0} follows from combining these steps.

\smallskip \noindent \textbf{Step 4: Conclusion of the proof.}
Since $V := \int_{\Sigma} \nu \, d \mu$ has $|V| \neq 0$, by Step 2, we can form the unit vector $a = \frac{V}{|V|}$ and write $g(x) := 1 - \la a, \nu_x \rg \in [0,2]$.
Then, $\int_{\Sigma} g \, d \mu = \cH^1(\Sigma) - |V| = L \delta$ where $\delta := 1 - \frac{|V|}{L}$.
Then, we can apply Lemma~\ref{lemma:batcave} for the function $f = \frac{1}{2} g$ to obtain, for every $x \in \Sigma$,
\[
\int_{\Sigma} W_0(\la x,y \rg) \, g(y) \, d\mu_y \leq 2 \sup_{ \mu(F) = \frac{1}{2} L \delta } \int_F W_0(\la x,y \rg) \, d \mu_y \leq \tfrac{1}{4} L \delta \bigl[ 1 + \log (1 + 96 L^{-1} \delta^{-1}) \bigr].
\]
In the second step, we used the bound~\eqref{eqn:mu-bound-W0}.
Therefore, using the bound of Corollary~\ref{cor:limit-Jacobi-inequality}, we have
\begin{align*}
    \iint_{\Sigma \times \Sigma} B_0(\la x,y \rg) \, (1 - \la \nu_x, \nu_y \rg) \, d\mu_x \, d \mu_y &\leq \iint_{\Sigma \times\Sigma} W_0(\la x,y \rg) \, g(x) \, g(y) \, d\mu_x \, d\mu_y \\
    &\leq \tfrac{1}{4} ( L \delta)^2 [ 1 + \log (1 + 96 L^{-1} \delta^{-1}) ].
\end{align*}
The inequality $L^2 - |V|^2 \leq 24$ becomes $L^2 \delta(2-\delta) \leq 24$, and using $L \geq 2 \pi$, we obtain $\delta(2-\delta) \leq \frac{6}{\pi^2} < \frac{39}{64}$.
By the monotonicity of $\delta \mapsto \delta(2-\delta)$ for $\delta \in [0,1]$, this forces $\delta < \frac{3}{8}$.
Now, if $\delta>0$,
\begin{align*}
& \tfrac{1}{2} L^2 \delta(2-\delta) \leq \iint B_0(\la x,y \rg) \, (1 - \la \nu_x, \nu_y \rg) \, d \mu_x \, d \mu_y \leq \tfrac{1}{4} ( L \delta)^2 [ 1 + \log (1 + 96 L^{-1} \delta^{-1})] \\
& \implies 2(2-\delta) \leq \delta \, [ 1 + \log (1 + 96 L^{-1} \delta^{-1}) ] \leq \delta \, [ 1 + \log (1 + 16 \delta^{-1})]
\end{align*}
due to $L \geq 2 \pi > 6$.
The left-hand side is strictly decreasing, while the right-hand side is strictly increasing for $\delta \in [0,1]$.
However, for $\delta = \frac{3}{8}$, we have
\[
2 (2 - \tfrac{3}{8}) = \tfrac{13}{4} > 3 > 2> \tfrac{3}{8} [ 1 + 2 \log 7] > \tfrac{3}{8} [ 1 + \log (1+ 16 \cdot \tfrac{8}{3}) ],
\]
hence the above inequality fails for all $\delta \in (0, \frac{3}{8}]$.
Finally, if $\delta=0$, then $\int_{\Sigma} g \, d \mu = 0$ and $g = 1 - \la a, \nu_x \rg \geq 0$ implies that $\la a, \nu_x \rg = 1$, so $\nu_x = a$ for $\mu$-a.e. $x$.
Because $\nu_x \in T_x \bS^2$, this implies $\la x,a\rg =0$ for $\mu$-a.e. $x$, so $\partial^* \Omega$ is contained, up to an $\cH^1$-null set, in the equator $\bS^2 \cap a^{\perp}$. 
Also, $D \mathbf{1}_{\Omega}=0$ in each hemisphere and $|\Omega|=2 \pi$, so $\mathbf{1}_{\Omega}$ is constant a.e., and the rigidity of the spherical isoperimetric inequality~\cite{duzaar-fusco} implies that $\Omega$ is a hemisphere up to a null set.
\end{proof}

We record a standard consequence of dimension reduction for $s$-perimeter minimizing cones~\cite{caffarelli-roquejoffre-savin}*{Theorem 10.3}.
Using the improvement of flatness~\cite{caffarelli-roquejoffre-savin}*{Theorem 6.1}, whereby regular points are $C^{1,\alpha}$, and the bootstrap arguments of~\cites{barrios-figalli-valdinoci , figalli-valdinoci }, this implies the following reduction.
\begin{lemma}\label{lemma:first-singular-dimension}
Given $s \in (0,1)$, if locally $s$-perimeter minimizing cones in $\bR^n$ are flat, then locally $s$-perimeter minimizing cones in $\bR^{n+1}$ is flat or have an isolated singularity and $C^{\infty}$ link.  
\end{lemma}
\iffalse
\begin{proof}
    The result is immediate if all $s$-minimizing cones in $\bR^{n+1}$ are flat, so suppose that $(n+1)$ is the first singular dimension.
    By the dimension reduction property of $s$-perimeter minimizing cones, every non-flat cone that is locally minimizing for the $s$-perimeter in the first singular dimension has an isolated singularity and regular boundary away from the vertex.
    Improvement of flatness~\cite{caffarelli-roquejoffre-savin}*{Theorem 6.1} implies that regular points are $C^{1,\alpha}$, and the bootstrap arguments of~\cites{barrios-figalli-valdinoci , figalli-valdinoci } shows that the boundary is smooth.
\end{proof}
\fi

\begin{proof}[Proof of Theorem~\ref{thm:main}]
For stable cones with $C^2$ links, Proposition~\ref{prop:stable-links-s=0} shows that the only $s=0$ limit links of stable $s$-minimal cones in $\bR^3$ are hemispheres.
Thus, Proposition~\ref{prop:dimension-descent} implies that stable cones with $C^2$ links in $\bR^3$ are flat for $s \in (0,s_*)$.
For locally $s$-minimizing cones, the works~\cites{savin-valdinoci , savin-valdinoci-2 , cinti-serra-valdinoci} imply flatness in $\bR^2$.
Therefore, Lemma~\ref{lemma:first-singular-dimension} shows that any singular minimizers in $\bR^3$ would have an isolated singularity and a smooth link.
However, for $s \in (0,s_*)$, the first part of our argument shows that such cones must be flat.
This completes the proof.
\end{proof}

\appendix

\section{Estimates for integral kernels}\label{app:kernel-estimates}

We collect here some auxiliary estimates for the integral kernels $A_s, B_s, W_s$ introduced in~\eqref{eqn:spherical-kernels}, in Section~\ref{sec:preliminaries}.
These estimates are important in the proof of Theorem~\ref{thm:main}, notably for our compactness Theorem~\ref{thm:compactness-thm}.
We will repeatedly use the following simple consequence of an upper density bound.

\begin{lemma}\label{lemma:potential-upper-growth}
We consider some $0 < \beta_0 < m$ and let $\mu$ be a finite Radon measure on $\bS^{N-1}$ satisfying $\mu(B_r(x)) \leq \Lambda r^m$ for every $x\in\bS^{N-1}$ and $r\in(0,2]$. If $0\leq\alpha \leq m - \beta_0$ and $q\geq0$, then
\begin{equation}\label{eqn:potential-upper-growth}
\sup_{x\in\bS^{N-1}}\int_{B_R(x)} \frac{(1+|\log|x-y||)^q}{|x-y|^\alpha} \, d\mu(y) \leq C(m,\beta_0,q) \, \Lambda \, R^{m-\alpha}(1+|\log R|)^q
\end{equation}
for every $R\in(0,1]$. 
Moreover, if $0 \leq h \leq 1$ is measurable and $t := \int h \, d\mu \leq \Lambda$, then
\begin{equation}\label{eqn:potential-small-set}
\sup_{x\in\bS^{N-1}} \int_{\bS^{N-1}} \frac{(1+|\log|x-y||)^q}{|x-y|^\alpha} \, h(y) \, d\mu(y) \leq C(m,\beta_0,q)\Lambda^{\frac{\alpha}{m}}t^{1- \frac{\alpha}{m}}
\bigl(1+\log(\Lambda/t)\bigr)^q.
\end{equation}
In particular, this holds for a measurable set $F \subset \bS^{N-1}$ with $h=\mathbf{1}_F$ and $t=\mu(F) \leq \Lambda$.

Additionally, if $K(x,y)$ is a measurable kernel with $0 \leq K(x,y) \leq K_0 \frac{(1 + |\log |x-y||)^q}{|x-y|^{\alpha}}$, then
\begin{equation}\label{eqn:K-double-bound}
    \begin{split}
\sup_x \int_{B_R(x)} K(x,y) \, d\mu(y) &\leq C(m,\beta_0,q) K_0 \Lambda R^{m-\alpha} (1 + |\log R|)^q, \\
\sup_x \int K(x,y) h(y) \, d\mu(y) &\leq C(m,\beta_0,q) K_0 \Lambda^{\frac{\alpha}{m}} t^{1 -\frac{\alpha}{m}} (1 + \log (\Lambda/t) )^q.
    \end{split}
\end{equation}
\end{lemma}
\begin{proof}
    We fix a point $x$ and argue as in~\cite{maggi}*{Theorem 15.5}: let $M(r) := \mu(B_r(x)) \leq \Lambda r^m$, so $\mu( \{ x\} ) = M(0) =  0$.
    For $r \in (0,1]$, the function $f(r) := r^{- \alpha} (1 + |\log r|)^q$ is decreasing and
    \[
    - f'(r) = r^{-\alpha-1} [ \alpha(1-\log r)^q + q(1 - \log r)^{q-1}] \leq (m+q) r^{-\alpha-1}( 1 + |\log r|)^q.
    \]
    Also, $f(r) M(r) \leq \Lambda r^{m-\alpha} ( 1 + |\log r|)^q \to 0$ as $r \downarrow 0$ for $m-\alpha>0$.
    Thus, integration by parts gives
    \[
    \int_{B_R(x)} f( |x-y|) \, d\mu(y) \leq \Lambda R^{m-\alpha} ( 1 + |\log R|)^q + C \Lambda \int_0^R r^{m-\alpha-1} (1 + |\log r|)^q \, dr.
    \]
    Since $m-\alpha \geq \beta_0$, we can set $r = R e^{-u}$ to bound the last integral by
    \begin{align*}
        \int_0^R r^{m-\alpha-1} (1 + |\log r|)^q \, dr &= R^{m-\alpha} \int_0^{\infty} e^{- (m-\alpha) u} ( 1 + |\log R| + u)^q \, du \\
        &\leq R^{m-\alpha} (1 + |\log R|)^q \int_0^{\infty} e^{- \beta_0 u}(1+u)^q \, du \\
        &\leq C(\beta_0,q) R^{m-\alpha}(1 + |\log R|)^q.
    \end{align*}
    This proves~\eqref{eqn:potential-upper-growth}.
    For~\eqref{eqn:potential-small-set}, the assertion is immediate if $t=0$. Otherwise, let $R=(\frac{t}{\Lambda})^{\frac{1}{m}}\leq1$, so
    \[
    f(r) \leq C(\beta_0,q) R^{-\alpha} (1 + |\log R|)^q \qquad \text{for every } \; r \in [R,2],
    \]
    since $R \leq 1$.
    Indeed, for $r \in [R,1]$, this follows from the monotonicity $f(r) \leq f(R)$, while for $r \in [1,2]$, we have $f(r) \leq (1+ \log 2)^q \leq C(q) R^{-\alpha}(1 + |\log R|)^q$.
Hence, using $0\leq h\leq1$ and~\eqref{eqn:potential-upper-growth},
\begin{align*}
\int f(|x-y|)h(y)\,d\mu(y) &\leq \int_{B_R(x)}f(|x-y|)\,d\mu(y) +C tR^{-\alpha}(1+|\log R|)^q\\
&\leq C \Lambda R^{m-\alpha} (1 + |\log R|)^q + C t R^{-\alpha} (1 + |\log R|)^q.
\end{align*}
Finally, we can substitute $\Lambda R^m = t$ to bound both terms by $C(m,\beta_0,q) \Lambda^{\frac{\alpha}{m}} t^{1 - \frac{\alpha}{m}} ( 1 + |\log R|)^q$.
The kernel bound~\eqref{eqn:K-double-bound} follows from these considerations.
\end{proof}

\begin{corollary}\label{cor:local-integrability}
    Let $\Sigma \subset \bS^{N-1}$ be a compact $C^1$ $m$-dimensional submanifold and let $\mu := \cH^m \mres \Sigma$.
    If $0 \leq \alpha < m$ and $q \geq 0$, then every kernel $K$ has
    \[
    |K(x,y)| \leq C \, |x-y|^{- \alpha} ( 1 + |\log |x-y||)^q, \qquad \implies \qquad K \in L^1(\Sigma \times \Sigma).
    \]
    In particular, we have $\sup_{x \in \Sigma} \int_{\Sigma} |x-y|^{-\alpha} (1+ |\log |x-y||)^q \, d\mu(y) < \infty$.
\end{corollary}
\begin{proof}
    Since $\Sigma$ is compact and $C^1$, there is a constant $C_{\Sigma}$ such that $\mu(B_r(x)) \leq C_{\Sigma} r^m$ uniformly in $x \in \Sigma$ and $r \in (0,2]$.
    The conclusion follows from Lemma~\ref{lemma:potential-upper-growth}.
\end{proof}

\begin{lemma}\label{lemma:radon-Sm}
Let $M$ be a compact Riemannian $m$-manifold with Riemannian distance $d(x,y)$.
Suppose that a non-negative Radon measure on $M$ on $\lambda$ satisfies
\[
\iint_{M \times M} d(x,y)^{-m} \, d \lambda(x) \, d \lambda(y) < \infty.
\]
Then, $\lambda=0$ as a measure. 
\end{lemma}
\begin{proof}
First, there exists a small $\rho_0>0$ and constants $c_M, C_M> 0$ such that $c_M r^m \leq \textup{Vol}(B_r(x)) \leq C_M r^m$ for every $x \in M$ and $\rho \in (0,\rho_0)$.
It is standard that for every $r \in (0,\rho_0)$, there exists a finite Borel partition $M = \bigsqcup_{i=1}^{N_r} E_i$ with $N_r \leq C_M r^{-m}$ and $\textup{diam}(E_i) \leq 2r$.
For example, since $M$ is compact, hence totally bounded, it admits a maximal $r$-separated set $S_r = \{ x_1, \dots, x_{N_r} \} \subset M$, so the balls $B_{\frac{r}{2}}(x_i)$ are pairwise disjoint.
By~\cite{burago-ivanov}*{Exercise 1.6.4(2)}, $S_r$ is an $r$-net, hence $M = \bigcup_{i=1}^{N_r} B_r(x_i)$
and $N_r c_M r^m \leq \sum_i \textup{Vol}(B_{\frac{r}{2}}(x_i)) \leq \textup{Vol}(M)$.
In particular, $N_r \leq C_M r^{-m}$ and the sets $E_k := B_r(x_k) \setminus \bigcup_{j<k} B_r(x_j)$ form a Borel partition with the required properties.
This property can also be obtained upon covering $M$ by coordinate charts and applying the Besicovitch covering theorem~\cite{maggi}*{Theorem 5.1 and \S 5.1}.
We now compute
\begin{align*}
& F(\rho) := ( \lambda \otimes \lambda) ( \{ (x,y) \in M \times M : d(x,y) < \rho \} ) = \int_M \lambda (B_{\rho}(x)) \, d \lambda(x), \\
& \implies F(2r) \geq \sum_{j=1}^{N_r} \lambda(E_j)^2 \geq N_r^{-1} \Bigl( \sum_{j=1}^{N_r} \lambda(E_j) \Bigr)^2 \geq c_M \lambda(M)^2 r^m
\end{align*}
due to $\textup{diam}(E_j) \leq 2r$, the Cauchy-Schwarz inequality, and $N_r \leq C_M r^{-m}$.
After changing constants, this gives $F(r) \geq c_M \lambda(M)^2 r^m$ for every small $r \in (0,\rho_1)$.
We let $D := \textup{diam}(M)$ and write
\[
s^{-m} = D^{-m} + m \int_0^D \mathbf{1}_{ \{ s<r \} } r^{-m-1} \, dr, \qquad \text{for } \; s \in (0,D].
\]
This identity also holds in the extended sense when $s=0$.
Hence, by Tonelli's theorem,
\begin{align*}
    \iint_{M \times M} d(x,y)^{-m} \, d \lambda(x) \, d \lambda(y) &= D^{-m} \lambda(M)^2 + m \int_0^D r^{-m-1} F(r) \, dr \geq m c_M \lambda(M)^2 \int_0^{\rho_1} \frac{dr}{r}
\end{align*}
upon combining the above identity with $r^{-m} F(r) \geq c_M \lambda(M)^2$ for $r<\rho_1$.
The last integral diverges unless $\lambda(M) = 0$.
Hence, the finiteness assumption forces $\lambda=0$.
This completes the proof.
\end{proof}

\bibliography{ref}

\end{document}